\documentclass{siamltex}

\usepackage{verbatim}
\usepackage{placeins}
\usepackage{fancyhdr}
\usepackage{graphicx}
\usepackage{latexsym}
\usepackage{mathrsfs}
\usepackage{amssymb,amsmath, amsfonts}

\usepackage{paralist}
\usepackage{hyperref}
\usepackage{color}

\usepackage{algorithmic}
\usepackage{algorithm}
\usepackage{euscript}
\usepackage{bbm}

\def\qH2opt{\mbox{\textsc{q{\scriptsize $\mathcal{H}_2$}opt-ph}}}   
\def\podph{\mbox{\textsc{pod-ph}}}  
\def\poddeimph{\mbox{\textsc{pod-deim-ph}}}  
\def\HtwoEpsph{\mbox{\textsc{{\scriptsize$\mathcal{H}_2^{\varepsilon}$}-ph}}}
\def\HtwoEpsdeimph{\mbox{\textsc{{\scriptsize$\mathcal{H}_2^{\varepsilon}$}-deim-ph}}}
\def\HtwoEps{\mbox{$\mathcal{H}_2^{\Large \varepsilon}$}}
\def\podHtwoEpsph{\mbox{\textsc{{\scriptsize pod-$\mathcal{H}_2^{\varepsilon}$}-ph}}}
\def\podHtwoEpsdeimph{\mbox{\textsc{{\scriptsize pod-$\mathcal{H}_2^{\varepsilon}$}-\textsc{deim}-ph}}}
\def\podHtwoEps{\mbox{\textsc{{\scriptsize pod-$\mathcal{H}_2^{\varepsilon}$}}}}

\newcommand{\bfrho}{{\boldsymbol{\rho}}}
\newcommand{\bfeta}{{\boldsymbol{\eta}}}
\newcommand{\bfzeta}{{\boldsymbol{\zeta}}}
\newcommand{\bfxi}{{\boldsymbol{\xi}}}
\newcommand{\bfeps}{{\boldsymbol{\varepsilon}}}

\newcommand{\cbfA}{\mbox{\boldmath${\EuScript{A}}$} }

\newcommand{\cbfG}{\mbox{\boldmath${\EuScript{G}}$}}

\newcommand{\cbfP}{\mbox{\boldmath${\EuScript{P}}$} }
\newcommand{\bhatx}{{\boldsymbol{\hat{\mathit{x}}}}}
\newcommand{\bhatxdot}{{\boldsymbol{\dot{\hat{x}}}}}
\newcommand{\bhaty}{{\boldsymbol{\hat{\mathit{y}}}}}

\newcommand{\real}{\mathbb{R}}

\newcommand {\cE} {{\cal E}}

\newcommand {\cL} {{\cal L}}

\newcommand {\cV} {{\cal V}}
\newcommand {\cW} {{\cal W}}

\newcommand{\be}{\mathbf{e}}
\newcommand{\bff}{\mathbf{f}}
\newcommand{\bfe}{\mathbf{e}}

\newcommand{\bp}{\mathbf{p}}
\newcommand{\bq}{\mathbf{q}}

\newcommand{\bu}{\mathbf{u}}
\newcommand{\bv}{\mathbf{v}}
\newcommand{\bw}{\mathbf{w}}
\newcommand{\bx}{\mathbf{x}}
\newcommand{\by}{\mathbf{y}}
\newcommand{\bz}{\mathbf{z}}
\newcommand{\bA}{\mathbf{A}}
\newcommand{\bB}{\mathbf{B}}

\newcommand{\bE}{\mathbf{E}}
\newcommand{\bF}{\mathbf{F}}
\newcommand{\bG}{\mathbf{G}}

\newcommand{\bI}{\mathbf{I}}
\newcommand{\bJ}{\mathbf{J}}

\newcommand{\bQ}{\mathbf{Q}}
\newcommand{\bR}{\mathbf{R}}

\newcommand{\bU}{\mathbf{U}}
\newcommand{\bV}{\mathbf{V}}
\newcommand{\bW}{\mathbf{W}}

\newcommand{\bfg}{\mathbf{g}}

\newcommand{\bfu}{\mathbf{u}}
\newcommand{\bfv}{\mathbf{v}}

\newcommand{\bfx}{\mathbf{x}}
\newcommand{\bfy}{\mathbf{y}}
\newcommand{\bfz}{\mathbf{z}}

\newcommand{\bfB}{\mathbf{B}}

\newcommand{\bfE}{\mathbf{E}}
\newcommand{\bfF}{\mathbf{F}}

\newcommand{\bfI}{\mathbf{I}}
\newcommand{\bfJ}{\mathbf{J}}
\newcommand{\bfK}{\mathbf{K}}
\newcommand{\bfL}{\mathbf{L}}
\newcommand{\bfM}{\mathbf{M}}

\newcommand{\bfP}{\mathbf{P}}
\newcommand{\bfQ}{\mathbf{Q}}
\newcommand{\bfR}{\mathbf{R}}

\newcommand{\bfU}{\mathbf{U}}
\newcommand{\bfV}{\mathbf{V}}
\newcommand{\bfW}{\mathbf{W}}

\newcommand{\bPi}{{\boldsymbol{\Pi}}}
\newcommand{\btheta}{{\boldsymbol{\theta}}}

\newfont{\Bb}{msbm10 scaled\magstep0}
\def\IR{\mbox {\Bb R}}
\def\IC{\mbox {\Bb C}}

\title{Structure-preserving model reduction 
for nonlinear port-Hamiltonian systems\thanks{This work was supported in
part by NSF through Grant DMS-1217156.}}

\author{S. Chaturantabut\thanks{Department of Mathematics and Statistics,Thammasat University, Pathumthani, 12120, Thailand. {\small \tt saifon@mathstat.sci.tu.ac.th}} \and C.~Beattie\thanks{Department of Mathematics,
Virginia Polytechnic Institute and State University,
460 McBryde, Virginia Tech,
Blacksburg, VA 24061-0123.  {\small \tt \{beattie,gugercin\}@math.vt.edu}} \and S.~Gugercin$^\ddagger$}

\begin{document}
\maketitle
\footnotetext{DRAFT: \today}

\begin{abstract}
This paper presents a structure-preserving model reduction approach applicable to  
large-scale, nonlinear port-Hamiltonian systems.  Structure preservation in the reduction step
ensures the retention of port-Hamiltonian structure which, in turn, assures the stability and passivity of the reduced 
model.  Our analysis provides  \emph{a priori} error 
bounds for both state variables and outputs.  Three techniques are considered for constructing
bases needed for the reduction: one that utilizes proper orthogonal decompositions;
one that utilizes $\mathcal{H}_2/\mathcal{H}_{\infty}$-derived optimized bases; and one that is a mixture of the two.    
The complexity of evaluating the reduced nonlinear term is managed efficiently using a modification of the discrete
empirical interpolation method (\textsc{deim}) that also preserves port-Hamiltonian structure. 
The efficiency and accuracy of this model reduction framework
are illustrated with two examples: a nonlinear ladder network 
and a tethered Toda lattice.
\end{abstract}


\begin{keywords}
nonlinear model reduction, proper orthogonal decomposition,
port-Hamiltonian, $\mathcal{H}_2$ approximation, structure preservation
\end{keywords}

\begin{AMS}
37M05, 65P10, 93A15
\end{AMS}

\pagestyle{myheadings}
\thispagestyle{plain}
\markboth{S. CHATURANTABUT, C. BEATTIE, AND S. GUGERCIN}{ Model reduction for nonlinear port-Hamiltonian systems}


\section{Introduction and Background}

The modeling of complex physical systems often
involves systems of coupled partial differential equations, which upon spatial discretization, lead to dynamical models and systems of 
\emph{ordinary} differential equations with very large state-space dimension.  
This motivates  \emph{model reduction} methods that produce
 low dimensional surrogate models capable of mimicking 
the input/output behavior  of the original system model. 
Such reduced-order models could then be used as proxies,
replacing the original system model in various computationally 
intensive contexts that are sensitive to system order, for example as a component 
in a larger simulation.  
Dynamical systems frequently have structural features that reflect underlying physics and conservation laws characteristic of the phenomena modeled.  Reduced models that do not share such key structural features with 
the original system may produce response artifacts that are ``unphysical"  and 
as a result, such reduced models may be unsuitable for 
use as dependable surrogates for the original system, even if they otherwise yield high response fidelity.  
The key system feature that we wish to retain in our reduced models will be port-Hamiltonian structure.  In a certain sense,  
this will be an expression of system passivity. 

\subsection{Port-Hamiltonian systems}
Models of dynamic phenomena may be constructed within a system-theoretic network modeling paradigm that formalizes the interconnection of naturally specified subsystems.  If the core dynamics of subsystem components are described by variational principles (e.g., least-action or virtual work),  the aggregate system model typically has structural features that characterize it as a 
\emph{port-Hamiltonian system}.  While greater generality is both possible and useful (see, in particular, the review article \cite{vanderSchaft2006phs}, and the monographs \cite{duindam2009modeling} and \cite{zwart2009dpp}),  it will suffice to consider 
realizations of finite-dimensional nonlinear port-Hamiltonian (\textsc{nlph}) systems that appear as:
\begin{equation}\label{PHdef}
\begin{array}{l}
\dot{\bfx}=\left(\bfJ-\bfR\right)\nabla_{\!\bfx}{H}(\bfx)+\bfB\bfu(t) \\[.1in]
\bfy=\bfB^T\nabla_{\!\bfx}{H}(\bfx),
\end{array}
\end{equation}
where $\bfx \in \IR^n$ is the $n$-dimensional state vector; ${H}:\IR^n\rightarrow [0,\infty)$ is a continuously differentiable scalar-valued vector function - the \emph{Hamiltonian}, describing the internal energy of the system as a function of state; $\bfJ=-\bfJ^T \in \IR^{n \times n}$ is the \emph{structure matrix} describing the interconnection of energy storage elements in the system; $\bfR=\bfR^T\geq \mathbf{0}$ is the $n\times n$ \emph{dissipation matrix} describing energy loss in the system; and, $\bfB\in\IR^{n\times m}$ is the \emph{port matrix}  describing how energy enters and exits the system.   We will always assume that the Hamiltonian is bounded below and so without loss of generality, strictly positive, ${H}(\bfx)>0$ for all $\bfx$. 

The family of systems characterized by (\ref{PHdef}) generalizes the classical notion of \emph{Hamiltonian systems} which would be expressed in our notation as  $\displaystyle \dot{\bfx}=\bfJ\nabla_{\!\bfx}{H}(\bfx)$.  
The analog of conservation of energy for Hamiltonian systems becomes for (\ref{PHdef}):
\begin{equation}\label{dissIneq}
 {H}(\bfx(t_1))-{H}(\bfx(t_0)) \leq \int_{t_{0}}^{t_{1}} \bfy(t)^T\bfu(t)\ dt,
\end{equation}
 which is to say, the change in the internal energy of the system, as measured by  ${H}$, is bounded by the total work done on the system.    
 The presumed positivity of $H$ conforms with its use as a ``supply function" in the sense of Willems \cite{willems1972dissipativePt1}, and so 
  \textsc{nlph} systems are always \emph{stable} and \emph{passive}.   Furthermore, the class of  \textsc{nlph} systems defined as in (\ref{PHdef}) is closed under power conserving interconnection - connecting port-Hamiltonian systems together produces an aggregate system that must also be port-Hamiltonian, and hence \emph{a fortiori}, must be both stable and passive.  
 This last fact provides compelling motivation to preserve port-Hamiltonian structure when producing low-order surrogate models intended to be used as proxies for systems of the sort defined by (\ref{PHdef}).  
 
 We assume in all that follows that the matrices $\bfJ$, $\bfR$, and $\bfB$ are \emph{constant}, however this assumption is adopted here largely for convenience.  $\bfJ$, $\bfR$, and $\bfB$ may each depend on the state vector, $\bfx$, input vector, $\bfu$, and may also carry an explicit time dependence, all without introducing any complications to port-Hamiltonian aspects of system structure; the dissipation inequality (\ref{dissIneq}) will still hold and the system remains passive. For this reason, port-Hamiltonian systems can accommodate a very rich variety of nonlinear interactions.  Moreover, the structure-preserving strategies for model reduction that we describe below may be adapted with negligible modification in this more complex setting.   
 
\subsection{Petrov-Galerkin reduced models}
Most model reduction approaches involve some variation of a Petrov-Galerkin projective approximation to the equations describing the system dynamics.  This  proceeds by  choosing two subspaces of $\IR^n$: an $r$-dimensional trial subspace, 
${\mathcal V}_r\subset \IR^n$, and an $r$-dimensional test subspace,  ${\mathcal W}_r\subset \IR^n$.
It is convenient and nonrestrictive in practice to assume additionally that ${\mathcal V}_r$ and ${\mathcal W}_r$ have 
a ``generic orientation" with respect to one another so that neither subspace contains any nontrivial vectors that are orthogonal to all vectors in the other subspace. 
The evolution of an associated reduced-order model may be described in the following (initially indirect) way:
\begin{equation} \label{PetrGal}
\begin{array}{c}
\mbox{Find a trajectory, } \bfv_r(t), \mbox{ contained in }{\mathcal V}_r\mbox{ such that }\\
 \dot{\bfv}_r(t) -\left(\bfJ-\bfR\right)\nabla_{\!\bfx}{H}(\bfv_r)-\bfB\bfu(t) 
\quad \perp \quad {\mathcal W}_r;  \\
 \mbox{the  associated output is } \bfy_r(t)= \bfB^T\nabla_{\!\bfx}{H}(\bfv_r).
\end{array}
\end{equation}

The dynamics  described by (\ref{PetrGal}) 
can be represented directly as a dynamical system evolving in a state-space of reduced dimension $r$ once bases are chosen for the two subspaces  ${\mathcal V}_r$ and  ${\mathcal W}_r$.  Let 
$\mbox{\textsf{Ran}}(\bfM)$ denote the range of a matrix $\bfM$. 
Define matrices  $\bfV_{\! r},\, \bfW_{\! r}\in \IR^{n \times r}$
 so that $\mathcal{V}_{r}=\mbox{\textsf{Ran}}(\bfV_{\! r})$ and 
$\mathcal{W}_{r}=\mbox{\textsf{Ran}}(\bfW_{\! r})$. 
  We can represent reduced system trajectories as 
$\bfv_r(t)= \bfV_r \bfx_r(t)$ with $\bfx_r (t)\in \IR^r$ for each $t$;  
the Petrov-Galerkin approximation (\ref{PetrGal})
can be rewritten as 
\begin{align*}
\bfW_{\! r}^{T}&\left[\bfV_{\! r}\dot{\bfx}_r (t)-\left(\bfJ-\bfR\right)\nabla_{\!\bfx}{H}(\bfV_r \bfx_r)-\bfB\bfu(t) \right]=\mathbf{0}\\
\mbox{and} & \quad \bfy_r(t) = \bfB^T\nabla_{\!\bfx}{H}(\bfV_r \bfx_r).
\end{align*}
Since ${\mathcal V}_r$ and ${\mathcal W}_r$ are assumed to have 
a generic orientation with respect to one another, $\bfW_{\! r}^{T}\bfV_{\! r}$ is invertible and we may choose bases  for  ${\mathcal V}_r$ and ${\mathcal W}_r$ such that $\bfW_{\! r}^{T}\bfV_{\! r}=\bfI$. 
This leads to a  state-space representation of a reduced-order nonlinear dynamical system approximating (\ref{PHdef}):
\begin{equation}\label{PetrGalROM}
\begin{array}{l}
\dot{\bfx}_r=\bfW_r^{T}\left(\bfJ-\bfR\right)\nabla_{\!\bfx}{H}(\bfV_r \bfx_r) +\bfW_r^{T}\bfB\bfu(t) \\[.1in]
\bfy_r=\bfB^T\nabla_{\!\bfx}{H}(\bfV_r \bfx_r),
\end{array}
\end{equation}

Typically $r\ll n$ and (\ref{PetrGalROM}) describes a reduced-order model for the original system (\ref{PHdef}).
There are two shortcomings that may be anticipated.    First, (\ref{PetrGalROM}) will not have the form of (\ref{PHdef}) 
unless special subspaces are chosen, and so, (\ref{PetrGalROM})  will not typically be a \textsc{nlph} system and passivity may be lost.   
Secondly, if the Hamiltonian function, $H$, is non-quadratic, each evaluation of 
$\nabla_{\!\bfx}{H}(\bfV_r \bfx_r) $ in (\ref{PetrGalROM}) occurring in the course of a simulation will likely require a lifting of $\bfx_r$ to $\IR^n$ (implicit in the formation of $\bfV_r \bfx_r $), and so direct simulation of (\ref{PetrGalROM}) is still likely to have complexity proportional to $n\gg r$;  little or no savings may be realized from reducing the system order.

We consider each of these issues in subsequent sections.  
In  \S \ref{sec:MOR_Proj_NPH}, a structure-preserving 
model reduction approach for large-scale  \textsc{nlph} systems  
will be introduced.  
This approach is built  upon 
Petrov-Galerkin projections that are modified to assure that the resulting reduced system 
retains port-Hamiltonian structure; thus stability and passivity. 
Three types of reduced-order 
bases used to define these projections will be considered:  
\begin{inparaenum}[(i)] \item one based on the Proper Orthogonal Decomposition (\textsc{pod}), 
\item one derived from $\mathcal{H}_2$-optimal approaches for a related linear problem 
(which we refer to as  ``$\HtwoEps$-bases"), \item hybrid $\HtwoEps$- \textsc{pod} bases that combine both types. 
\end{inparaenum} The bases (i) and (ii) were originally considered in \cite{Beattie2011_NPH}.
 Numerical experiments in \S \ref{sec:MOR_PH_exLadder} illustrate that the hybrid $\HtwoEps$- \textsc{pod} bases 
significantly outperform the other two.
 In \S \ref{sec:MOR_Proj_NPHerror}, we develop corresponding error analyses and bounds for 
reduced states and  outputs. 
In order to resolve the ``lifting bottleneck" described above, we develop, in  \S \ref{sec:MOR_DEIMstr},  a variant of  the  Discrete Empirical Interpolation Method (\textsc{deim}) \cite{ChatSorDEIM_siam2010} that incorporates the structure-preserving model reduction approach of 
\S \ref{sec:MOR_Proj_NPH}. 
In \S \ref{sec:Numer_examples}, 
the efficiency and accuracy of our approach 
are illustrated with two examples: a nonlinear ladder network 
and a tethered Toda lattice.
Corresponding \emph{a priori} error bounds 
for states and outputs  are derived in \S \ref{sec:MOR_DEIMstr_err}.


\section{Preserving port-Hamiltonian Structure in Reduced Models}
\label{sec:MOR_Proj_NPH}

The process of obtaining a reduced model from an original full-order model can be viewed as one of identifying and preserving high-value portions of the state space, i.e., portions of the state space that contribute substantively to the system response.  We proceed with the following heuristics: Suppose we have identified two $r$-dimensional subspaces, 
$\mathcal{V}_{r}$ and $\mathcal{W}_{r}$, that are ``high-value" in the sense that
 for ``most" input signal profiles, $\bfu(t)$,  in (\ref{PHdef}), we have  
\begin{equation} \label{GoodSpace}
{\mbox{\small\textsc{High Value}} \atop \mbox{\small\textsc{ PH-Spaces: }}}
\left\{ \rule{0mm}{11mm}\right.\mbox{\raisebox{8mm}
{\hspace*{-8mm}\begin{minipage}[t]{3.8in} {\small 
 \begin{itemize}
\item The associated trajectory, $\bfx(t)$, stays ``close" to $\mathcal{V}_{r}$, so that\\
 $\bfx(t)\approx \bv_r(t)$ for some $\bv_r(t)\in \cV_r$, and \\[-2mm]
\item the internal force, $\nabla_{\!\bfx}{H}(\bfx)$, stays ``close" to $\mathcal{W}_{r}$, so that\\
 $\nabla_{\!\bfx}{H}(\bfx(t))\approx \bw_r(t)$ for some $ \bw_r(t) \in \cW_r$.
\end{itemize}  }
 \end{minipage}
 } }
 \end{equation}
 Evidently if  $\mathcal{V}_{r}=\mathsf{Ran}(\bV_r)$ and $\mathcal{W}_{r}=\mathsf{Ran}(\bW_r)$, then
$\bv_r(t)=\bfV_r\bfx_r(t)$ for some trajectory $\bfx_r\in\IR^r$ and  $\bw_r(t)=\bfW_r\bff_r(t)$,  for some choice of $\bff_r\in\IR^r$.
Exactly what comprises ``most" input signal profiles and how one measures ``closeness" in (\ref{GoodSpace}) will vary depending on context; 
we consider different possibilities  later in this section.   For the time being, note that if $\bfx(t)\approx \bfV_r\bfx_r(t)$ then plausibly, 
$$\nabla_{\!\bfx}{H}(\bfV_r\bfx_r(t))\approx\nabla_{\!\bfx}{H}(\bfx(t))\approx \bfW_r\bff_r(t).$$

Significantly, these statements amount to assertions about the \emph{subspaces}, $\mathcal{V}_{r}$ and $\mathcal{W}_{r}$, and do not constrain the choice of bases for the subspaces.  As long as $\mathcal{V}_{r}$ and $\mathcal{W}_{r}$ have a generic orientation with respect to one another,  bases may be chosen so that $\bfW_{\! r}^{T}\bfV_{\! r}=\bfI$.  With this in mind, note further that 
\begin{equation} \label{reduced_f}
\bff_r(t)=\bfV_r^T\bfW_r\bff_r(t)\approx\bfV_r^T\nabla_{\!\bfx}{H}(\bfV_r\bfx_r(t))= \nabla_{\!\bfx_r}{H_r}(\bfx_r(t)),
\end{equation}
where we have introduced a \emph{reduced Hamiltonian}, $H_r(\bfx_r)=H(\bfV_r\bfx_r)$. Thus, 
\begin{equation} \label{redHam}
\nabla_{\!\bfx}{H}(\bfV_r\bfx_r(t))\approx \bfW_r\nabla_{\!\bfx_r}{H_r}(\bfx_r(t)).
\end{equation}

Substituting $\bfV_r\bfx_r(t)$ for $\bfx(t)$ and $\bfW_r\nabla_{\!\bfx_r}{H_r}(\bfx_r(t))$ for $\nabla_{\!\bfx}{H}(\bfx(t))$ 
in (\ref{PHdef}) and then multiplying by $\bfW_r^T$, 
leads to a state-space representation of a reduced port-Hamiltonian approximation: 
\begin{equation} \label{redPHdef}
\begin{array}{l}
 \dot{\mathbf{x}}_r =(\bfJ_r-\bfR_r)\nabla_{\!\bfx_r}H_r(\mathbf{x}_r) +\bfB_r\bfu(t),\\[.1in]
      \quad    \mathbf{y}_r(t) =\mathbf{B}_r^T\, \nabla_{\!\bfx_r}H_r(\mathbf{x}_r) 
\end{array}
\end{equation} 
with $H_r(\bfx_r)=H(\bfV_r\bfx_r)$,
$\bfJ_r= \bfW_r^{T} \bfJ \bfW_r$,  $\bfR_r= \bfW_r^{T} \bfR \bfW_r$, and $\bfB_r= \bfW_r^{T} \bfB$. 
Note that $H_r:\IR^r\rightarrow [0,\infty)$ is continuously differentiable; $\bfJ_r=-\bfJ_r^T$; and $\bfR_r=\bfR_r^T\geq \mathbf{0}$, so (\ref{redPHdef}) retains the structure of  (\ref{PHdef}) and is a port-Hamiltonian system. 

The earlier works, \cite{fujimoto2007brm} and \cite{scherpen2008spm}, also offer structure-preserving model reduction methods  for \textsc{nlph} systems.  These papers exploit the Kalman decomposition and balanced truncation in deriving reduced models of \textsc{nlph} systems. Obtaining the Kalman decomposition of the full-order original system or balancing it,  is computationally demanding  for  nonlinear systems of even modest order; see e.g.  \cite{fujimoto2010brm,scherpen2008spm} and  references therein.   Such approaches are infeasible for the problem class we consider, which may have thousands of state-variables.   In what follows, we develop approaches 
that remain feasible for this problem class; they depend on the construction of low-dimensional projecting subspaces motivated by the heuristics in (\ref{GoodSpace}) .

\subsection{POD subspaces}
\label{sec:basis_POD}

The Proper Orthogonal Decomposition (\textsc{pod}) is a natural approach to producing  high-value modeling spaces as described in (\ref{GoodSpace}).  \textsc{pod} is  a popular approach to (unstructured) model reduction (\cite{lumley,sirovich1}) which we adapt to our setting as follows: Fix a square integrable input signal, $\bfu(t)$, for the system (\ref{PHdef}).  The corresponding trajectory, $\bfx(t)$, will then also be square integrable.  Denoting orthogonal projections, $\bfP$ and $\bfQ$, we consider the minimization problem:
\begin{eqnarray*}
\bfP_{\!\star}&= \begin{array}{c}\mbox{argmin} \\ {\tiny \mathsf{rank}(\bfP) = r} \end{array}
 \int_0^{\infty}\|\left(\bfI-\bfP\right)\bfx(t)\|^2\,dt~~ \\ {\rm and}\qquad~\\
 \bfQ_{\star}&= \begin{array}{c}\mbox{argmin} \\ {\tiny \mathsf{rank}(\bfQ) = r} \end{array}
\int_0^{\infty}\|\left(\bfI-\bfQ\right) \nabla_{\!\bfx}{H}(\bfx(t))\|^2\,dt. ~~
\end{eqnarray*}
We would like to take $\mathcal{V}_{r}=\mathsf{Ran}(\bfP_{\!\star})$ and $\mathcal{W}_{r}=\mathsf{Ran}(\bfQ_{\star})$,
but this is not a computationally tractable approach as it stands.  If the integrals are truncated and then 
approximated with a Trapezoid Rule, one arrives at a characterization of \textsc{pod} subspaces which is incorporated into our first method, summarized as \emph{Algorithm 1}.  As our numerical results show (and consistent with common experience), the use of \textsc{pod}  in \emph{Algorithm 1} provides subspaces $\mathcal{V}_r$ and $\mathcal{W}_r$ that can be very effective in capturing dynamic features that are present in the original sampled system response.  However, as an empirical method, it is incapable of providing information about dynamic response features that are absent in the sampled system response, but that \emph{could} have been present had a different choice of input profile been made.   One can mitigate this difficulty somewhat by extending \textsc{pod} by including a representative sampling of input profiles, but this leaves open the question of what constitutes a representative sampling of input profiles.  A different approach leads to our next class of subspaces.

\begin{algorithm}[tbp]
\caption{: Structure-preserving \textsc{pod} reduction of NLPH systems (\podph)}
 \label{alg:podph}
\vspace{2mm}

\begin{algorithmic}[1]
\STATE 
Generate a trajectory $\bfx(t)$, and collect snapshots:\\
\quad $\mathbb{X}=\left[\bfx(t_0),\bfx(t_1),\bfx(t_2),\ldots,\bfx(t_N)\right]$
\vspace{2mm}
\STATE    
Simultaneously collect associated \emph{force} snapshots: \\
\quad $\mathbb{F}=\left[\nabla_{\!\bfx}H(\bfx(t_0)),\nabla_{\!\bfx}H(\bfx(t_1)),\ldots,\nabla_{\!\bfx}H(\bfx(t_N))\right]$.
\vspace{2mm}
\STATE    
Truncate an SVD of the snapshot matrix, $\mathbb{X}$, to get a \textsc{pod} basis, $\widetilde{\bfV}_r$, for a ``high-value" subspace of the state space. ($\bfx(t)\approx \widetilde{\bfV}_r\tilde{\bfx}_r(t)$)
\vspace{2mm}
\STATE    
Truncate an SVD of $\mathbb{F}$ to get a second \textsc{pod} basis, $\widetilde{\bfW}_r$, spanning a  second ``high-value" subspace approximating the range of 
$\nabla_{\!\bfx}H(\bfx(t))\approx \widetilde{\bfW}_r \tilde{\bff}_r(t)$.
\vspace{2mm}
\STATE    Change bases $\widetilde{\bfW}_r \mapsto \bfW_r$ and $\widetilde{\bfV}_r\mapsto \bfV_r$ such that~$\bfW_r^T\bfV_r=\bfI$.
\vspace{2mm}
\STATE
With $\bfV_r$ and $\bfW_r$ determined in this way, the \textsc{pod-ph} reduced port-Hamiltonian system is then specified by (\ref{redPHdef}).    
\end{algorithmic}
\end{algorithm}

\subsection{\HtwoEps-optimal subspaces}
\label{sec:basis_HtwoEps}

We next consider a choice of subspaces, $\mathcal{V}_r$ and $\mathcal{W}_r$, that are optimal (in a sense we will describe) for 
\emph{all possible} input profiles that are sufficiently small (``$\varepsilon$-optimal").  It is often the case that this choice will be effective for larger input profiles as well.   In contrast to the previous \textsc{pod} approach, no choice of inputs is necessary and no simulations need to be performed in order to derive the approximating subspaces.  The proviso that (virtual) inputs be sufficiently small, but otherwise arbitrary, allows us to tailor the subspaces $\mathcal{V}_r$ and $\mathcal{W}_r$, so as to provide near-optimal reduction for the corresponding \emph{linearized} port-Hamiltonian model.  Note that linearization is used here only as a tool to obtain useful information that will be encoded into  the projection subspaces, $\mathcal{V}_r$ and $\mathcal{W}_r$, which are then used for reduction of the \emph{nonlinear} system (\ref{PHdef}). 

Any input profile, $\bfu(t)$, may be scaled to have sufficiently small magnitude so that the resulting trajectory, $\bfx(t)$, is small as well.  Then linear terms in the internal forcing dominate, and
 $\nabla_{\!\bfx}H( \bfx ) \approx \bfQ\bfx$ for some symmetric positive definite matrix $\bfQ\in\IR^{n\times n}$.  Indeed, $\bfQ=\nabla^2 H(\mathbf{0})$, the Hessian matrix for $H(\bfx)$ evaluated at $\bfx=\mathbf{0}$.   
 This approximation leads to an ancillary \emph{linear port-Hamiltonian system}:
\begin{equation} \label{eq:lphfom}
\begin{array}{l}
\dot{\bfx}=\left(\bfJ-\bfR\right)\bfQ\,\bfx +\bfB\bfu(t) \\[.1in]
\bfy=\bfB^T\bfQ\,\bfx.
\end{array}
\end{equation}

We proceed to construct projecting subspaces, $\mathcal{V}_r$ and $\mathcal{W}_r$, that would produce,
 via the reduction described in (\ref{redPHdef}), an effective reduced  port-Hamiltonian model approximating (\ref{eq:lphfom}):
 \begin{equation} \label{eq:lphrom}
 \begin{array}{l}
           \dot{\mathbf{x}}_r=(\bfJ_r-\bfR_r)\bfQ_r\bfx_r +\bfB_r\bfu(t), \\[.1in]
           \mathbf{y}_r(t)=\mathbf{B}_r^T\bfQ_r\bfx_r,\
 \end{array}
  \end{equation}
where  $\bfJ_r= \bfW_r^{T} \bfJ \bfW_r$,  $\bfR_r= \bfW_r^{T} \bfR \bfW_r$, $\bfQ_r= \bfV_r^{T} \bfQ \bfV_r$
and $\bfB_r= \bfW_r^{T} \bfB$.  Observe that $\bfJ_r= -\bfJ_r^T$, $\bfR_r= \bfR_r^T \geq 0$ and $\bfQ_r= \bfQ_r^T > 0$. 
The quality of the subspaces is interpreted now as how well (\ref{eq:lphrom}) approximates (\ref{eq:lphfom}). 
This, in turn, may be interpreted as a rational approximation problem: 
We associate the linear dynamical systems (\ref{eq:lphfom}) and (\ref{eq:lphrom}) with their transfer functions,
\begin{equation} \label{eq:G}
 \begin{array}{c}
\cbfG(s)  = \bfB^T( s \bfQ^{-1} - (\bfJ-\bfR))^{-1}\bfB \quad \mbox{and}  \\[.1in]
\cbfG_{r}(s)  = \bfB_r^T( s \bfQ_r^{-1} - (\bfJ_r-\bfR_r))^{-1}\bfB_r,
 \end{array}
\end{equation}
respectively.  If $\cbfG_{r}(s)$ approximates $\cbfG(s)$ well with respect to some (appropriately chosen) norm, then the reduced system outputs will approximate the full order system outputs uniformly well over all inputs with bounded energy (square integrable);  our model reduction problem has been reduced to a  rational approximation problem.  We wish to find a degree-$r$ rational function, $\cbfG_{r}(s)$ having the structure given in (\ref{eq:G}) that also approximates $\cbfG(s)$ well. Tangential rational interpolation provides a useful tool for this: Given $r$ interpolation points $\sigma_1,\ldots, \sigma_r$ in the complex plane with corresponding tangent directions 
$\{\mathsf{b}_1,\ldots, \mathsf{b}_r\}\in \IC^m$, construct an $r$th order system,  $\cbfG_r$, so that $\cbfG_r$ is port-Hamiltonian, and  
\begin{equation} \label{eq:intph}
\cbfG_r(\sigma_i)\mathsf{b}_i=\cbfG(\sigma_i)\mathsf{b}_i \quad \mbox{for}\quad i=1,\ldots,r.
\end{equation}
A solution to this problem was given in \cite{gugercin2011spt}.
\begin{theorem}
  \label{thm:mimointph}
Given interpolation points $\sigma_1,\ldots, \sigma_r$  and tangent directions \\
$\mathsf{b}_1,\ldots, \mathsf{b}_r$, construct 
{\small
\begin{eqnarray*}
\widetilde{\bfV}_r = [(\sigma_1 \bfI - (\bfJ-\bfR)\bfQ)^{-1}\bfB\mathsf{b}_1,\ldots,(\sigma_r \bfI - (\bfJ-\bfR)\bfQ)^{-1} \bfB
\mathsf{b}_r].
\end{eqnarray*}
}
Define the Cholesky factorization of  $\widetilde{\bfV}_r^T\bfQ\widetilde{\bfV}_r = \bfR^T \bfR$, assign
 $\bfV_r =\widetilde{\bfV}_r\bfR^{-1}$, and then construct $\bfW_r = \bfQ \bfV_r$.  
 Set
\begin{equation} \label{redJQR}
\bfJ_r = \bfW_r ^T\bfJ \bfW_r,~\bfQ_r =  \bfV_r^T \bfQ \bfV_r  =  \bfI_r,~
\bfR_r = \bfW_r ^T\bfR \bfW_r, \quad \mbox{and}\quad  \bfB_r = \bfW_r ^T\bfB.
\end{equation}
\begin{equation} \label{PHRed}
\mbox{Then~the~reduced~model,}\ \cbfG_r: \qquad
\begin{array}{l}
\dot \bfx_r =  (\bfJ_r - \bfR_r)\bfQ_r\bfx_r + \bfB_r\,\bfu \\
\bfy_r =  \bfB_r^T \bfQ_r \bfx_r
\end{array} \qquad \qquad
\end{equation}
is port-Hamiltonian (hence stable and passive) and also satisfies the interpolation conditions (\ref{eq:intph}). 
\end{theorem}

\vspace{1ex}
For information on transfer function interpolation in the special case of single-input/single-output port-Hamiltonian systems, see \cite{Polyuga_vdSchaft2009_infinity,gugercin2009ibh,Polyuga_vdSchaft2010_s0}.  For an overview of model reduction methods for linear port-Hamiltonian systems, see 
\cite{Polyuga_2010_thesis}.
\subsubsection{\HtwoEps~port-Hamiltonian approximation}
 port-Hamiltonian approximations of reduced order may be constructed using Theorem \ref{thm:mimointph} once shifts, $\{\sigma_i\}$,  and  tangent directions, $\{\mathsf{b}_i\}$, are chosen,  but this does not give any information on how best to choose  $\{\sigma_i\}$ and $\{\mathsf{b}_i\}$. 
We discuss issues related to finding effective approximations with respect to the $\mathcal{H}_2$ norm: The $\mathcal{H}_2$ norm of $\cbfG$ is defined as 
$$
\left\| \cbfG\right\|_{\mathcal{H}_2} = \left(\frac{1}{2\pi}\int_{-\infty}^{\infty}
  \left\| \cbfG(\imath \omega )  \right\|_{\rm F}^2 d\omega\right)^{1/2}.
$$
Let $\cbfG_r(s)$ minimize the $\mathcal{H}_2$ error $\left\| \cbfG - \cbfG_r\right\|_{\mathcal{H}_2} $ 
over all possible degree-$r$ rational functions and suppose that $\cbfG_r(s)$ has  a partial fraction expansion
 $\cbfG_r(s)=\sum_{k=1}^r\frac{\mathsf{c}_k\mathsf{b}_k^T}{s-\widehat{\lambda}_k}$. Then, as shown in \cite{gugercin2008hmr}, the interpolation conditions 
\begin{equation} 
\label{eqn:h2cond}
\cbfG(-\widehat{\lambda}_k)\mathsf{b}_k = \cbfG_r(-\widehat{\lambda}_k)\mathsf{b}_k,~~~{\rm for}~k=1,\ldots,r
\end{equation}
are necessary conditions for $\mathcal{H}_2$ optimality (there are additional conditions that must also be satisfied in general).  In other words, 
the optimal $\mathcal{H}_2$ approximant $\cbfG_r$ is a tangential interpolant to $\cbfG$ at the mirror images of the reduced-order poles. For the full set of necessary conditions required for optimality, see \cite{gugercin2008hmr}.

A method was introduced in \cite{gugercin2011spt} that produces 
an interpolatory reduced-order port-Hamiltonian system satisfying the conditions given in
(\ref{eqn:h2cond}). Since the interpolation points $-\widehat{\lambda}_k$ and the tangential directions 
$\mathsf{b}_k$ depend on the reduced-model to be computed, an iterative process is used  
to correct the interpolation points and tangential directions until the desired conditions in (\ref{eqn:h2cond})
are obtained.  These constitute only a subset of the necessary conditions required for $\mathcal{H}_2$-optimality. 
The remaining degrees of freedom are used in maintaining port-Hamiltonian structure.  
An algorithm that accomplishes was introduced in  \cite{gugercin2011spt}. Algorithm \ref{alg:qh2ph} below 
uses this methodology to construct the model reduction bases for the second structure-preserving 
model reduction of \textsc{nlph} systems.
\begin{algorithm}[h!tbp]                 
\caption{: Structure-preserving \HtwoEps~ reduction of \textsc{nlph} systems 
 (\HtwoEpsph)}
\label{alg:qh2ph}                        
\vspace{2mm}
\begin{algorithmic} [1]                
\STATE Linearize the \textsc{nlph} system (\ref{PHdef}) to obtain a linear port-Hamiltonian system of the form in (\ref{eq:lphfom}).
    \STATE  
    Make an initial selection of shifts $\{\sigma_i\}_1^r$, and tangent directions $\{\mathsf{b}_i\}_1^r$.
    \vspace{2mm}
    \WHILE{(not converged)}
        \STATE 
            $\widehat{\bfV}_r = [(\sigma_1 \bfI - (\bfJ-\bfR)\bfQ)^{-1}\bfB\mathsf{b}_1,\ldots,$
             $ (\sigma_r \bfI - (\bfJ-\bfR)\bfQ)^{-1} \bfB\mathsf{b}_r]$
		\vspace{2mm}             
        \STATE 
		    Set $\bfV_r = \widehat{\bfV}_r\bfL^{-1}$ with $\widehat{\bfV}_r^T\bfQ\widehat{\bfV}_r = \bfL^T \bfL$ (so 	$\bfQ_r = \bfV_r ^T\bfQ \bfV_r =  \bfI_r$).
		  \vspace{2mm}
		  \STATE 
		  Set $\bfW_r = \bfQ\bfV_r$. {(so $\bfV_r^T\bfW_r =  \bfI_r$).}
		  \vspace{2mm}
		  \STATE 
		  Set $\bfJ_r = \bfW_r^T\bfJ \bfW_r $, $\bfR_r = \bfW_r^T\bfR \bfW_r $,  and $\bfB_r = \bfW_r^T\bfB$.
		  \vspace{2mm}
		  \STATE 
		   Calculate left eigenvectors: $\bfz_i^T (\bfJ_r-\bfR_r)= \lambda_i\bfz_i^T$.
		   \vspace{2mm}
		  \STATE 
		  Set $\sigma_i \longleftarrow -\lambda_i$ and $\mathsf{b}_i \longleftarrow \bfB_r^T\bfz_i$ for $i=1,\ldots,r$
		      \vspace{2mm}
     \ENDWHILE
     
     \vspace{2mm}
    \COMMENT{    Calculate final \HtwoEpsph~bases: }
    \vspace{2mm}
    	\STATE  Find $\widehat{\bfV}_r = [(\sigma_1 \bfI - (\bfJ-\bfR)\bfQ)^{-1}\bfB\mathsf{b}_1,\ldots,$
$ (\sigma_r \bfI - (\bfJ-\bfR)\bfQ)^{-1} \bfB\mathsf{b}_r]$
		\vspace{2mm}
   	 	\STATE 
    	Set $\bfV_r = \widehat{\bfV}_r\bfL^{-1}$ with $\widehat{\bfV}_r^T\bfQ\widehat{\bfV}_r = \bfL^T \bfL$. 
    	\vspace{2mm}
      	\STATE 
         Set $\bfW_r = \bfQ\bfV_r$. 
         \vspace{2mm}
        \STATE 
		With $\bfV_r$ and $\bfW_r$ determined in this way, construct the reduced nonlinear PH system
		using (\ref{redPHdef}).
\end{algorithmic}
\end{algorithm}

Note that we use the linearized port-Hamiltonian model only to obtain the \HtwoEpsph\ model reduction subspaces. Once
$\bfV_r$ and $\bfW_r$ are obtained using the \HtwoEps-approach outlined in
Algorithm \ref{alg:qh2ph}, we use these subspaces to reduce the original nonlinear system as  
shown in (\ref{redPHdef}).   See \cite{phillips2000pfm,rewienski2003tpl}
for other approaches that derive useful information from linearized systems 
in order to reduce nonlinear systems.
\subsection{A hybrid POD\,-\HtwoEps~approach}
\label{sec:PODH2}
The model reduction subspaces that \textsc{pod} provides are effective in capturing 
the dynamics that are represented in the original snapshot data; but naturally will miss features that are absent in this data. 
To resolve this issue in part, we have proposed to use $\varepsilon$-optimal  subspaces that were accurate for input profiles that are sufficiently 
small (``$\varepsilon$-optimal").   These subspaces are generated from a near-optimal reduction of a linearized port-Hamiltonian model using an approach proposed in .  As the trajectory moves away from the linearization point, the efficiency of these subspaces in capturing the true dynamics might degrade. Therefore, we propose to combine the subspaces resulting from \textsc{pod} (i.e., directly obtained from a simulation of the \textsc{nlph} system) together with the  $\varepsilon$-optimal subspaces resulting from the linearized model.   The structure-preserving reduction is applied as in  (\ref{redPHdef}) with the aggregate subspace.

For a given reduced dimension $r$, let 
$\widehat{\bV}_{\widehat{r}} = [\widehat{\bv}_1, \widehat{\bv}_2, \dots, \widehat{\bv}_{\widehat{r}} ] $ be the \textsc{pod-ph} basis of dimension $\widehat{r}$ obtained from Algorithm \ref{alg:podph} and 
let $\bar{\bV}_{\bar{r}} = [\bar{\bv}_1, \bar{\bv}_2, \dots, \bar{\bv}_{\bar{r}} ] $ be the \HtwoEpsph  \ basis of dimension $\bar{r}$ obtained from Algorithm \ref{alg:qh2ph}, where $\widehat{r}$ and $\bar{r}$ are positive integers with $ \widehat{r} + \bar{r}=r$. 
Then, the hybrid basis $\bV_r$ of dimension $r$ can be obtained from the orthonormal basis of concatenated matrix $[\widehat{\bV}_{\widehat{r}} \ \ \bar{\bV}_{\bar{r}}] \in \mathbb{R}^{n \times r}$
The other projection basis $\bW_r$ can be constructed similarly by combining 
\podph \ and \HtwoEpsph \ basis vectors. This leads to the third algorithm for structure-preserving model reduction of \textsc{nlph} systems.

\begin{algorithm}[h!tbp]
\caption{: Structure-preserving \podHtwoEps~ reduction of \textsc{nlph}  (\podHtwoEpsph)}
 \label{alg:podh2ph}
\vspace{2mm}

\begin{algorithmic}[1]
\STATE  Pick $ \widehat{r}$ and $\bar{r}=r$ so that $\widehat{r} + \bar{r}=r$.
\vspace{2mm}
\STATE   Obtain the \podph\ bases $\widehat{\bV}_{\widehat{r}}$ and $\widehat{\bW}_{\widehat{r}}$
of dimension $\widehat{r}$ using  Algorithm \ref{alg:podph}.
\vspace{2mm}
\STATE  Obtain the \HtwoEpsph\ bases $\bar{\bV}_{\bar{r}}$ and $\bar{\bW}_{\bar{r}}$
of dimension $\bar{r}$ using  Algorithm \ref{alg:qh2ph}.
\vspace{2mm}
\STATE    
Construct $\widetilde{\bV}_r$ of dimension $n\times r$ as the orthonormal basis of the concatenated matrix $[\widehat{\bV}_{\widehat{r}} \ \ \bar{\bV}_{\bar{r}}] \in \mathbb{R}^{n \times r}$\vspace{2mm}
\STATE    
Construct $\widetilde{\bW}_r$ of dimension $n\times r$ as the orthonormal basis of the concatenated matrix $[\widehat{\bW}_{\widehat{r}} \ \ \bar{\bW}_{\bar{r}}] \in \mathbb{R}^{n \times r}$\vspace{2mm}
\STATE   Change bases $\widetilde{\bfW}_r \mapsto \bfW_r$ and $\widetilde{\bfV}_r\mapsto \bfV_r$ such that~$\bfW_r^T\bfV_r=\bfI$.
\vspace{2mm}
\STATE
With $\bfV_r$ and $\bfW_r$ determined in this way, the \textsc{pod}-\HtwoEps~ reduced port-Hamiltonian system is then specified by (\ref{redPHdef}).    
\end{algorithmic}
\end{algorithm}

\subsection{An illustrative example}
\label{sec:MOR_PH_exLadder}
To illustrate the structure-preserving model reduction techniques described in Section~\ref{sec:MOR_Proj_NPH}, we consider an $N$-stage nonlinear ladder network (see Figure \ref{laddnetgraphic}) producing reduced models for three cases of bases:~(i) \textsc{pod} bases (Algorithm \podph), ~(ii) \HtwoEps~bases (Algorithm \HtwoEpsph), and (iii) hybrid \textsc{pod}-\HtwoEps~bases (Algorithm \podHtwoEpsph). 
The system has two inputs and two outputs. The two inputs are a voltage signal applied to the left-hand
terminal pair and a current injection across the right-hand
terminal pair. The symmetrically paired outputs are the
induced current across the left-hand terminal pair and the
induced voltage signal across the right-hand terminal pair.
For simplicity, we assume that each stage of the ladder network
is built from identical components and that the current
injection from the 
right is zero. Resistors and inductors are
assumed to behave linearly. The capacitors 
have a nonlinear C-V characteristic of the form
$
C_k(V)=\frac{C_0 \, V_0}{V_0+V}.
$

\begin{figure}[hhh]
\caption{Ladder network circuit topology (capacitors are nonlinear)}
  \centerline{
   {\includegraphics[width=5.0cm]{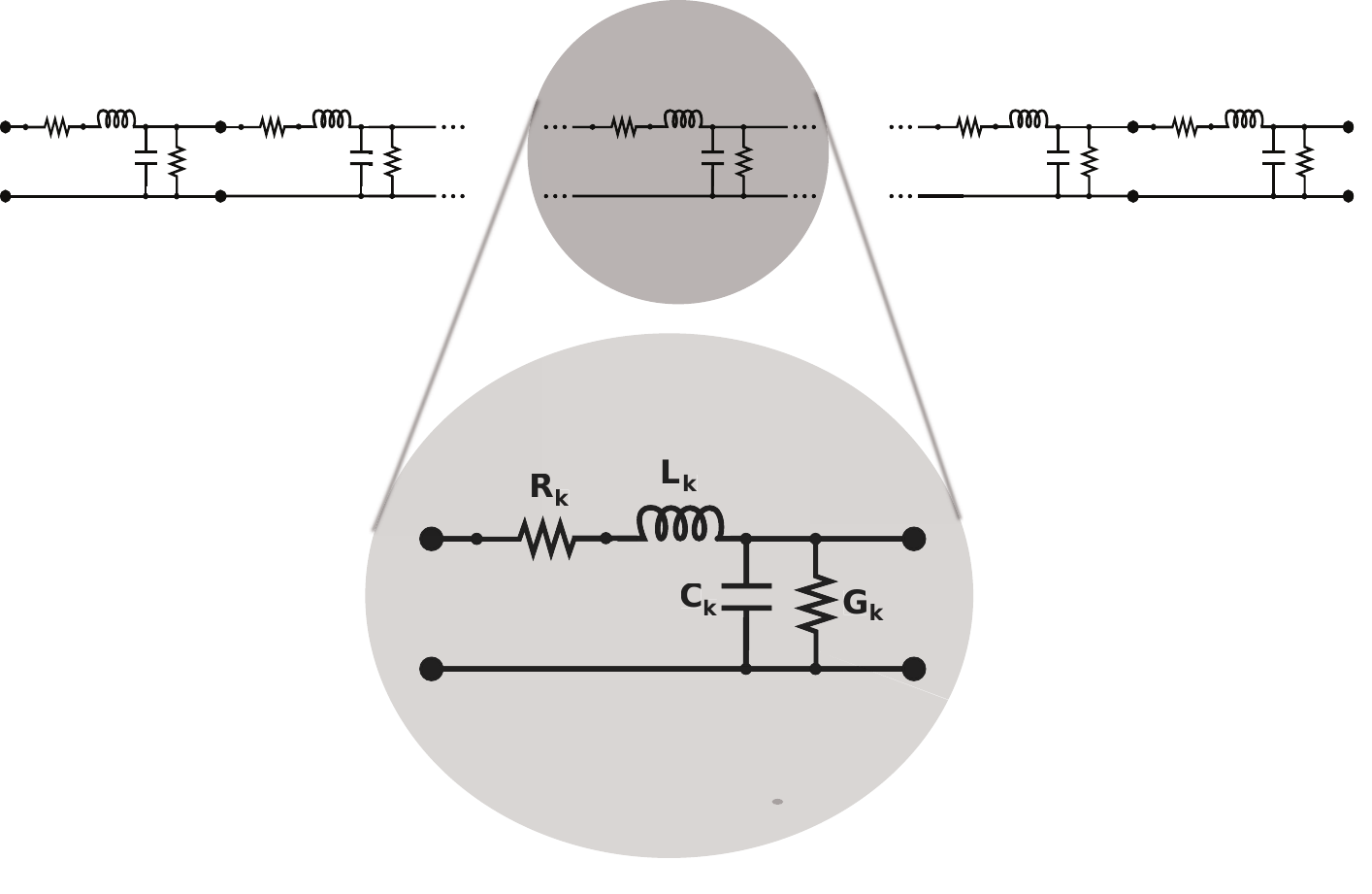}}
   } \label{laddnetgraphic}
 \end{figure}
Inductors and capacitors are evidently the energy storage elements of the circuit, so we take as state variables the magnetic fluxes in the inductors, $\{\phi_k(t)\}_{k=1}^N$,  and the charges on the capacitors, $\{Q_k\}_{k=1}^N$, with labels referring to the Stages $k=1,\ldots, N$, respectively, where they occur.   The energy stored in the Stage $k$ (linear) inductor may be expressed in terms of its magnetic flux as $\frac{1}{2\,L_0}\phi_k^2$.   To determine the energy stored in the nonlinear capacitors, note first that the charge on a capacitor may be expressed as a function of the voltage, $V$, held across the capacitor:  
$$
Q_k(V)=\int_0^{V} C(v)\,dv = C_0\,V_0\, \log\left(1+\frac{V}{V_0}  \right), 
$$
which may be inverted to find 
$$
V_k(Q_k)= V_0\, \left[\exp\left(\frac{Q_k}{C_0V_0}\right)-1\right].
$$
The energy stored in the capacitor at Stage $k$ of the circuit, is then given by 
$$
\int_0^{Q_k} V_k(q)\,dq=C_0V_0^2 \left[\exp\left(\frac{Q_k}{C_0V_0}\right)-1\right]-Q_kV_0,
$$
and the total energy stored in Stage $k$ is then 
$$H^{[k]}(\phi_k,Q_k)=C_0V_0^2 \left[\exp\left(\frac{Q_k}{C_0V_0}\right)-1\right]-Q_kV_0+ \frac{1}{2\,L_0}\phi_k^2.$$  
The Hamiltonian for this system is 
$$H(Q_1,\ldots,Q_N,\phi_1,\ldots,\phi_N)=\sum_{k=1}^N H^{[k]}(\phi_k,Q_k).$$
We order the state variables so that $\bfx=[Q_1,\ldots,Q_N,\phi_1,\ldots,\phi_N]^T$ 
and therefore 
$\displaystyle \bfJ=\left[\begin{array}{cc} \mathbf{0} & \mathsf{S} \\ -\mathsf{S}^T & \mathbf{0} \end{array}\right]$ where 
$\mathsf{S}$ is an upper bidiagonal matrix with $1$ on the diagonal and $-1$  on the superdiagonal; 
$\displaystyle \bfR=\left[\begin{array}{cc} G_0\mathbf{I} & \mathsf{0} \\ \mathsf{0} & R_0\mathbf{I} \end{array}\right]$; and 
$\displaystyle \bfB=\left[\bfe_{N+1},\bfe_{N}\right]$ where $\bfe_k$ denotes the $k^{th}$ column of the identity.
Consider the particular case of a 50-stage ($N= 50$)
circuit with  parameters:
$
L_0=2\mu\mathsf{H}, \quad
V_0=1\mathsf{V}, \quad
R_0=1\Omega, \quad
G_0=10\mu\mho. 
$

We applied a voltage pulse to the left port of the network (Gaussian pulse windowed to 3$\mu$sec with a magnitude of 3V, $\sigma$ of $0.5$)  and observed the output voltage at the right port. 
The output is displayed as a solid green trace in Figure~\ref{fig:out_lin_pod_h2}.  
The induced response of the \emph{linearized} full-order network is also displayed  (green dashed line) for comparison.  Notice that nonlinearity sharpens the peak of the response and significantly reduces dispersion. 
The \textsc{pod} basis sets are generated from 
uniformly sampled snapshots $\bx(t)$ and  $\nabla_\bx H(\bx(t))$
of this full-order system with Gaussian impulse training input. 
These basis sets are then used to construct the \podph~reduced system. In the cases of 
\HtwoEpsph
 bases, the procedure described in Algorithm~\ref{alg:qh2ph} is used.  We also use Algorithm \ref{alg:podh2ph} to generate the \podHtwoEpsph
bases. All three reduced models are then simulated for the same Gaussian impulse training input, which was used for generating the \textsc{pod} snapshot  as well as a different one, a  sinusoidal input.
First we investigate the accuracy of \podph~and \HtwoEpsph~for  $r=6$. 
 Figure~\ref{fig:out_lin_pod_h2} illustrates the two structure-preserving nonlinear reduced models  capture the output of the original nonlinear system very accurately for both types of excitations. 

\begin{figure}[h!]
\caption{\scriptsize 
Ladder network: Time responses of the reduced-order systems from Gaussian pulse (left) and from sinusoidal input(right).
  }
  \centerline{
 \includegraphics[scale=0.32]{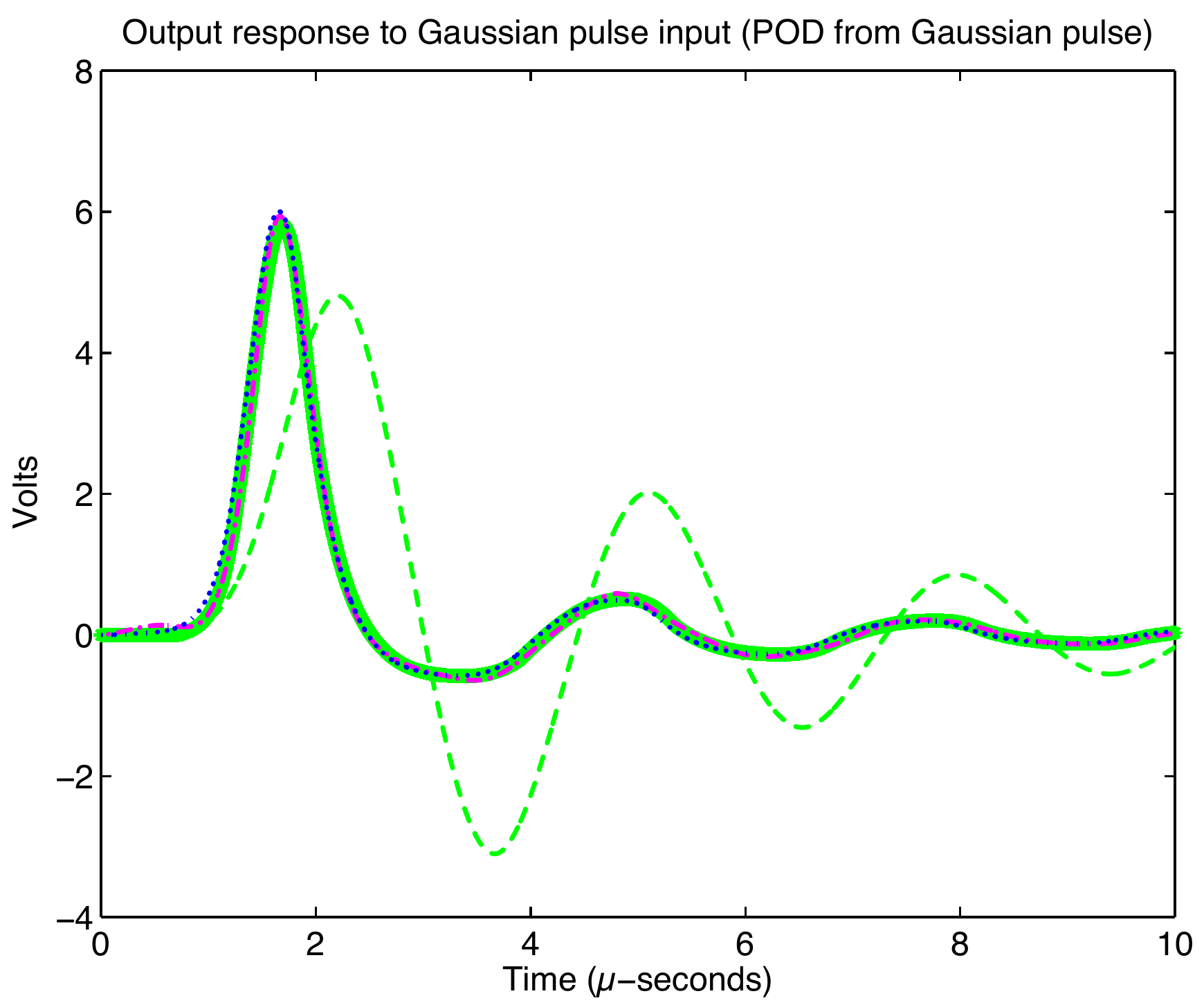}
  \includegraphics[scale=0.321]{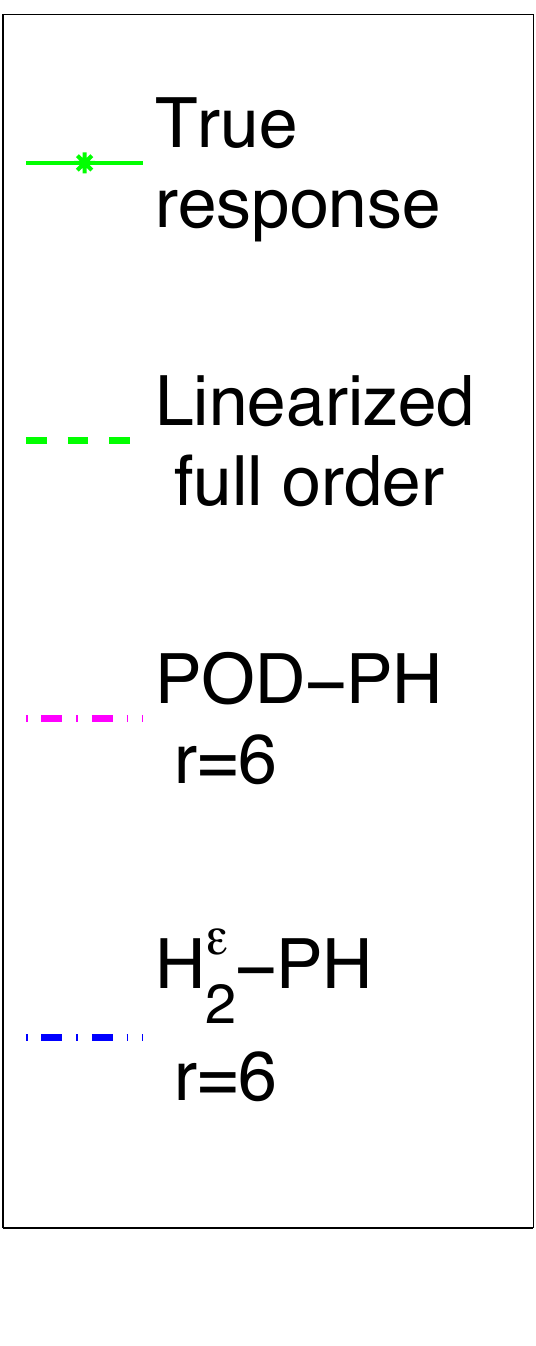}
 \includegraphics[scale=0.32]{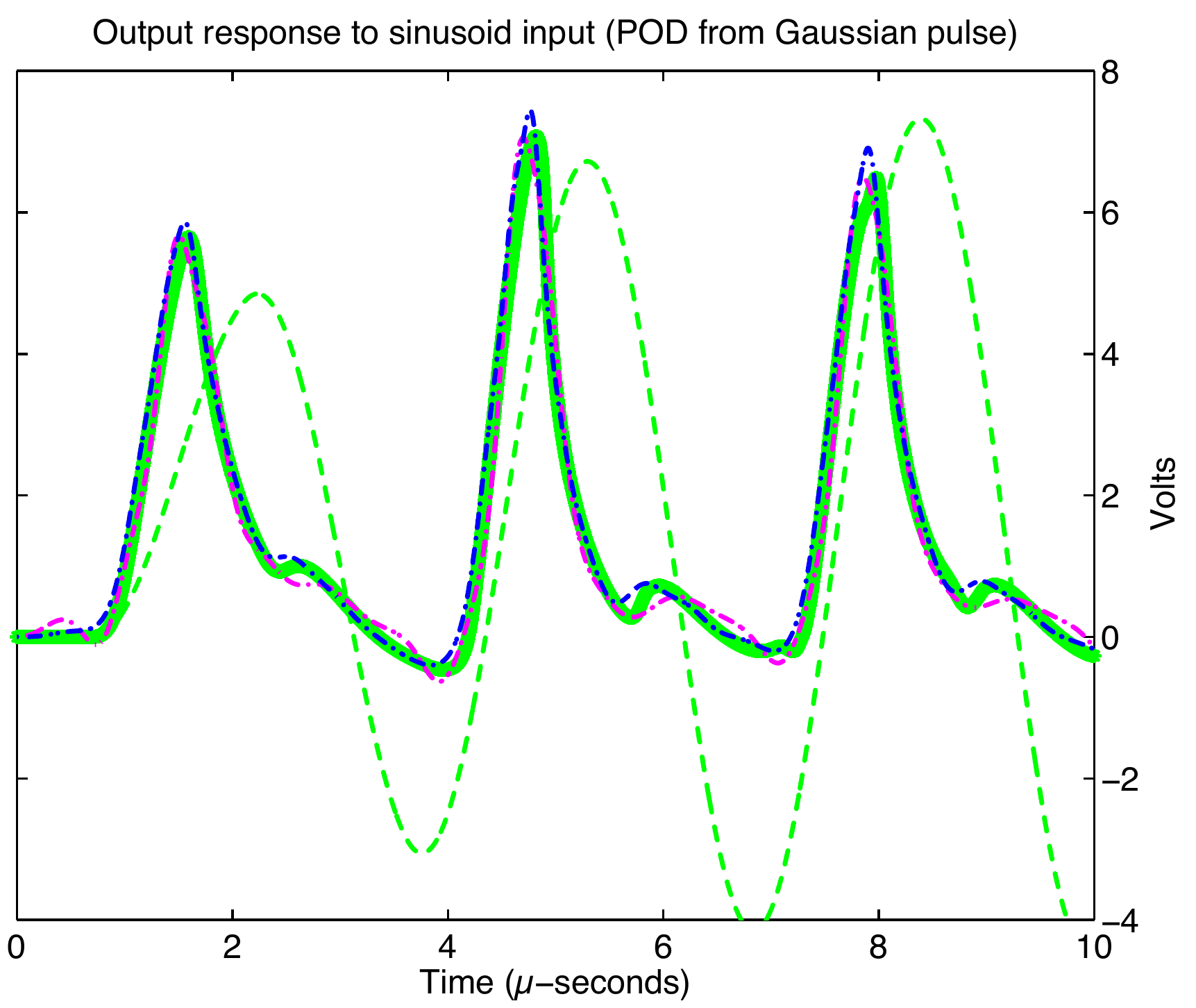}\hspace{-2ex}
 }
\label{fig:out_lin_pod_h2}
\end{figure} 
\begin{figure}[ht!]
\caption{
 \scriptsize 
  Ladder Network:
  Average relative errors of outputs and state variables of structure-preserving reduced systems (\ref{redPHdef})
using hybrid bases with different numbers of POD and \HtwoEps vectors. 
  } 
  \centerline{
 \includegraphics[scale=0.34]{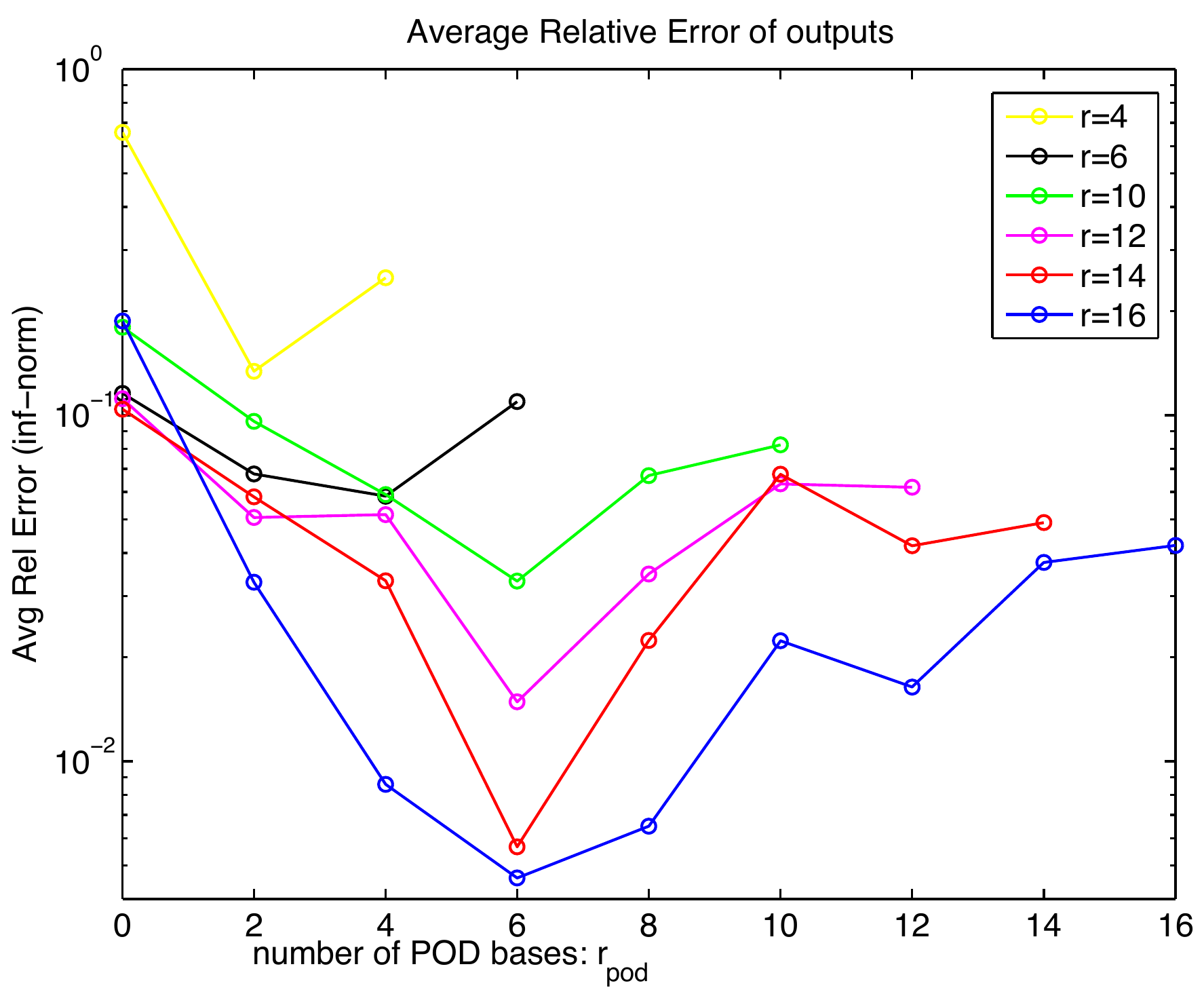}\hspace{1ex}
 \includegraphics[scale=0.34]{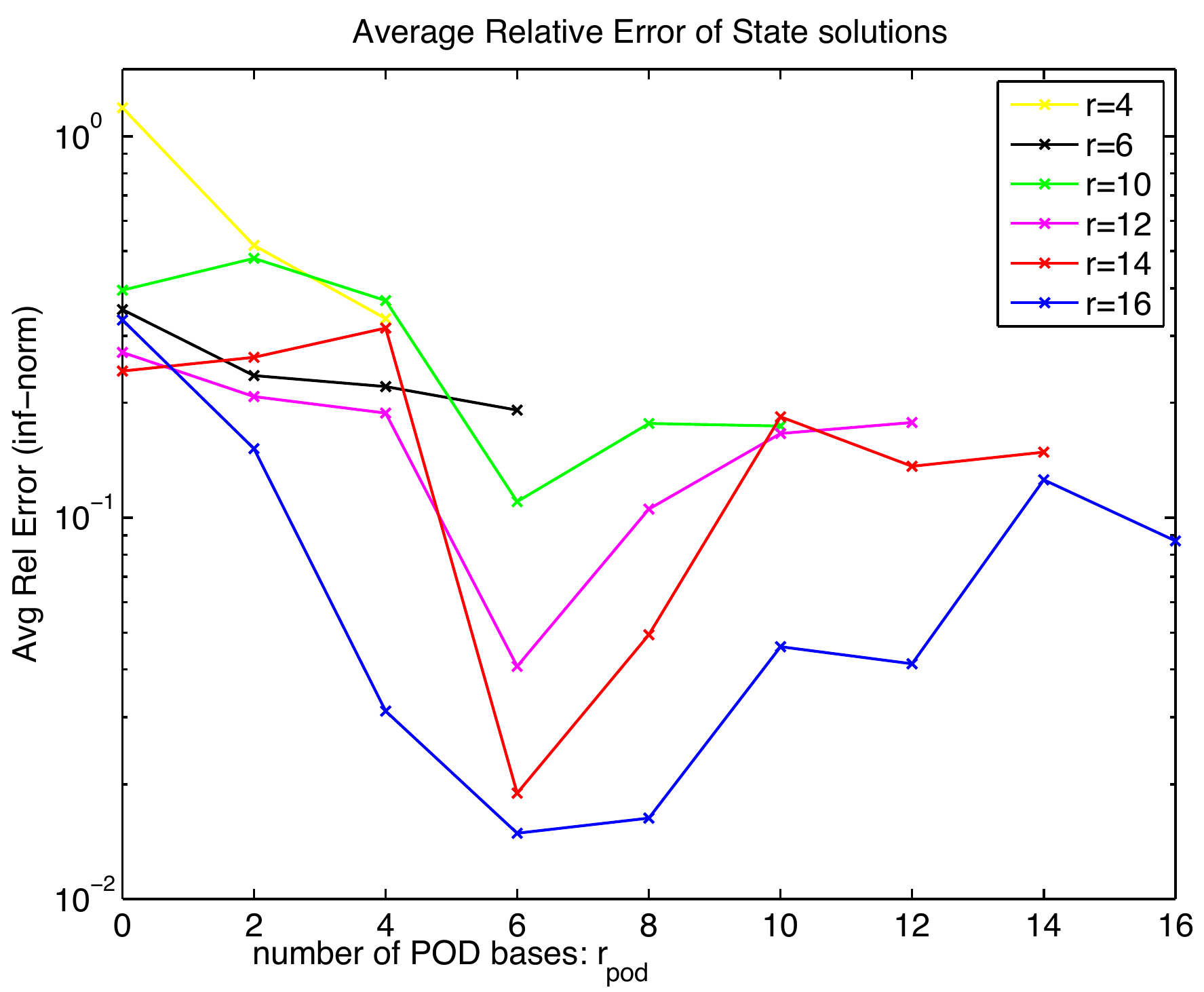}
 }
\label{fig:err_combinePODH2}
\end{figure}
\begin{figure}[h!]
  \caption{
  \scriptsize 
  Ladder Network: 
  Average relative errors of outputs and state variables of reduced 
  systems using  bases from
  (i) \HtwoEps, (ii)  POD  (iii) combination of
  POD and \HtwoEps\ (with same number of basis vectors for both POD and \HtwoEps). 
  The  \podph\ reduced systems are constructed using
   Gaussian impulse training input.
  The top and the bottom plots, respectively, use the  impulse input (same as the training input for POD) and with sinusoidal input (different from the training input for POD). 
   }
  \centerline{
 \includegraphics[scale=0.34]{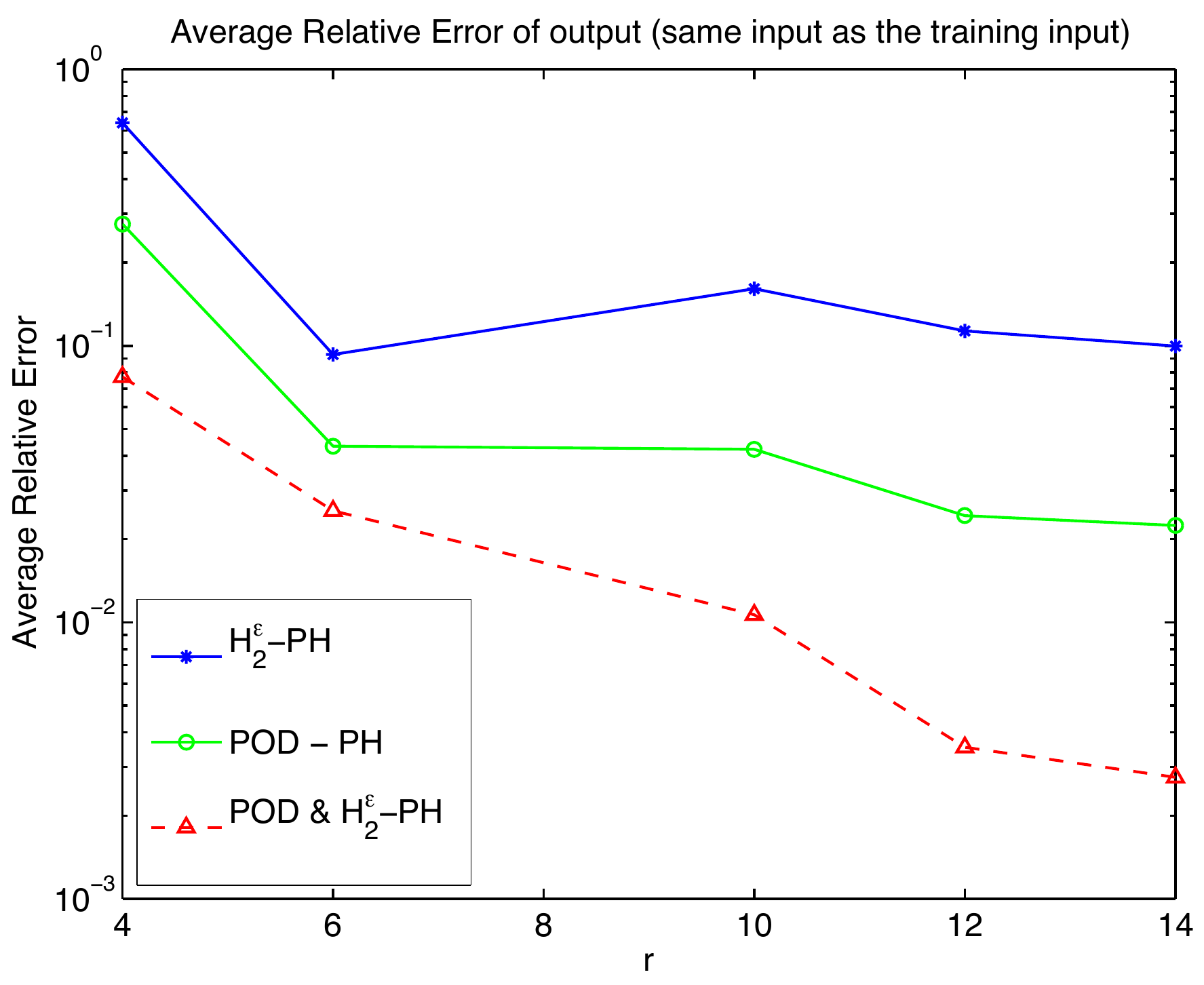}\hspace{1ex}
 \includegraphics[scale=0.34]{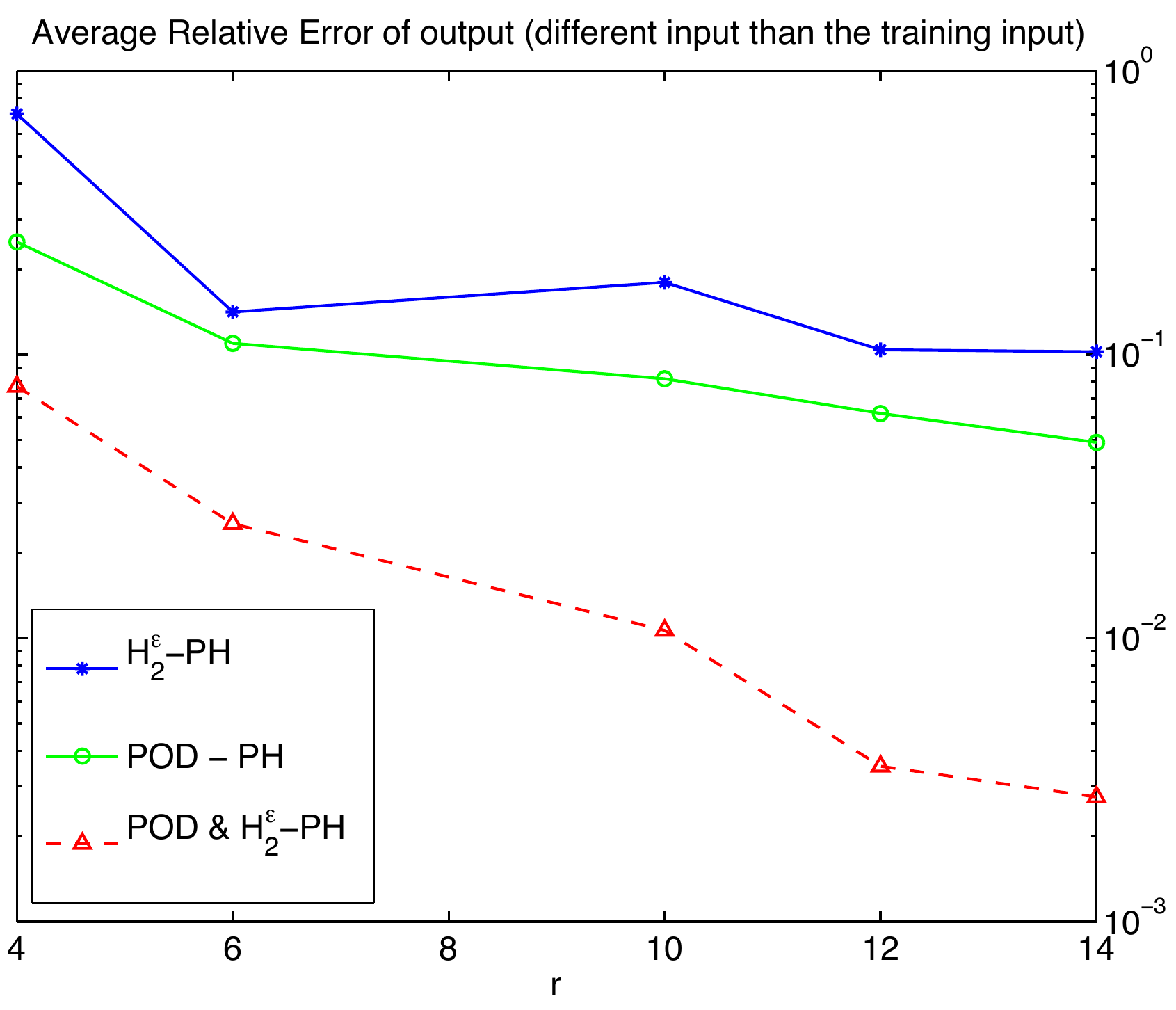}
 }\centerline{
 \includegraphics[scale=0.34]{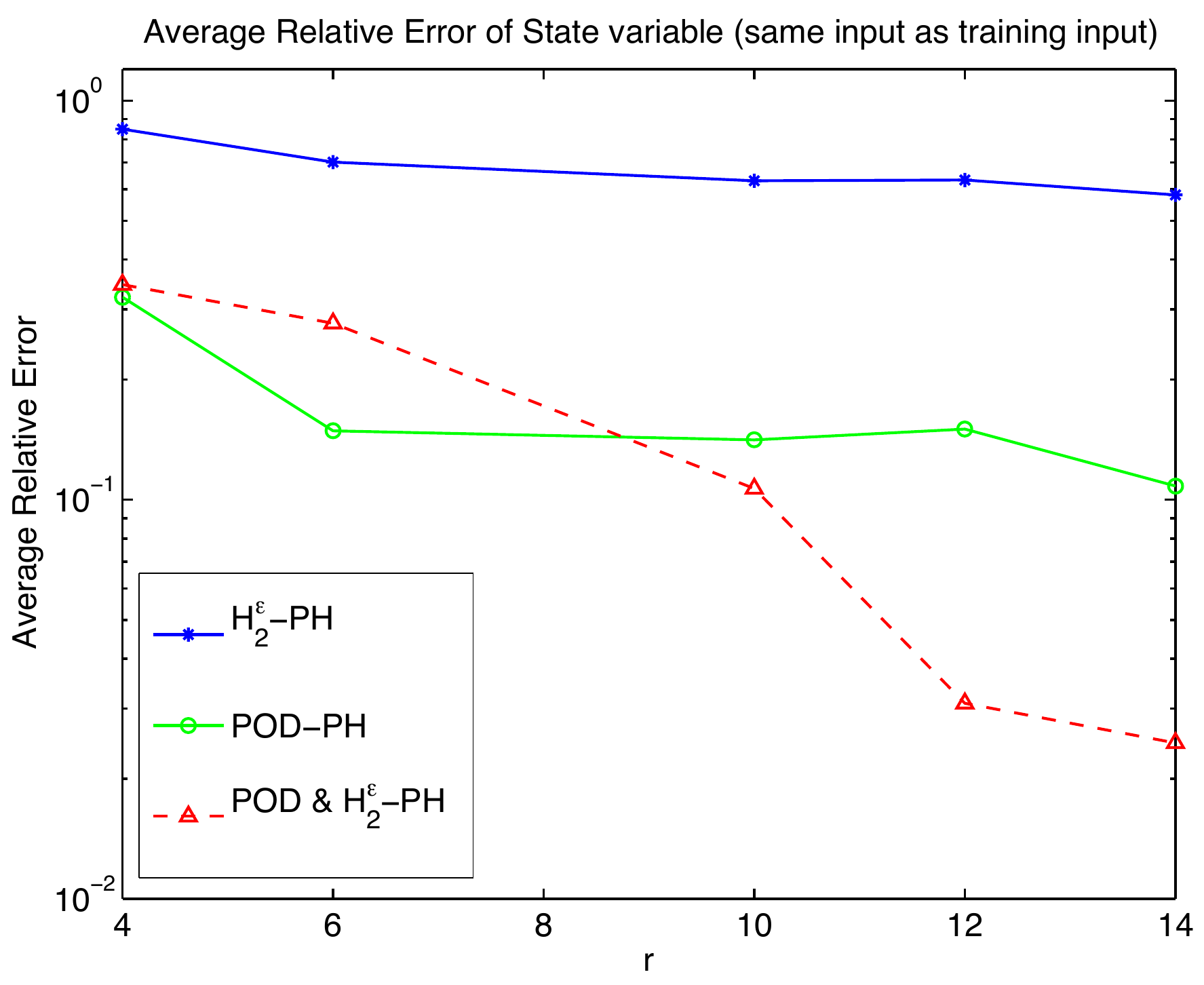}\hspace{1ex}
 \includegraphics[scale=0.34]{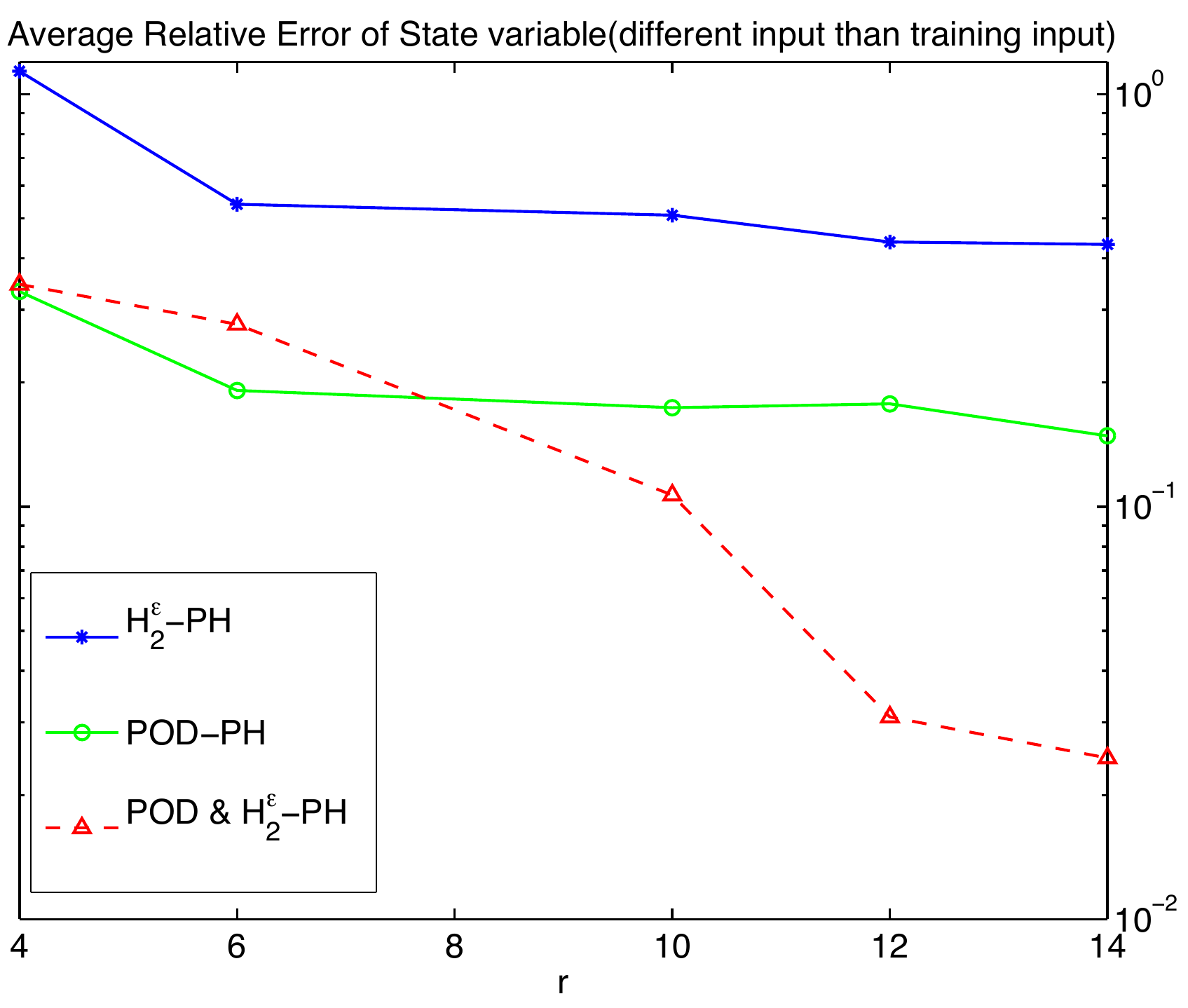}
 }
 \label{fig:errPODnH2ratio1}
\end{figure}
Next, we consider the effect on the accuracy of the state space solutions of using different proportions of \podph~and \HtwoEpsph~basis vectors in the hybrid \podHtwoEpsph~basis.  Results for a wide range of 
$r$ values are  illustrated in Figure~\ref{fig:err_combinePODH2}, depicting the relative error  in the output. In these figures, the colors correspond to a fixed order $r$. The $x$-axis is the number of \podph~basis vectors for that given order, so e.g., the blue line corresponds to relative error for $r=16$. For that line, the value corresponding  to $x=12$ means that for the  $r=16$ model, $\hat{r} =12$ \podph~basis vectors are combined with $\bar{r} = r-\hat{r} = 4$ \HtwoEpsph~basis vectors.
The most accurate approximation resulted from reduced systems
with bases that combined roughly equal numbers of \podph~and \HtwoEpsph~basis vectors.
 Figure~\ref{fig:errPODnH2ratio1} shows that reduced systems constructed from these hybrid bases (using equal numbers of  \podph~and \HtwoEpsph~vectors) can give much more accurate approximations than 
those constructed with  \podph~alone or   
 \HtwoEpsph~alone. 
For $r=12$, say, a reduced system with a hybrid basis using $\hat{r} = \bar{r} = 6$   (red dashed line) 
produced roughly $10$ times smaller error 
for both state variables and outputs 
than those using only \podph~(green solid lines) or \HtwoEpsph~bases  (blue solid lines). 
%
\subsection{An \emph{a priori} error bound for NLPH-reduced models}
\label{sec:MOR_Proj_NPHerror}

Fix a reduction order, $r$, and let $\cV_r = \mathsf{Ran}(\bV_r)$ and  $\cW_r = \mathsf{Ran}(\bW_r)$ denote $r$-dimensional reduction subspaces used in creating reduced \textsc{nlph} systems as in (\ref{redPHdef}).     
We provide an error analysis here that bounds the deviation between the true state trajectory of an \textsc{nlph} system and that provided by a reduced \textsc{nlph} system of the form given  in  (\ref{redPHdef}).  This leads in turn to a  bound on the error between the true system output and the reduced \textsc{nlph} system output.  
Typically, the reduction subspaces $\cV_r$ and $\cW_r$ will be chosen consistently with the heuristics of (\ref{GoodSpace}) but the bounds we derive apply more generally. 

Let $\bfQ\in\real^{n\times n}$ be a symmetric positive-definite matrix and define a weighted inner product on $\real^n$ as
$\langle \bx , \bz \rangle_{\bfQ} = \bx^T \bfQ \bz$, with a related norm, $\|\bx\|_ {\bfQ} =\sqrt{\langle \bx , \bx \rangle_{\bfQ}}$.
We leave the choice of $\bfQ$ open for the time being, however the choices $\bfQ=\bfI$ and $\bfQ=\nabla^2 H (\bx_0)$ at a locally stable equilibrium point $\bx_0$ will have particular merit. 

For a mapping $\bF: \real^{n} \to \real^{n}$, we define the associated \emph{Lipschitz constant} and \emph{logarithmic Lipschitz constant}
of $\bF$ relative to $\bfQ$ (see \cite{Soderlind2006} ) as
\begin{equation} \label{def:LipClogLipC}
L_{\bfQ} [\bF] =  \sup_{\bu \ne \bv} \frac{\| \bF(\bu) - \bF(\bv)\|_{\bfQ} }{ \| \bu - \bv \|_{\bfQ} },
\quad\mbox{and}\quad
\cL_{\bfQ} [\bF] =  \sup_{\bu \ne \bv} \frac{\langle \bu - \bv, \bF(\bu) - \bF(\bv)\rangle_{\bfQ} }{ \| \bu - \bv \|^2_{\bfQ} }
\end{equation}
respectively.   Note that $\cL_{\bfQ} [\bF]$ could be negative and $-L_{\bfQ} [\bF]\leq \cL_{\bfQ} [\bF] \leq L_{\bfQ} [\bF]$.

 Suppose $\bfQ$-orthogonal bases for $\cV_r$ and $\cW_r$ are chosen:  
 $\widetilde{\bV}_r, \widetilde{\bW}_r \in \real^{n \times r}$ such that
 $\cV_r = \mathsf{Ran}(\widetilde{\bV}_r)$ and $\cW_r  = \mathsf{Ran}(\widetilde{\bW}_r)$, 
 with
 $\widetilde{\bV}_r^T{\bfQ}\widetilde{\bV}_r = \bI$ and
 $\widetilde{\bW}_r^T{\bfQ}\widetilde{\bW}_r = \bI$.
Consider $\bfQ$-orthogonal projectors 
$\bPi_\cV: \real^{n} \to\cV_r$, 
$\bPi_\cW: \real^{n} \to \cW_r$ 
defined by 
$\bPi_\cV=  \widetilde{\bV}_r \widetilde{\bV}_r^T{\bfQ}$ and
$\bPi_\cW= \widetilde{\bW}_r \widetilde{\bW}_r^T{\bfQ}$. 
Define an ancillary state space projection as $\cbfP_r= \bV_r \bW_r^T$.
For a given (true) system trajectory, $\bx(t)$, the best approximation (relative to the $\bfQ$-norm) that is available by a path in 
$\cV_r$ is given by $\bPi_\cV \bx(t)$, and so the state space error $\bx(t) - \bV_r\bx_r(t)$ is bounded pointwise below as
$$
\|(\bI -\bPi_\cV) \bx(t)\|_{\bfQ}\leq \|\bx(t) - \bV_r \bx_r(t)\|_{\bfQ} 
$$
for all $t\geq 0$.  The error associated with the reduced internal force, $\nabla_{\bx_r} H_r$, is bounded similarly
$$
\|(\bI -\bPi_\cW) \nabla_\bx H(\bx(t)) \|_{\bfQ}\leq \|\nabla_\bx H(\bx(t)) - \bW_r \nabla_{\bx_r} H_r(\bx_r(t))\|_{\bfQ}.
$$
Let the optimal state-space and internal force residual vectors be defined as 
$$
\bfeps_\bx(t) = (\bI -\bPi_\cV) \bx(t)\quad\mbox{and}\quad \bfeps_\bF(t) =(\bI -\bPi_\cW) \nabla_\bx H(\bx(t)).
$$ 
The squared residual state-space and internal force errors integrated over $[0, T]$ will be denoted as  
 $$
\cE_\bx = \int_{0}^{T} \|\bfeps_\bx(t)\|^2_{\bfQ}\, dt = \int_{0}^{T} \|(\bI -\bPi_\cV) \bx(t)\|^2_{\bfQ}\, dt
$$ 
and 
$$
\cE_\bF = \int_{0}^{T} \|   \bfeps_\bF(t) \|^2_{\bfQ}\, dt = \int_{0}^{T} \|  (\bI - \bPi_\cW) \nabla_\bx H(\bx(t)) \|^2_{\bfQ}\, dt.
$$

We seek to bound the state space error, $\bx - \bV_r \bx_r$, and output error, $\by-\by_r$, in terms of  
$\cE_\bx $, $\cE_\bF$, and the deviation between the projected and reduced initial condition.  
Note that when the reduction spaces, $\cV_r$ and $\cW_r$, are generated via a \textsc{pod} approach 
(e.g., Algorithm \ref{alg:podph}) 
 then the  aggregate errors $\cE_\bx $, $\cE_\bF$ are approximately minimized and
can be expressed directly as the sum of neglected singular values of the
snapshot matrix (either of  $\bx(t)$ or of $\nabla_\bx H(\bx(t))$.  
As we have seen in \S \ref{sec:MOR_PH_exLadder}, the actual state space error or output error 
might not be minimized with this choice. 
\begin{theorem}
\label{thm:errBd_MORstr}
Suppose $\bfQ\in \real^{n \times n}$ is symmetric positive definite and that a reduced port-Hamiltonian system as in (\ref{redPHdef})
is constructed to approximate the full order \textsc{nlph} system (\ref{PHdef}) 
using reduction bases $\bV_r$, $\bW_r \in \real^{n \times r}$ that are defined so that 
$\bV_r^T{\bfQ}\bV_r = \bI_r$ and $\bV_r^T\bW_r = \bI_r$. (Note that it may or may not be the case that $\bW_r={\bfQ}\bV_r$.)
Suppose further that  $\bfF(\bx)=\nabla_\bx H(\bx)$  in (\ref{PHdef}) is Lipschitz continuous.   
Denote $\cbfA=\bJ-\bR$ and $\cbfP_r=\bV_r\bW_r^T $.  Then
\begin{eqnarray}
\int_{0}^{T}\left\|\bx(t) - \bV_r\bx_r(t)\right\|^2_{\bfQ} dt
& \le &  
        C_\bx \, \cE_\bx
    +   C_\bF\,  \cE_\bF
    + C_0 \mbox{\small $\| \bW_r^T\bx(0)-\bx_r(0)\|^2$ }
\label{eqn:errBd_x1}
\\
  \int_{0}^{T} \left\| \by(t) - \by_r(t) \right\|^2 dt
    & \le &
     \widehat{C}_\bx  \cE_\bx
    +  \widehat{C}_\bF \cE_\bF
        + \widehat{C}_0 \mbox{\small $\| \bW_r^T\bx(0)-\bx_r(0)\|^2$ }
\label{eqn:errBd_y1}
\end{eqnarray}
where    
$
\begin{array}{ccc}
\alpha  = \cL_{\bfQ}[\cbfP_r\cbfA \cbfP_r^T\bF], &
\beta=   \|\cbfP_r \cbfA\|_{\bfQ}\|\cbfP_r^T\|_{\bfQ}, &
 \gamma =  L_{\bfQ} [\bF]\  \|\cbfP_r\|_{\bfQ},
 \end{array}
 $
 $$
\begin{array}{ccc}
\delta = 2 \|\bB^T{\bfQ}^{-1}\bB\| \, \|\cbfP_r^T\|_{\bfQ}^2,  &
c_\alpha(t) =\int_0^t e^{2\alpha \tau}\, d\tau, & 
     C_\alpha(t) = \int_0^t c_\alpha(\tau)\, d\tau,
             \end{array}  
 $$
{\small
$$
\begin{array}{ccc}   
C_\bx =  (2\beta\gamma)^2\,C_\alpha(T)  + 2\left\|\cbfP_r \right\|^2_{\bfQ},  &
 C_\bF =  (2\beta)^2\,C_\alpha(T), &
 C_{0} = 2\,c_\alpha(T), \mbox{ and} 
 \end{array}
 $$
\[
\begin{array}{ccc}
 \widehat{C}_\bx  = \delta \cdot L_{\bfQ} [\bF]^2 C_\bx,  &
\widehat{C}_\bF  =   \delta  \cdot (1+ L_{\bfQ} [\bF]^2 C_\bF), &
 \widehat{C}_{0} =  \delta \cdot L_{\bfQ} [\bF]^2 C_0 . 
\end{array}
\] 
}
\end{theorem}

\begin{proof}
First note that  
\begin{eqnarray}  \label{note:LS2err_x}   
 (\bI - \cbfP_r) \bx(t) 
 & =  &  (\bI - \cbfP_r) (\bI -\bPi_\cV) \bx(t)  
   = (\bI - \cbfP_r) \bfeps_\bx(t) \quad \mbox{and}
   \\[2mm]
 (\bI - \cbfP_r^T)\bF(t) 
 & = &  (\bI - \ \cbfP_r^T) (\bI - \bPi_\cW)\bF(t)
   =  (\bI - \cbfP_r^T)\bfeps_\bF(t).
\label{note:LS2err_f}
\end{eqnarray}

The state space error can be separated into the sum of a component in $\mathsf{Ker}(\cbfP_r)$ (i.e.,
orthogonal to $\cW_r$) and a component contained in $\mathsf{Ran}(\cbfP_r)=\cV_r$:
\[
\bx(t) - \bV_r \bx_r(t) = \bfrho(t) + \bV_r\btheta(t),
\]
where
$ \bfrho(t)=  \bx(t) -  \bV_r\bW_r^T\bx(t)= (\bI - \cbfP_r)\bfeps_\bx(t)\in\cW_r^\perp$ 
and
$\btheta(t) = \bW_r^T\bx(t)-\bx_r(t)$.
Then 
\begin{equation} \label{baseEstimate}
\int_0^T \|\bx(t) - \bV_r \bx_r(t)\|_{\bfQ}^2\, dt \leq 
2\int_0^T \| \bfrho(t)\|_{\bfQ}^2\, dt + 2\int_0^T \|\btheta(t)\|^2\, dt.
\end{equation}

For the first term, we may estimate immediately
$$
\int_0^T \| \bfrho(t)\|_{\bfQ}^2\, dt  \leq 
\int_0^T \| (\bI - \cbfP_r)\bfeps_\bx(t) \|_{\bfQ}^2\, dt \leq \|\cbfP_r \|_{\bfQ}^2 \ \cE_\bx
$$
where we have made use of the identity, $ \| \bI - \cbfP_r \|_{\bfQ} = \| \cbfP_r \|_{\bfQ}$ (see \cite{Szyld.06}).

To bound the second term of (\ref{baseEstimate}), note that
\begin{eqnarray*}
  \dot{\btheta}(t)
  & = & \bW_r^T\dot{\bx}(t) - \dot{\bx}_r(t)  
  =  \bA_r \bV_r^T\left[\bF(\bV_r\bW_r^T\bx(t)) -  \bF(\bV_r\bx_r(t))\right]  + \bfeta(t),
\end{eqnarray*}
where
$\bA_r =\bJ_r -\bR_r$ (see (\ref{redJQR})) and
{\small  $\bfeta(t) =\bW_r^T(\bJ -\bR ) \left[ \bF(\bx(t)) - \bW_r\bV_r^T\bF(\bV_r\bW_r^T\bx(t))\right] $. }
Note that  $\|\btheta(t)\|\,\frac{d}{dt} \|\btheta(t)\| =\frac{1}{2} \frac{d}{dt} \|\btheta(t)\|^2   =  \left\langle \btheta(t), \dot{\btheta}(t) \right\rangle$,
so we have
{\small 
\begin{equation*}
   \frac{d}{dt} \|\btheta(t)\|    =  \left\langle \frac{\btheta(t)}{\|\btheta(t)\|}, \dot{\btheta}(t) \right\rangle
    =  \left\langle \frac{\btheta(t)}{\|\btheta(t)\|}, \bA_r \bV_r^T\left[\bF(\bV_r\bW_r^T\bx(t)) -  \bF(\bV_r\bx_r(t))\right]  + \bfeta(t) \right\rangle.
  \end{equation*}
  }
  Observe that 
 {\small 
\begin{align*}
   |\left\langle \btheta(t), \bfeta(t) \right\rangle| & =  
   \left|\left\langle \bV_r\btheta(t), \bV_r \bW_r^T(\bJ -\bR ) \left[ \bF(\bx(t)) - \bW_r\bV_r^T\bF(\bV_r\bW_r^T\bx(t))\right] \right\rangle_{\bfQ}\right|\\
   & \leq
   \left\|\bV_r\btheta(t) \right\|_{\bfQ} \ \|\cbfP_r \cbfA  \left[ \bF(\bx(t)) -\cbfP_r^T\bF(\cbfP_r\bx(t))\right]\|_{\bfQ} \\
   & \leq
   \left\|\btheta(t) \right\|  \|\cbfP_r\cbfA\|_{\bfQ}
   \|(\bfI-\cbfP_r^T)\bF(\bx(t)) + \cbfP_r^T(\bF(\bx(t))-\bF(\cbfP_r\bx(t)) \|_{\bfQ} \\
   & \leq
   \left\|\btheta(t) \right\| \|\cbfP_r\cbfA\|_{\bfQ}  \cdot \left(
   \|(\bI - \cbfP_r^T)\bfeps_\bF(t)\|_{\bfQ} + \|\cbfP_r^T\|_{\bfQ} \|\bF(\bx(t))-\bF(\cbfP_r\bx(t)) \|_{\bfQ} \right) \\
   & \leq
   \left\|\btheta(t) \right\|  \|\cbfP_r\cbfA\|_{\bfQ}  \cdot \left(
   \|\bI - \cbfP_r^T\|_{\bfQ}\|\bfeps_\bF(t)\|_{\bfQ} + \|\cbfP_r^T\|_{\bfQ} \ 
   L_{\bfQ} [\bF] \|(\bfI-\cbfP_r)\bx(t) \|_{\bfQ} \right)   \\
   & \leq
   \left\|\btheta(t) \right\| \left( \|\cbfP_r \cbfA\|_{\bfQ}\|\cbfP_r^T\|_{\bfQ}\right) 
    \cdot \left(\|\bfeps_\bF(t)\|_{\bfQ} +  L_{\bfQ} [\bF]\  \|\cbfP_r\|_{\bfQ} \|\bfeps_\bx(t)\|_{\bfQ} \right)  \\
    & \leq
   \left\|\btheta(t) \right\| \, \cdot\,   \beta \left(\|\bfeps_\bF(t)\|_{\bfQ} +  \gamma \|\bfeps_\bx(t)\|_{\bfQ} \right)  
  \end{align*}
  }
  and
{\small 
\begin{align*}
&\left\langle \btheta(t), \bA_r \bV_r^T\left[\bF(\bV_r\bW_r^T\bx(t)) -  \bF(\bV_r\bx_r(t))\right] \right\rangle  \\
& \qquad \qquad =
\left\langle \bV_r\btheta(t), \cbfP_r\cbfA \cbfP_r^T\left[\bF(\bV_r\bW_r^T\bx(t)) -  \bF(\bV_r\bx_r(t))\right] \right\rangle_{\bfQ} \\
& \qquad \qquad \qquad \leq  \cL_{\bfQ} [\cbfP_r\cbfA \cbfP_r^T\bF]\,\cdot\, \left\| \bV_r\btheta(t) \right\|_{\bfQ}^2 
=  \alpha\, \left\| \btheta(t) \right\|^2,
  \end{align*}
  }
  where we make use of the fact that 
  {\small $\|\bfI-\cbfP_r\|_{\bfQ}=\|\cbfP_r\|_{\bfQ}$} and {\small $\|\bfI-\cbfP_r^T\|_{\bfQ}=\|\cbfP_r^T\|_{\bfQ}$}.

For (\ref{eqn:errBd_y1}), we find 
\begin{align*}
    \|\by - \by_r\| 
 & =  \|\bB^T\bF(\bx) - \bB^T \bW_r \bV_r^T\bF(\bV_r\bx_r)\|\\
 & =  \|\bB^T{\bfQ}^{-1/2}{\bfQ}^{1/2}\left(\bF(\bx) -  \cbfP_r^T\bF(\bV_r\bx_r)\right)\|\\
 & \leq \|\bB^T{\bfQ}^{-1/2}\| \, \|(\bI -\cbfP_r^T)\bfeps_\bF + \cbfP_r^T\left(\bF(\bx) - \bF(\bV_r\bx_r)\right)\|_{\bfQ}\\
 & \leq  \|\bB^T{\bfQ}^{-1}\bB\|^{1/2} \left(\|(\bI -\cbfP_r^T)\bfeps_\bF\|_{\bfQ} 
 + \|\cbfP_r^T\left(\bF(\bx) - \bF(\bV_r\bx_r)\right)\|_{\bfQ}\right)\\
 & \leq  \|\bB^T{\bfQ}^{-1}\bB\|^{1/2} \, \|\cbfP_r^T\|_{\bfQ}
  \left(\|\bfeps_\bF\|_{\bfQ} + \|\left(\bF(\bx) - \bF(\bV_r\bx_r)\right)\|_{\bfQ}\right)\\
  & \leq  \|\bB^T{\bfQ}^{-1}\bB\|^{1/2} \, \|\cbfP_r^T\|_{\bfQ}
  \left(\|\bfeps_\bF\|_{\bfQ} + L_{\bfQ} [\bF] \|\bx - \bV_r\bx_r\|_{\bfQ}\right)
 \end{align*} 
 Thus,
 {\small
\begin{align*}
 \int_0^T  \|\by(t) - \by_r(t)\|^2\, dt &\leq 
2 \|\bB^T{\bfQ}^{-1}\bB\| \, \|\cbfP_r^T\|_{\bfQ}^2
  \int_0^T  \left(\|\bfeps_\bF(t)\|_{\bfQ}^2 + L_{\bfQ} [\bF]^2 \|\bx - \bV_r\bx_r\|_{\bfQ}^2\right)\, dt\\
  &\leq 
2 \|\bB^T{\bfQ}^{-1}\bB\| \, \|\cbfP_r^T\|_{\bfQ}^2 \ \cdot \\
    & \qquad \qquad  \left(\cE_\bF + L_{\bfQ} [\bF]^2 \left(C_\bx \, \cE_\bx
    +   C_\bF\,  \cE_\bF + C_0 \mbox{\small $\| \bW_r^T\bx(0)-\bx_r(0)\|^2$ }\right)\right)
\end{align*} 
}
and (\ref{eqn:errBd_y1}) follows. 
\end{proof}

 
\section{Structure-preserving model reduction with DEIM}
\label{sec:MOR_DEIMstr}

%
%
The performance of projection-based reduels of nonlinear systems can be degraded by the 
need to lift the reduced state to the full state dimension in order to evaluate the nonlinear term.  
For example,  consider the evaluation of 
$\bfV_r^T\nabla_{\!\bfx}{H}(\bfV_r\bfx_r)$  in (\ref{reduced_f}).  
If $\nabla_{\!\bfx}{H}(\bfx)$ is nonlinear in $\bfx$, it is likely that $\bfV_r^T\nabla_{\!\bfx}{H}(\bfV_r\bfx_r)$
cannot be precomputed explicitly as a map from $\mathbb{R}^r$ to $\mathbb{R}^r$ without an intermediate lifting to 
$\mathbb{R}^n$. The order of complexity required to evolve the reduced system will then remain at least $\mathcal{O}(n)$ 
and there may be little benefit seen in the use of a reduced model. 
Several approaches have been proposed to address this difficulty; see, e.g.,  \cite{Astrid2008, Barrault2004,Carlberg2013,ChatSorDEIM_siam2010, Everson1995}. 
We resolve the complexity issue by developing a variant of the \emph{discrete empirical interpolation method} (\textsc{deim}) \cite{ChatSorDEIM_siam2010}, which is itself a discrete variant of the Empirical Interpolation Method introduced in \cite{Barrault2004}. 
 Since \textsc{deim} (as presented in \cite{ChatSorDEIM_siam2010}) does not typically preserve port-Hamiltonian structure,  
 we develop here a structure-preserving variant of \textsc{deim} together with associated error estimates.   
  There have been other structure-preserving methods developed recently that employ methods resolving the lifting bottleneck, e.g.,
 see \cite{carlberg2015} for an approach that preserves Lagrangian structure in structural dynamics using gappy POD.
\subsection{The Discrete Empirical Interpolation Method}
\label{sec:DEIM}

The `lifting bottleneck' in the reduction of large scale nonlinear models as described above is resolved by \textsc{deim} 
through the approximation of the nonlinear system function via interpolation. This approximation is done in such a way as to
not require a prolongation of the reduced state variables (`lifting') back
to the original high dimensional state space.  
Only a few selected entries of the original nonlinear term need be evaluated at each time step.

In particular, let $\mathbf{f}:\mathscr{D} \mapsto \mathbb{R}^{n}$ be a nonlinear vector-valued function
defined on a domain $\mathscr{D} \subseteq \mathbb{R}^{n}$.
Let $\mathbf{U}_m = [\bu_1, \dots, \bu_m] \in \real^{n \times m}$  have rank  $m$ and define
$\mathbf{E}_m = [\mathbf{e}_{\wp_1}, \dots, \mathbf{e}_{\wp_m}] \in \mathbb{R}^{n \times m}$
with the index set $\{\wp_1, \dots, \wp_{m}\}$  output from Algorithm~\ref{alg:DEIM_Ubasis} using the input
basis $\{\mathbf{u}_i\}_{i = 1}^m$.  $\be_{\wp_j} \in \real^{n}$ 
denotes the $\wp_j$th column of the $n$-by-$n$ identity matrix.

The \textsc{deim} approximation of order $m \le n$ for $\bff$ in ${\rm span}\{\mathbf{U}_m\}$ 
is given by
\begin{equation} \label{def:DEIM_approxfun}
\widehat{\mathbf{f}}(\tau)  := \mathbb{P}\, \mathbf{f}(\tau), \quad \mbox{ where }
\quad \mathbb{P}= \mathbf{U}_m(\mathbf{E}_m^T\mathbf{U}_m)^{-1}\mathbf{E}_m^T.
\end{equation}
Note that in the original work, \cite{ChatSorDEIM_siam2010}, 
$\mathbf{U}_m$ is assumed to have orthonormal columns, 
yet linear independence suffices for $\mathbb{P}$ to be well defined 
and that is all that we require here.  
%
\begin{algorithm}[h!tbp]
\caption{: \textsc{deim} \cite{ChatSorDEIM_siam2010} }
\label{alg:DEIM_Ubasis}
\vspace{2mm}

\hspace{.5mm} {\bf INPUT: $\{\mathbf{u}_\ell\}_{\ell=1}^m \subset \mathbb{R}^{n}$} linearly
independent \\
\vspace{.5mm}
{\bf OUTPUT:} $\vec{\wp} = [\wp_1, \dots, \wp_{m}]^T \in \mathbb{R}^{m}$
\vspace{1mm}
\begin{algorithmic}[1]
\STATE \label{algLine:DEIM_initial}
          $[|\rho|, \hspace{2mm} \wp_{1}] = {\tt max}\{|\mathbf{u}_1|\}$\\
          \vspace{2mm}
\STATE    $\mathbf{U} = [\mathbf{u}_1]$, $\mathbf{E} = [\mathbf{e}_{\wp_{1}}]$, $\vec{\wp}  = [\wp_1]$\\
          \vspace{2mm}
\FOR{$\ell=2$ to $m$}
\vspace{2mm}

\STATE  \label{algLine:DEIM_solve_coeff}
        Solve $(\mathbf{E}^T\mathbf{U})\mathbf{c} = \mathbf{E}^T\mathbf{u}_{\ell}$
for $\mathbf{c}$
        \vspace{2mm}

\STATE \label{algLine:DEIM_set_r}
         $\mathbf{r} = \mathbf{u}_{\ell} - \mathbf{U}\mathbf{c}$
         \vspace{2mm}

\STATE \label{algLine:DEIM_find_index}
         $[|\rho|, \hspace{2mm} \wp_{\ell}] = {\tt max}\{|\mathbf{r}|\}$ \quad  
         $\left(\begin{array}{c} \mbox{ $|\rho| = |v_{{\wp}}| =  \max_{i=1,\dots,n}\{|v_{i}|\}$, with the}\\
                                           \mbox{smallest index taken for $\wp$ in case of ties.} \end{array}\right)$

\STATE \label{algLine:DEIM_update_indexvec}
       $\mathbf{U} \leftarrow [ \mathbf{U} \hspace{2mm} \mathbf{u}_{\ell}]$,
       $\mathbf{E} \leftarrow [\mathbf{E}  \hspace{2mm} \mathbf{e}_{\wp_{\ell}}]$,
                 $\vec{\wp} \leftarrow
                \left[\begin{array}{c}
                \vec{\wp}\\
                \wp_{\ell}
                \end{array} \right]$

\ENDFOR
\end{algorithmic}
\end{algorithm}

%

We adapt an error bound for the \textsc{deim} function approximation derived 
in \cite{ChatSorDEIM_siam2010} for the case that $\bU_m$ has $\bQ$-orthonormal columns.
\begin{lemma}
\label{lemma:DEIMErrBd}
If $\bU_m^T\bQ\bU_m = \bI_m$, then
  $  \|\mathbf{f}(\tau)- \widehat{\mathbf{f}}(\tau)\|_{\bQ}  \le  \|\mathbb{P}\|_{\bQ} \, \hspace{1mm} \mathscr{E}_{\bQ}(\bff(\tau),\bU_m)$,
    where
    $\mathscr{E}_{\bQ}(\bff(\tau),\bU_m)$ is the best $\bQ$-norm approximation error for $\bff(\tau)$ from $\mathsf{Ran}(\bU_m)$. 
\end{lemma}

The invertibility of $\bfE^T\bU$ at the end of each cycle of the \textsc{deim} procedure 
is verified in \cite{ChatSorDEIM_siam2010} where
it is shown  that each \textsc{deim} interpolation index is selected in order to limit the stepwise
growth of  the factor $ \| (\bE_m^T \bU_m)^{-1} \|= \|\mathbb{P}\|_{\bQ}$ in the error bound. 
This will be used in the next section to assess the accuracy of the 
state variables in the \textsc{deim} reduced system.
\textsc{deim}  shown in Algorithm \ref{alg:DEIM_Ubasis} uses LU with partial pivoting for interpolation indices.
A new selection operator for \textsc{deim} based on the pivoted QR has been recently 
introduced by Drma\v{c} and Gugercin \cite{drmacgugercin2015}.  Our numerical results in this paper are based the original implementation in Algorithm \ref{alg:DEIM_Ubasis}. We also refer the reader to \cite{peherstorfer2014localized} for a recently introduced localized version of  \textsc{deim} and to \cite{peherstorfer2015online} for online adaptivity approach to \textsc{deim} that adjusts the \textsc{deim} subspace and the interpolation indices with online low-rank updates.

Consider the nonlinear \textsc{ph} system in (\ref{PHdef}) where $\bfJ$, $\bfR$, and  $\bfB$  are constant, and 
$\bF(\bx) =\nabla_{\bx} H(\bx)$ is nonlinear.   In the reduced system (\ref{redPHdef}),
we use the approximation (\ref{redHam}):
$\nabla_{\bx} H(\bx) \approx \bW_r \bV_r^T \nabla_{\bx} H( \bx) $, for $\bx\in \mathsf{Ran}(\bV_r)$. 
\textsc{deim} can be applied directly to $\bF$ (using $\bW_r$ instead of $\bU$) to obtain an approximation
in the form of
\[
\bF(\bV_r\bx_r) \approx \bW_r (\bfE^T\bW_r)^{-1}\bfE^T\bF(\bV_r\bx_r),
\]
 allowing us to evaluate the nonlinear term 
with low complexity. However, this approach will not preserve the underlying \textsc{ph} structure; it
generally will not produce a passive system; and indeed, the reduced system might no longer be stable. 
We modify \textsc{deim} to overcome these shortcomings. 

\subsection{The DEIM Hamiltonian} \label{DEIMHamiltonian}
We continue to assume that the source of nonlinearity in the system (\ref{PHdef}) lies in 
the Hamiltonian gradient: $\nabla_{\bx} H(\bx)$, and we
focus on approximating this nonlinear term in a way that is consistent with the \textsc{ph} structure, 
so that the complexity does not depend on the original full-order dimension.  
This restriction comes largely without loss of generality since additional state-space dependence of $\bfJ$, $\bfR$, and $\bfB$ can be accommodated with usual \textsc{deim}-based approaches with
no threat to the underlying port-Hamiltonian structure. 

We first identify a linear component of $\nabla_{\bx} H(\bx)$, or equivalently, 
a quadratic component of $H(\bx)$:
\begin{equation}
H(\bx) =  \mbox{$\frac{1}{2}$}\bx^T\bQ\bx + h(\bx)
 \label{defn:nonlinFull:2terms}
\end{equation}
where $\bQ$ is an $n\times n$ positive-definite, constant matrix.  A typical choice may be $\bQ = \nabla^2_\bx H (\bx_0) $ at 
 an equilibrium point, $\nabla_\bx H(\bx_0) = {\bf 0}$. Similar strategies are considered in \cite{Hochman2011} to maintain high accuracy.   Once $\bQ$ is selected, (\ref{defn:nonlinFull:2terms}) determines $h(\bx)$ and $\nabla_{\bx} h(\bx)$ then captures the remaining nonlinear portion of $\nabla_{\bx} H(\bx)$. 

We select a new modeling basis, $\bfU_m$,  orthogonalized with respect to $\bQ$, so that
$\nabla_{\bx}h(\bx(t)) \approx \bfU_m \bfg_m (t)$ 
and $\bfU_m^T \bfQ\bfU_m =\bfI$.  There will be a variety of choices for $\bfU_m$; 
we consider a couple of them below. 
  Using $\bfU_m$, we calculate \textsc{deim} indices, $\wp_1, \dots, \wp_{m}$ 
  with Algorithm~\ref{alg:DEIM_Ubasis},
  define the associated \textsc{deim} projection, 
$\mathbb{P}= \bfU_m(\bfE_m^T \bfU_m )^{-1}\bfE_m^T$,   
and finally, introduce a ``\textsc{deim} Hamiltonian":
\begin{equation} \label{DEIMHam}
\widehat{H}(\bx)= \frac{1}{2}\,\bx^T\, \bfQ\,\bx\, +\, h(\mathbb{P}^T\bx).
\end{equation}
Observe that $\nabla_{\bx}\widehat{H}(\bx)= \bfQ\bx+\mathbb{P}\nabla_{\bx}h(\mathbb{P}^T\bx)$, so 
the error induced by $\nabla_{\bx}\widehat{H}(\bx)$ is
$$
\nabla_{\bx}H(\bx)-\nabla_{\bx}\widehat{H}(\bx) = \nabla_{\bx} h(\bx)-\mathbb{P}\nabla_{\bx}h(\mathbb{P}^T\bx).
$$
If $\bfQ$ is chosen well then $\nabla_{\bx}\widehat{H}(\bx)$ can exactly recover the ``linear part" of $\nabla_{\bx}H(\bx)$. 
Observe that the evaluation of the remaining nonlinear term,
 $\mathbb{P}\nabla_{\bx}h(\mathbb{P}^T\bx)$, only involves the evaluation of $m$ elements of $\nabla_{\bx}h$  on only $m$ nonzero arguments - the remaining arguments having only $m$ nonzero values.    
  
 Observe that trivially $\nabla_{\bx}H(\bx)=\nabla_{\bx}\widehat{H}(\bx)$ when $m=n$.  However, the approximation  
 $\nabla_{\bx}H(\bx)\approx \nabla_{\bx}\widehat{H}(\bx)$ can be effective even for significantly smaller $m\ll n$.
 Nonetheless, the enforced symmetry in this approximation appears to require some additional considerations.  
 
Let $\Omega\subset \mathbb{R}^n$, be a compact convex set containing 
 each of the trajectories $\bx(t)$ and $\mathbb{P}^T\bx(t)$, for $0\leq t\leq T$ on its interior.   
 Define 
 $$
 \bfK^{\circ}=\int_{\Omega}\nabla_{\bx}h(\bx) \nabla_{\bx}h(\bx)^T \bQ \ d\nu(\bx).
 $$
 $\bfK^{\circ}$ is $\bQ$-selfadjoint and positive-definite with associated eigenpairs 
 $\bfK^{\circ}\,\bfu_k^{\circ}=\sigma_k^{\circ}\,\bfu_k^{\circ}$ 
 for $\sigma_1^{\circ}\geq\sigma_2^{\circ}\geq\ldots \geq\sigma_n^{\circ}\geq 0$ and 
  $\bQ$-orthonormal eigenvectors, $\bfu_i^{\circ T} \bfQ \bfu_j^{\circ} = \delta_{ij}$. 
  For any choice of $0<m\leq n$, designate the dominant modes as $\bfU_m^{\circ}=[ \bfu_1^{\circ} ,\,\ldots,\, \bfu_m^{\circ}]$ and 
  the complementary subdominant modes as $\widetilde{\bfU}_m^{\circ} =[ \bfu_{m+1}^{\circ} ,\,\ldots,\, \bfu_n^{\circ}]$.  Observe that if
  $\Omega$ shrank to $\bfx(t)$, $\bfU_m^{\circ}$ could be considered as the limiting case of a usual POD basis of dimension $m$.  
  
We have in that case,
 $$
 \bfK^{\circ}\, \bfU_m^{\circ} = \bfU_m^{\circ}\mathsf{diag}(\sigma_1^{\circ},\,\sigma_2^{\circ},\,\ldots,\,\sigma_m^{\circ})\quad \mbox{and}\quad
 \bfK^{\circ}\, \widetilde{\bfU}_m^{\circ} = \widetilde{\bfU}_m^{\circ}\mathsf{diag}(\sigma_{m+1}^{\circ},\,\ldots,\,\sigma_n^{\circ}).
 $$ 
 If $\mathscr{E}_{\bQ}(\nabla_{\bx}h(\bfx),\bU_m^{\circ})$ denotes the (local) best $\bQ$-norm approximation error 
to $\nabla_{\bx}h(\bfx)$ out of $\bU_m^{\circ}$ (as defined in Lemma \ref{lemma:DEIMErrBd}), then
 {\small 
\begin{align*}
 \left(\varepsilon_{\bQ}^{\circ(m)}\right)^2=\int_{\Omega} \mathscr{E}_{\bQ}&(\nabla_{\bx}h(\bfx),\bU_m^{\circ})^2 \, d\nu(\bx)=\int_{\Omega} \|(\bfI-\bU_m^{\circ}\bU_m^{\circ T} \bQ)\nabla_{\bx}h(\bx)\|_{\bQ}^2  \ d\nu(\bx) \\
 = \int_{\Omega}& \mathsf{trace}\left((\bfI-\bU_m^{\circ}\bU_m^{\circ T} \bQ)\nabla_{\bx}h(\bx)\nabla_{\bx}h(\bx)^T(\bfI-\bQ\bU_m^{\circ}\bU_m^{\circ T}) \bQ \right) \, d\nu(\bx) \\
 = &\int_{\Omega} \mathsf{trace}\left(\nabla_{\bx}h(\bx)\nabla_{\bx}h(\bx)^T \bQ(\bfI-\bU_m^{\circ}\bU_m^{\circ T} \bQ) \right) \ d\nu(\bx) \\
 & =  \mathsf{trace}\left( \bfK^{\circ}(\bfI-\bU_m^{\circ}\bU_m^{\circ T} \bQ) \right)  = \sum_{k=m+1}^n \sigma_{k}^{\circ}.
\end{align*}
}
\begin{lemma} \label{symmDEIMapprox_h}
Suppose $h(\bx)$ is continuously differentiable in $\Omega$. Construct the DEIM projection using $\bU_m^{\circ}$ with 
$\mathbb{P}_{\circ}=\bU_m^{\circ}(\bfE_m^T\bfU_m^{\circ} )^{-1} \bfE_m^T$.  Then
\begin{equation}  \label{hDEIMbound}
|h(\bfx)-h(\mathbb{P}_{\circ}^T\bx)| \leq \|(\bI-\mathbb{P}_{\circ}^T)\bx(t)\|_{\bQ^{-1}}  \cdot \varepsilon_{\bQ}^{\circ(m)}
\end{equation}
\end{lemma}
\begin{proof}  
Define the columns of $\widetilde{\bfE}_m$ to be complementary to those of $\bfE_m$ 
so that $[\bfE_m,\widetilde{\bfE}_m]$ is an $n\times n$ permutation matrix.
 Observe that $\mathsf{Ran}(\mathbb{P}_{\circ}^T)=\mathsf{Ran}(\bfE_m)$ and 
$\mathsf{Ran}(\mathbf{I}-\mathbb{P}_{\circ}^T)=\mathsf{Ker}(\mathbb{P}_{\circ}^T)=\mathsf{Ker}(\bfU_m^{\circ T})=\mathsf{Ran}(\bQ\widetilde{\bfU}_m^{\circ})$. Hence, 
$$
\bx=\mathbb{P}_{\circ}^T\bx+(\bI-\mathbb{P}_{\circ}^T)\bx= \bfE_m \bfzeta_1 + \bQ\widetilde{\bfU}_m^{\circ}\bfzeta_2
$$ 
for continuous functions given as
$$
\bfzeta_1(\bx) = (\bfU_m^{\circ T}\bfE_m )^{-1}\bfU_m^{\circ T}\bx\quad\mbox{ and }\quad 
\bfzeta_2(\bx) = (\widetilde{\bfE}_m^T \bQ\widetilde{\bfU}_m^{\circ} )^{-1}\widetilde{\bfE}_m^T\bx.
$$  
Note that, $\bQ^{\frac12}\widetilde{\bfU}_m^{\circ}\bfzeta_2 =  \bQ^{-\frac12}(\bI-\mathbb{P}_{\circ}^T)\bx$ 
and $\|\bfzeta_2\|=\|(\bI-\mathbb{P}_{\circ}^T)\bx\|_{\bQ^{-1}}$.

For each fixed $\bx\in\Omega$, 
define the path connecting $\bx$ and $\mathbb{P}_{\circ}^T\bx$:
$$
\mathcal{C}(\bx)=\{\theta\bx+(1-\theta)\mathbb{P}_{\circ}^T\bx\left|\ \theta\in[0,1]\right.\}
=\{\bfE_m \bfzeta_1 +\theta\, \bQ\widetilde{\bfU}_m^{\circ}\bfzeta_2 \left|\ \theta\in[0,1]\right. \}.
$$
If we define, 
$f(\bfzeta_1, \bfzeta_2)=h(\bfx)=h(\bfE_m \bfzeta_1 + \widetilde{\bfU}_m^{\circ}\bfzeta_2)$, observe that
$$
\nabla_{\bfzeta_2}f(\bfzeta_1, \bfzeta_2) = \widetilde{\bfU}_m^{\circ T} 
\bQ\nabla_{\bx}h(\bfE_m \bfzeta_1 + \bQ\widetilde{\bfU}_m^{\circ}.\bfzeta_2).
$$
Integrating $\nabla_{\bfzeta_2}f$ along $\mathcal{C}(t)$, we find  
{\small \begin{align*}
|h(\bfx)-h(\mathbb{P}_{\circ}^T\bx)|=&|f(\bfzeta_1, \bfzeta_2)-f(\bfzeta_1, 0)|
=\left|\int_{\mathcal{C}(\bx)} \widetilde{\bfU}_m^{\circ T} \bQ\nabla_{\bx}h(\bfxi) \,\cdot \, d\bfxi\right|\\
= &\left|\int_0^1 \bfzeta_2^T \widetilde{\bfU}_m^{\circ T}\bQ
\nabla_{\bx}h(\bfE_m \bfzeta_1 + \theta\, \bQ\widetilde{\bfU}_m^{\circ} \bfzeta_2) \, d\theta\right| \\
= & \| \bfzeta_2\|\left(\int_0^1 \| \widetilde{\bfU}_m^{\circ T}\bQ
\nabla_{\bx}h(\bfE_m \bfzeta_1 + \theta\, \bQ\widetilde{\bfU}_m^{\circ} \bfzeta_2)\|^2 \, d\theta\right)^{\!\frac12} \\
\leq & \| \bfzeta_2\|  \left(\int_{\Omega} \|\widetilde{\bfU}_m^{\circ T}\bQ \nabla_{\bx}h(\bfxi)\|^2\ d\nu(\bfxi)\right)^{\!\frac12} \\
\leq & \|(\bI-\mathbb{P}_{\circ}^T)\bx(t)\|_{\bQ^{-1}}  \cdot \varepsilon_{\bQ}^{\circ(m)}.
\end{align*}
}
 \end{proof}
 
Thus, a modest extension of the usual POD basis is sufficient to produce DEIM approximations  and associated DEIM Hamiltonians that converge uniformly to the true Hamiltonian with a rate related to the decay rate of $\sigma_{k}^{\circ}$.   
Stronger hypotheses on the approximating modes appear to be necessary in order to assure rapid
convergence of the corresponding gradients.  Towards that end, we assume that 
$h$ is twice continuously differentiable and define 
 $$
 \bfK=\int_{\Omega}\nabla_{\bx}h(\bx) \nabla_{\bx}h(\bx)^T \bQ \ d\nu(\bx)
 +\int_{\Omega}\left(\nabla_{\bx}^2 h(\bx)\right)^2\, \bQ \ d\nu(\bx). $$
 $\bfK$ is $\bQ$-selfadjoint and positive-definite with associated eigenpairs 
 $\bfK\,\bfu_k=\sigma_k\,\bfu_k$ 
 for $\sigma_1\geq\sigma_2\geq\ldots \geq\sigma_n\geq 0$ and 
  $\bQ$-orthonormal eigenvectors, $\bfu_i^{ T} \bfQ \bfu_j = \delta_{ij}$. 
  Analogous to what has gone before, designate dominant modes as $\bfU_m=[ \bfu_1 ,\,\ldots,\, \bfu_m]$ and 
   subdominant modes as $\widetilde{\bfU}_m =[ \bfu_{m+1} ,\,\ldots,\, \bfu_n]$, with
   $$
 \bfK\, \bfU_m = \bfU_m\mathsf{diag}(\sigma_1,\,\sigma_2,\,\ldots,\,\sigma_m),\qquad
 \bfK\, \widetilde{\bfU}_m = \widetilde{\bfU}_m\mathsf{diag}(\sigma_{m+1},\,\ldots,\,\sigma_n).
 $$ 
 and an associated error, $\varepsilon_{\bQ}^{(m)}$, defined similarly
{\small
 $$
 \left(\varepsilon_{\bQ}^{(m)}\right)^2 = \int_{\Omega}  \|(\bfI-\bU_m\bU_m^T \bfQ)\nabla_{\bx}h(\bfxi)\|_{\bfQ}^2 +  
 \|(\bfI-\bU_m\bU_m^T \bfQ)\nabla_{\bx}^2 h(\bfxi)\|_{\bfQ}^2 \ d\nu(\bfxi) =\sum_{k=m+1}^n \sigma_{k}. 
 $$
 }
 \begin{lemma} \label{symmDEIMapprox_gradh}
 Given the modeling basis, $\bU_m$, and the basis completion, $\widetilde{\bfU}_m$, as described above, 
 $$
 \|\nabla_{\bx} h-\mathbb{P} \circ(\nabla_{\bx}h)\circ\mathbb{P}^T\|_{{\mathcal L}_2(\Omega)}
 \leq \sqrt{2} \|\mathbb{P}\|_{\bQ}\ M\ \varepsilon_{\bQ}^{(m)}
$$
where $M=\max\{1,\, \|\mathbb{P}\|_{\bQ}\sqrt{\frac{\|\bE_m^T\bQ\bE_m\|}{n-m+1}} \sup_{\bx\in\Omega}\|\bx\|_{\bQ^{-1}}\}$
\end{lemma}
\begin{proof}  
Note that 
{\small 
\begin{align*}
&  \|\nabla_{\bx} h-\mathbb{P} \circ(\nabla_{\bx}h)\circ\mathbb{P}^T\|_{{\mathcal L}_2(\Omega)}^2  = \int_{\Omega}\|\nabla_{\bx} h(\bx)-\mathbb{P}\nabla_{\bx}h(\mathbb{P}^T\bx)\|_{\bQ}^2\,d\nu(\bx)\\
&\quad \,\leq 
   2\int_{\Omega} \|\mathbb{P}\left(\nabla_{\bx} h(\bx)-\nabla_{\bx}h(\mathbb{P}^T\bx)\right)\|_{\bQ}^2\, d\nu(\bx)
    + 2 \int_{\Omega} \|\left( \mathbf{I}-\mathbb{P}\right)\nabla_{\bx}h(\bx)\|_{\bQ}^2\, d\nu(\bx) \\
  &\quad \leq 
   2 \,\|\mathbb{P}\|_{\bQ}^2\, \|\bfE_m^T\bQ\bfE_m\|\,\int_{\Omega}  \|\bfE_m^T\left(\nabla_{\bx} h(\bx)-\nabla_{\bx}h(\mathbb{P}^T\bx)\right)\|^2\, d\nu(\bx) \\
    & \qquad \qquad \qquad + 2 \|\bI-\mathbb{P}\|_{\bQ}^2 \int_{\Omega} \|\widetilde{\bfU}_m^T\bQ \nabla_{\bx}h(\bx)\|^2\, d\nu(\bx)
 \end{align*}
 }
Consider the function $f(\bfzeta_1, \bfzeta_2)$ as defined in the proof of Lemma \ref{symmDEIMapprox_h}, but now adapted to the new basis, 
$\bfU_m$.  Note that 
$$
\nabla_{\bfzeta_1}f(\bfzeta_1, \bfzeta_2) = \bfE_m^T\nabla_{\bx}h(\bfE_m \bfzeta_1 + \bQ\widetilde{\bfU}_m\bfzeta_2).
$$
Since $\nabla_{\bx}h$ is continuously differentiable in an open neighborhood of $\mathcal{C}(t)$, we have that 
\begin{align*}
 \bfE_m^T(\nabla_{\bx}h(\bx)-\nabla_{\bx}h(\mathbb{P}^T\bx)) &= 
 \nabla_{\bfzeta_1}f(\bfzeta_1, \bfzeta_2)-\nabla_{\bfzeta_1}f(\bfzeta_1, 0) \\
&=\nabla_{\bfzeta_1}\int_0^1 \nabla_{\bx}h(\bfE_m \bfzeta_1 + \theta\, \bQ\widetilde{\bfU}_m \bfzeta_2) \,\cdot \,
\left(\bQ\widetilde{\bfU}_m \bfzeta_2\right)\, d\theta\\
&=\left(\int_0^1 \bfE_m^T\nabla^2_{\bx}h(\bfE_m \bfzeta_1 + \theta\, \bQ\widetilde{\bfU}_m \bfzeta_2) \, \bQ\widetilde{\bfU}_m \, d\theta\right)\, \bfzeta_2
\end{align*}
and so, 
{\small
\begin{align*}
 & \int_{\Omega} \|\bfE_m^T\left(\nabla_{\bx} h(\bx)-\nabla_{\bx}h(\mathbb{P}^T\bx)\right)\|^2\, d\nu(\bx) \\
 & \qquad \leq \sup_{\bx\in\Omega}\| \bfzeta_2(\bx) \|^2 \ \int_0^1 \int_{\Omega} \|\bfE_m^T\nabla^2_{\bx}
  h(\theta\,\bx+(1-\theta)\mathbb{P}^T\bx)\, \bQ\widetilde{\bfU}_m\|^2 \, \, d\nu(\bx)\, d\theta \\
  & \qquad \leq \sup_{\bx\in\Omega}\|(\bI-\mathbb{P}^T)\bx\|_{\bQ^{-1}}^2 \ \int_0^1 \int_{\Omega} \|\bfE_m^T\nabla^2_{\bx}
  h(\bfxi)\, \bQ\widetilde{\bfU}_m\|^2 \,|\det(\theta\,\bI+(1-\theta)\mathbb{P}^T)| \, d\nu(\bfxi)\, d\theta \\
& \qquad  \leq  \|\mathbb{P}\|_{\bQ}^2 \sup_{\bx\in\Omega}\|\bx\|_{\bQ^{-1}}^2
  \left(\int_0^1|\det(\theta\,\bI+(1-\theta)\mathbb{P}^T)|\,d\theta\right)  \ \int_{\Omega} \|\widetilde{\bfU}_m^T \bfQ\nabla_{\bx}^2 h(\bfxi) \bfE_m\|^2 \ d\nu(\bfxi)\\
  & \qquad \qquad \qquad  \leq  \frac{\|\mathbb{P}\|_{\bQ}^2}{n-m+1} \sup_{\bx\in\Omega}\|\bx\|_{\bQ^{-1}}^2
 \ \int_{\Omega} \|\widetilde{\bfU}_m^T \bfQ\nabla_{\bx}^2 h(\bfxi)\|^2 \ d\nu(\bfxi)
\end{align*}
}
The interchange of order of integration is justified in the first inequality since the integrand is uniformly bounded on the joint domain $[0,1]\times \Omega$.  In the second inequality, we have introduced the (linear) change of variable 
$\bfxi= [\theta\,\bI+(1-\theta)\mathbb{P}^T]\bx\in\Omega$ for all $\theta\in[0,1]$.  The third inequality incorporates the observation 
$ \|\bI-\mathbb{P}^T\|_{\bQ^{-1}}=\|\mathbb{P}\|_{\bQ}.$  The last inequality uses the elementary observation that the Jacobian matrix for the change-of-variables has exactly two eigenvalues: $\lambda=1$ with multiplicity $m$ and $\lambda=\theta$ with multiplicity $n-m$. 
\end{proof}

\subsection{Preserving port-Hamiltonian Structure with DEIM} \label{SectPreserveportHamStruct}
Equipped with the \textsc{deim} Hamiltonian, $\widehat{H}(\bfx)$,  we apply our previously described 
structure-preserving approach: 
we define a \emph{reduced} \textsc{deim} Hamiltonian, $\widehat{H}_r(\bfx_r)=\widehat{H}(\bfV_r\bfx_r)$. 
 As before, we expect that
$\nabla_{\!\bfx}\widehat{H}(\bfV_r\bfx_r(t))\approx \bfW_r\nabla_{\!\bfx_r}\widehat{H}_r(\bfx_r(t)) $,
and the \textsc{deim}-reduced port-Hamiltonian approximation becomes
\begin{equation} \label{redPHDEIMdef}
\begin{array}{l}
 \dot{\mathbf{x}}_r =(\bfJ_r-\bfR_r)\nabla_{\!\bfx_r}\widehat{H}_r(\mathbf{x}_r) +\bfB_r\bfu(t),\\[.1in]
   \quad  \mathbf{y}_r(t) = \mathbf{B}_r^T\, \nabla_{\!\bfx_r}\widehat{H}_r(\mathbf{x}_r) 
\end{array}
\end{equation} 
with $\bfJ_r= \bfW_r^{T} \bfJ \bfW_r$,  $\bfR_r= \bfW_r^{T} \bfR \bfW_r$, $\bfB_r= \bfW_r^{T} \bfB$,
and 
$$
\nabla_{\!\bfx_r}\widehat{H}_r(\mathbf{x}_r) = \bfV_r^T\nabla_{\bx}\widehat{H}(\bfV_r\bx_r)= \bx_r+ \bfV_r^T\mathbb{P}\nabla_{\bx}h(\mathbb{P}^T \bfV_r\bx_r).
$$   
In the evaluation of 
 $\bfV_r^T\mathbb{P}\nabla_{\bx}h(\mathbb{P}^T \bfV_r\bx_r)$, the $r\times m$ matrix, $\bfV_r^T\bfU_m$, may be precomputed and
 $\nabla_{\bx}h$ is evaluated only on arguments having $m$ nonzero values and 
 only $m$ values need to be evaluated (that is, only at the \textsc{deim} indices). 
%
Algorithm~\ref{alg:poddeimph} summarizes the steps for constructing 
a structure-preserving \textsc{pod-deim} reduced system.
{\small 
\begin{algorithm}[h!tbp]
\caption{: \textsc{pod-deim} Structure-preserving \textsc{deim} reduction of NLPH systems}
 \label{alg:poddeimph}
\vspace{2mm}

\begin{algorithmic}[1]
\STATE 
Select a positive-definite matrix $\bQ$ (nominally approximating $\nabla^2_\bx H (\bx_0) $) \\
and define $ h(\bx)= H(\bx) -  \frac{1}{2}\bx^T\bQ\bx$.

\STATE
Generate the trajectory $\bfx(t)$, (or ensemble of trajectories) and collect snapshots:\\
\quad $\mathbb{X}=\left[\bfx(t_0),\bfx(t_1),\bfx(t_2),\ldots,\bfx(t_N)\right]$.\\
\quad $\mathbb{F}=\left[\nabla_{\!\bfx}H(\bfx(t_0)),\nabla_{\!\bfx}H(\bfx(t_1)),\ldots,\nabla_{\!\bfx}H(\bfx(t_N))\right]$.\\
\quad $\mathbb{G}=\left[\nabla_{\!\bfx}h(\bfx(t_0)),\nabla_{\!\bfx}h(\bfx(t_1)),\ldots,\nabla_{\!\bfx}h(\bfx(t_N))\right]$.\\

\vspace{2mm}
\STATE    
Truncate  the SVDs of the snapshot matrices, $\mathbb{X}$, $\mathbb{F}$, 
and $\mathbb{G}$ to obtain, respectively, the \textsc{pod} basis matrices
 $ \widetilde{\bfV}_r$, $ \widetilde{\bfW}_r $, and $\bU_m$ used to approximate\\
 $\bx(t)\approx \bfV_r\bx_r(t)$;  $\nabla_{\bx}H(\bx)\approx \bfW_r\bff_r(t)$;  and
  $\nabla_{\bx}h(\bx(t)) \approx \bfU_m \bfg_m (t)$. \\
Choose bases so that $\bfV_r^T\bfQ\bfV_r =\bfI$, $\bfV_r^T\bfW_r =\bfI$, 
and $\bfU_m^T \bfQ\bfU_m =\bfI$.

\vspace{2mm}
\STATE    
From Algorithm~\ref{alg:DEIM_Ubasis}, calculate \textsc{deim} indices, $\wp_1, \dots, \wp_{m}$ and the \textsc{deim} projection, 
\quad $\mathbb{P}= \bfU_m(\bfE^T \bfU_m )^{-1}\bfE^T$. 

\vspace{2mm}
\STATE
The \textsc{deim}-reduced port-Hamiltonian approximation becomes
\begin{equation*}
\begin{array}{l}
 \bhatxdot_r = (\bJ_r -\bR_r ) \nabla_{\bhatx_r}\widehat{H}_r(\bhatx_r)  +\bB_r \bu(t) \\[.1in]
 \bhaty_r    \ = \ \bB_r^T \nabla_{\bhatx_r}\widehat{H}_r(\bhatx_r) 
\end{array}
\end{equation*} 
with 
$\bfJ_r= \bfW_r^{T} \bfJ \bfW_r$,  $\bfR_r= \bfW_r^{T} \bfR \bfW_r$, $\bfB_r= \bfW_r^{T} \bfB$,
and 
$$
\nabla_{\! \bhatx_r}\widehat{H}_r(\bhatx_r) = \bfV_r^T\nabla_{\bx}\widehat{H}(\bfV_r \bhatx_r)= \bx_r+ \bfV_r^T\mathbb{P}\nabla_{\bx}h(\mathbb{P}^T \bfV_r \bhatx_r).
$$   
\end{algorithmic}
\end{algorithm}
}

Note that we can apply this structure-preserving \textsc{deim}
approach to bases  $\bV_r$ and $\bW_r$ associated with \HtwoEpsph\  bases obtained 
from Algorithm~\ref{alg:qh2ph}, instead of \podph\ bases obtained from 
Algorithm  \ref{alg:podph}. 
In the numerical examples of \S \ref{sec:Numer_examples}, we consider both strategies, refering to the variant using 
Algorithm~\ref{alg:qh2ph} as  \HtwoEpsdeimph.  

\subsection{Numerical examples}
\label{sec:Numer_examples}

We consider two system models: the nonlinear LC  ladder network from Section~\ref{sec:MOR_PH_exLadder} 
and a Toda lattice model with exponential interactions.  We will illustrate the performance of the structure-preserving \textsc{deim}-based 
model reduction approaches introduced above using the \poddeimph,   \HtwoEpsdeimph \ and the hybrid  \podHtwoEpsph~bases.
\subsubsection{Ladder network}
The $N$-stage nonlinear ladder network studied in Section~\ref{sec:MOR_PH_exLadder} will be considered  again. We will investigate here  the effect of incorporating the structure-preserving \textsc{deim} approximation in reducing the nonlinear term. 
As in Section~\ref{sec:MOR_PH_exLadder}, we use a Gaussian  pulse as the training input and test the reduced models on both the Gaussian pulse and a sinusoidal input.
  Figure~\ref{fig:POD_H2_DEIM_output} shows that   
\poddeimph \ and \HtwoEpsdeimph   \ reduced systems capture the behavior of the original output responses accurately for both inputs, repeating the success of the \podph \ and \HtwoEpsph \ reduced models. Notice from Figure~\ref{fig:POD_H2_DEIM_output} that the linearized system does not give a good approximation to the true output response of the nonlinear system, as pointed out in Section~\ref{sec:MOR_Proj_NPH}. Notice also that, 
the accuracy of the \poddeimph~reduced system ($r=6$, $m=24$) is very close to the accuracy of the \podph~ reduced system ($r=6$); the \podph~ reduced system is slightly more accurate as expected; however the \poddeimph~reduced system has much lower computational complexity as explained in the previous section; 
 thus the structure-preserving \textsc{deim} approximation reduces the computational complexity yet
  does not degrade the accuracy of the reduced model.
 As  the \textsc{deim} dimension increases, the \poddeimph~yields the same accuracy as  \podph~as shown in Figure~\ref{fig:POD_H2_DEIM_change_m} for a fixed \textsc{pod} dimension of $r=6$. 
Figure~\ref{fig:POD_H2_DEIM_errCPU} shows the convergence of the average relative errors. 
Analogous to the previous numerical test in Section~\ref{sec:MOR_PH_exLadder}, we also consider projection bases constructed using the hybrid  \podHtwoEpsph~approach combined with \textsc{deim} approximation. 
As one may see in Figure~\ref{fig:combPOD_H2_DEIM_err}, this hybrid approach is more accurate than the ones constructed using only \textsc{pod}-bases (green lines) or using only  \HtwoEps-bases (blue lines). 

\begin{figure}[h!]
  \caption{
  \scriptsize Ladder Network: ROM  responses to Gaussian pulse (left) and sinusoidal input (right).}
  \centerline{
 \includegraphics[scale=0.31]{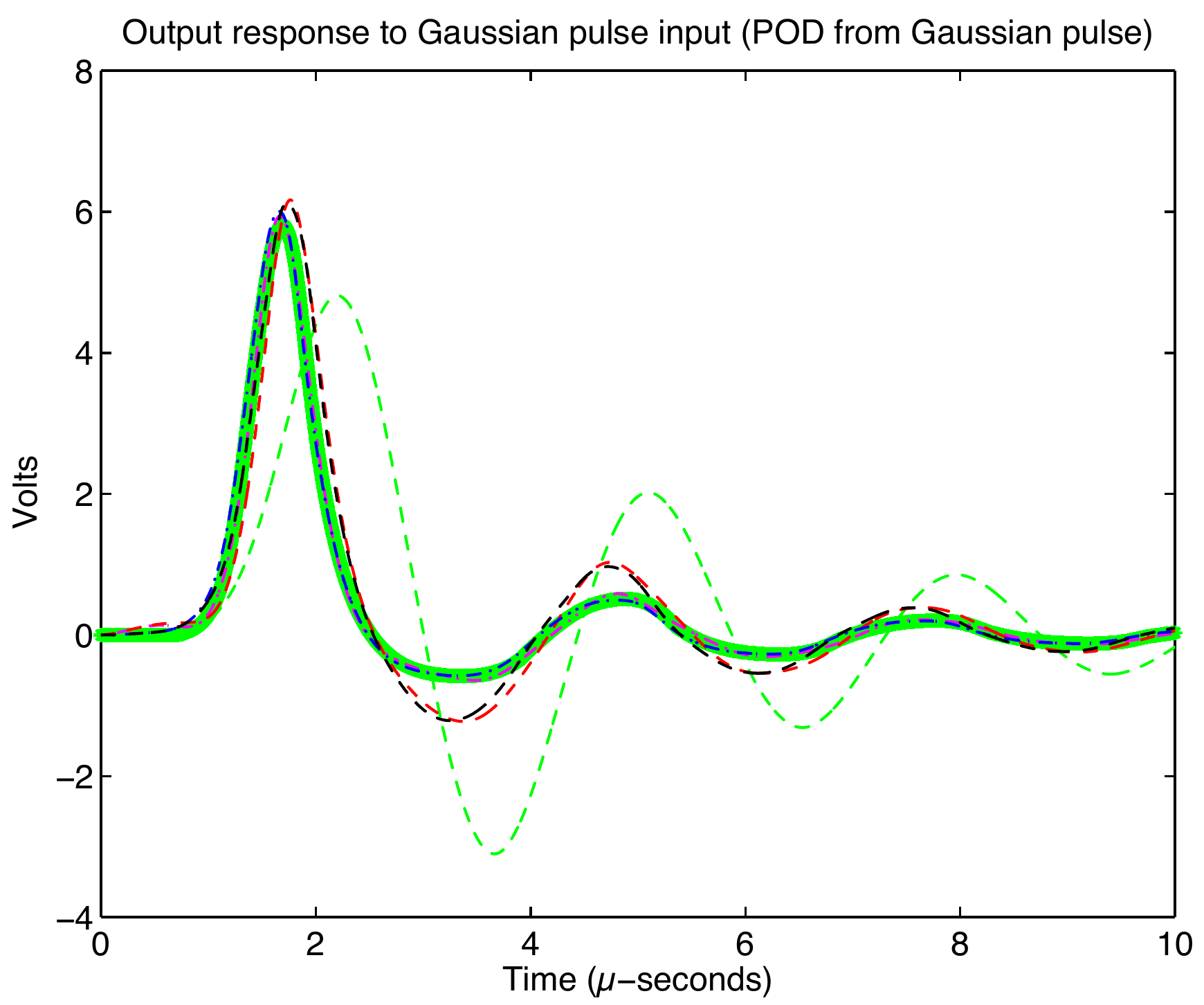}
  \includegraphics[scale=0.31]{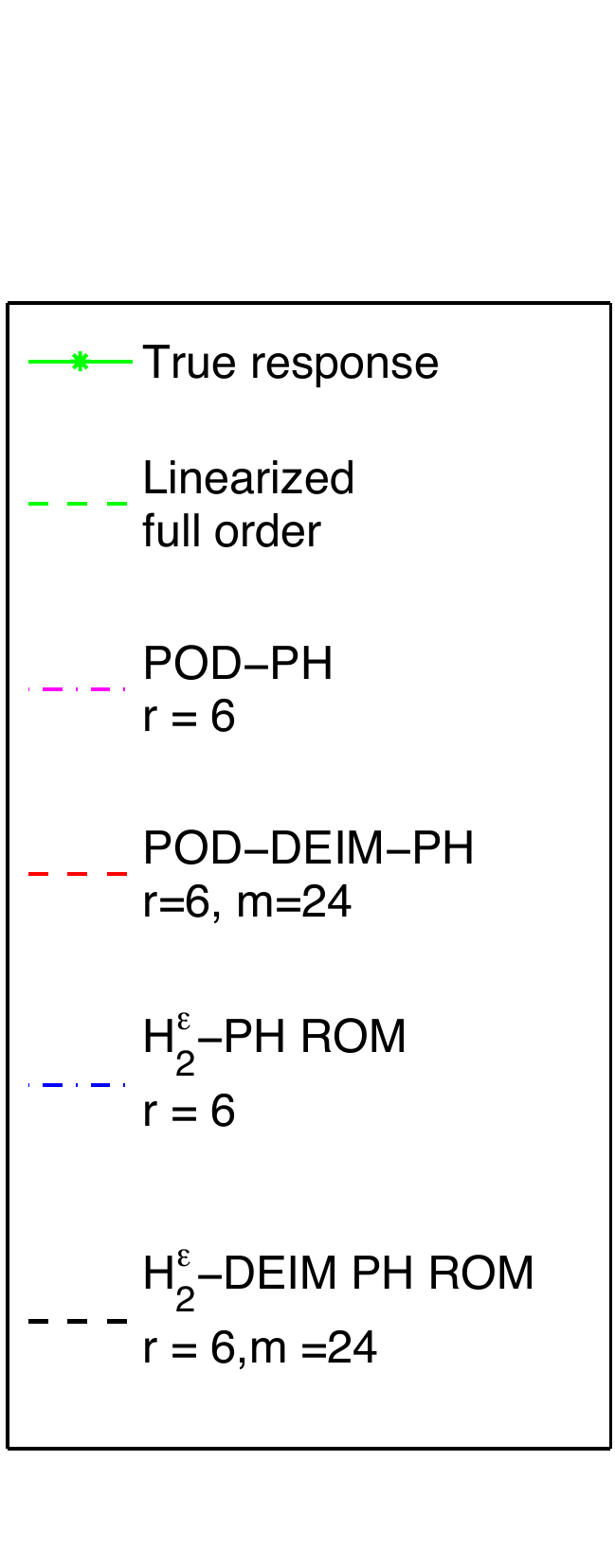}
 \includegraphics[scale=0.31]{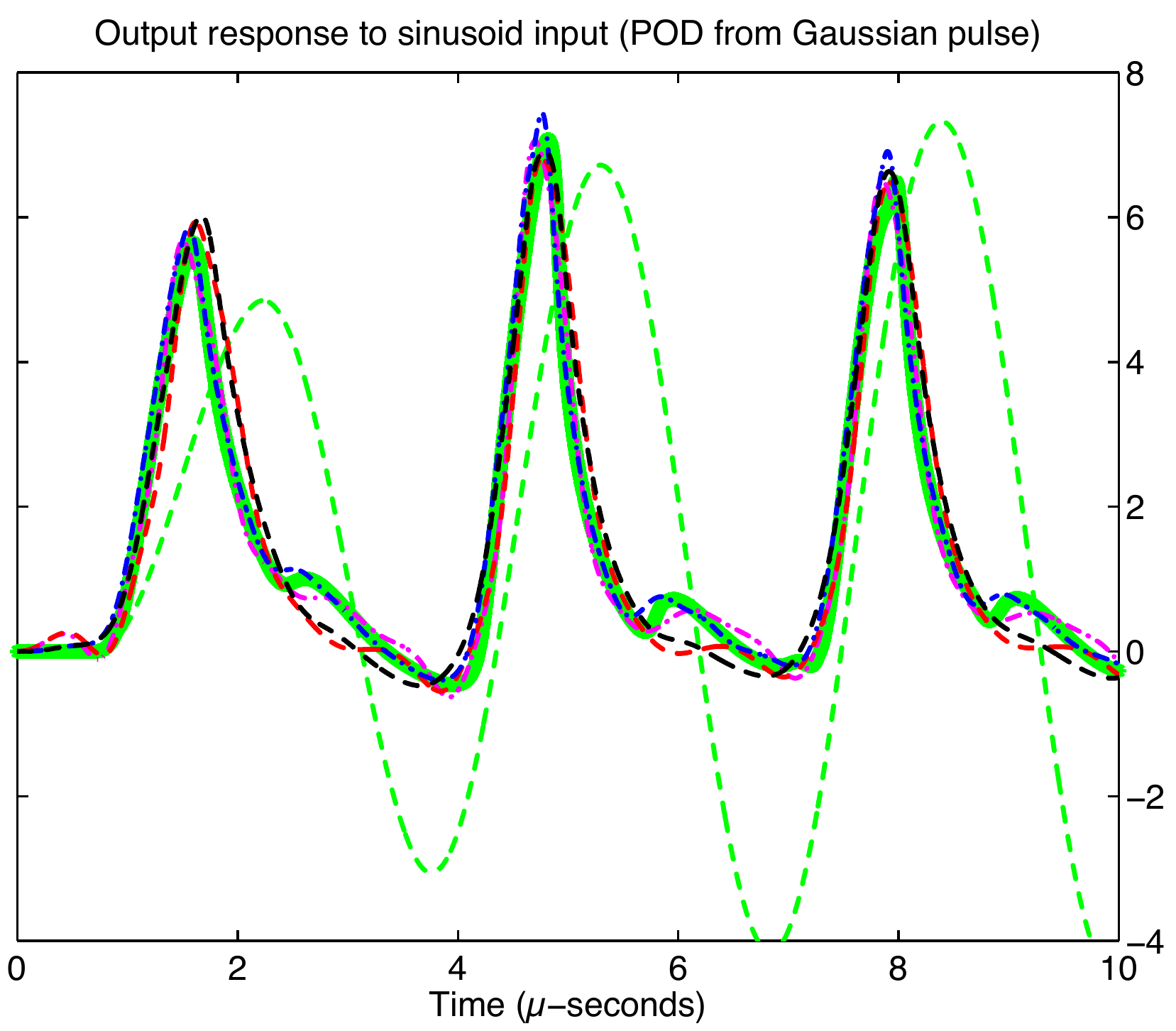}
}
  \label{fig:POD_H2_DEIM_output}
\end{figure}
 
\begin{figure}[h!]
  \caption{
  \scriptsize Ladder Network:
  Outputs from \HtwoEpsph~reduced system (dimension $r = 6$)
   and \HtwoEpsdeimph reduced systems (\HtwoEps dimension $r = 6$ with DIEM dimensions
$m = 6, 12, 14, 48$)--similarly for the plot of \podph and \poddeimph reduced systems. }
  \centerline{
 \includegraphics[scale=0.33]{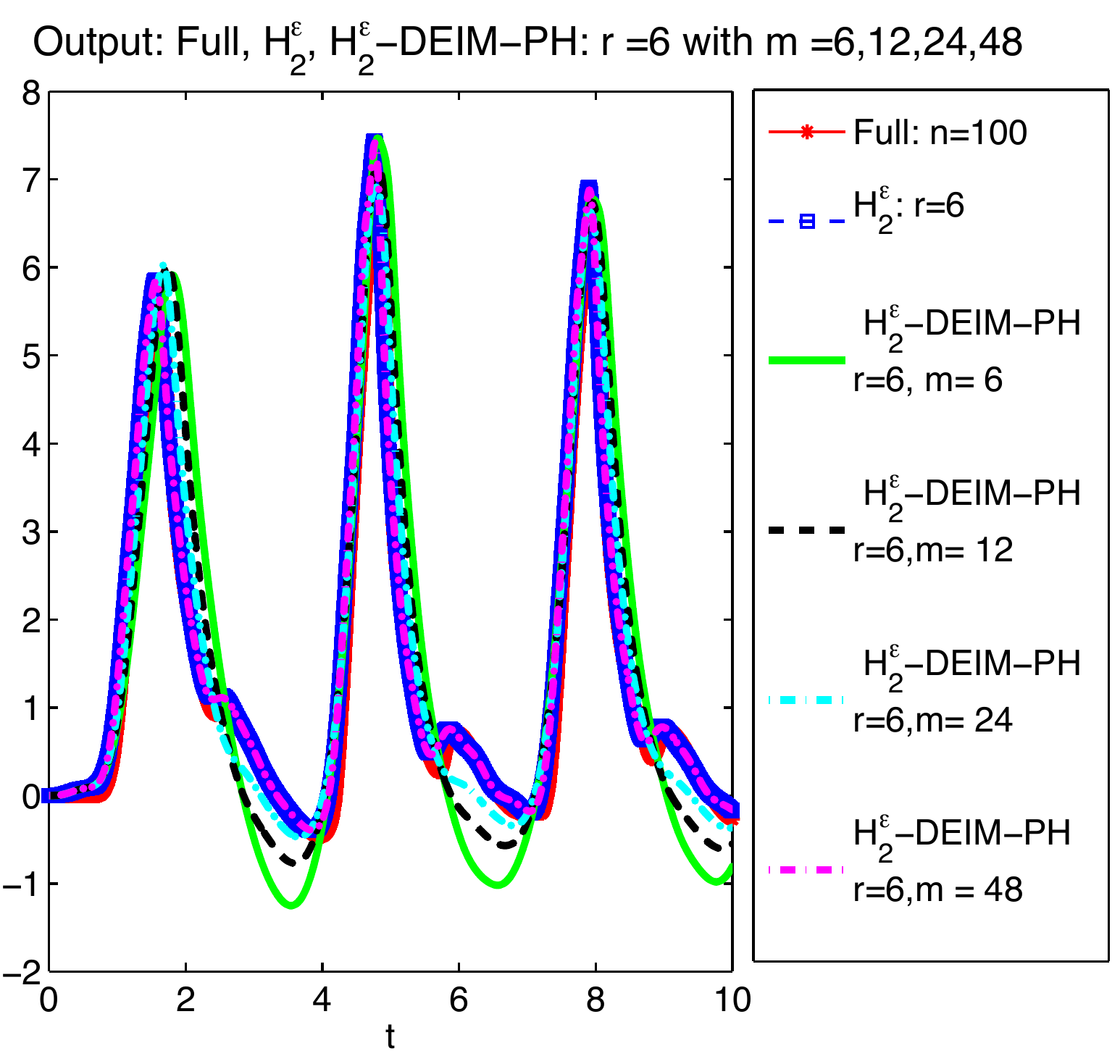}
\hspace{1ex}
 \includegraphics[scale=0.33]{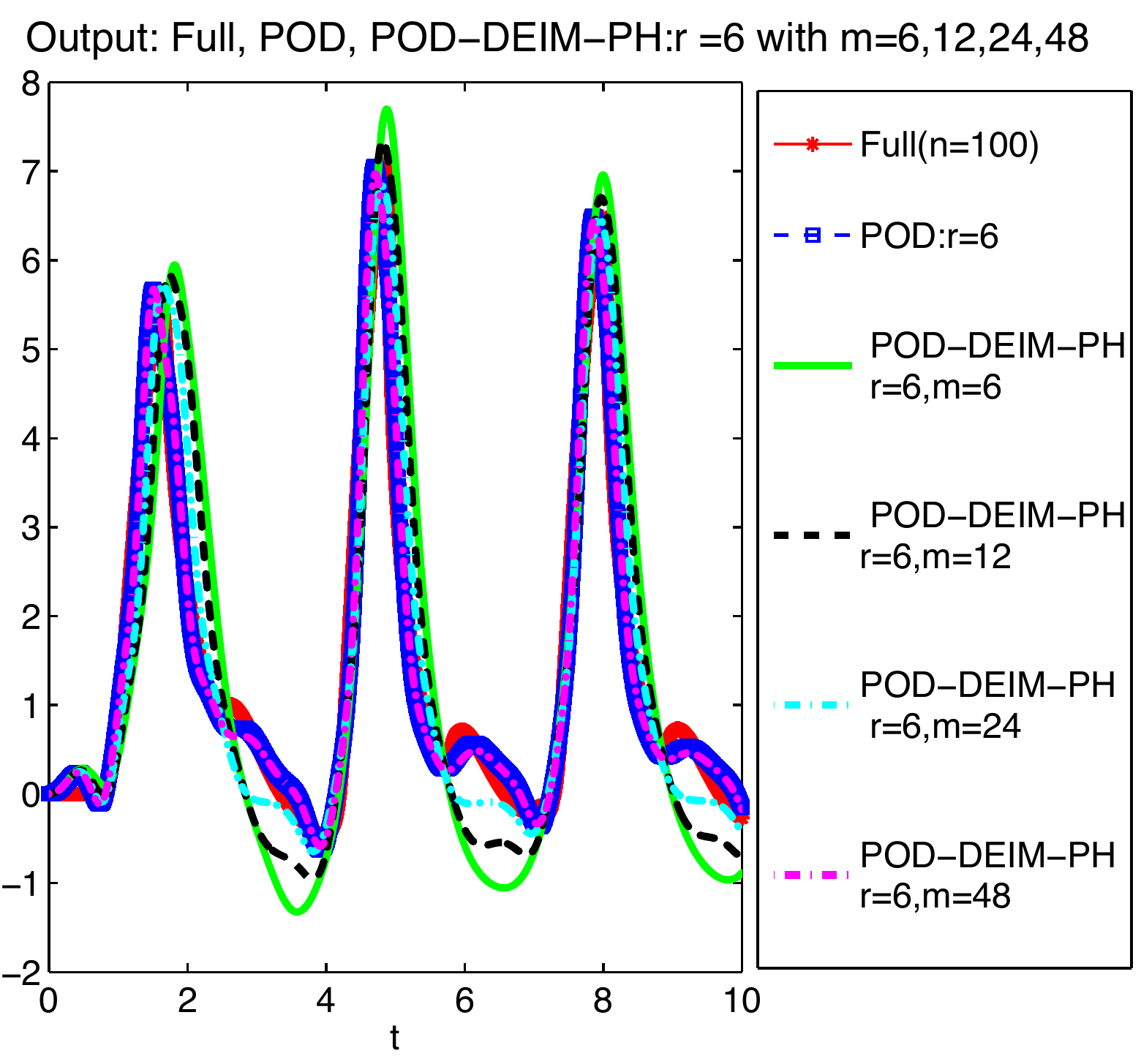}
} 
\label{fig:POD_H2_DEIM_change_m}
\end{figure}
\begin{figure}[h!]
  \caption{
  \scriptsize Ladder Network:
  Average relative errors 
  of the reduced systems constructed from \podph, \poddeimph, \HtwoEpsph, \ and \HtwoEpsdeimph \ bases 
  (with dimensions $r = 2, 4, ...., 18$ and  \textsc{deim} dimension $m=3r$).  
   }
  \centerline{
 \includegraphics[scale=0.35]{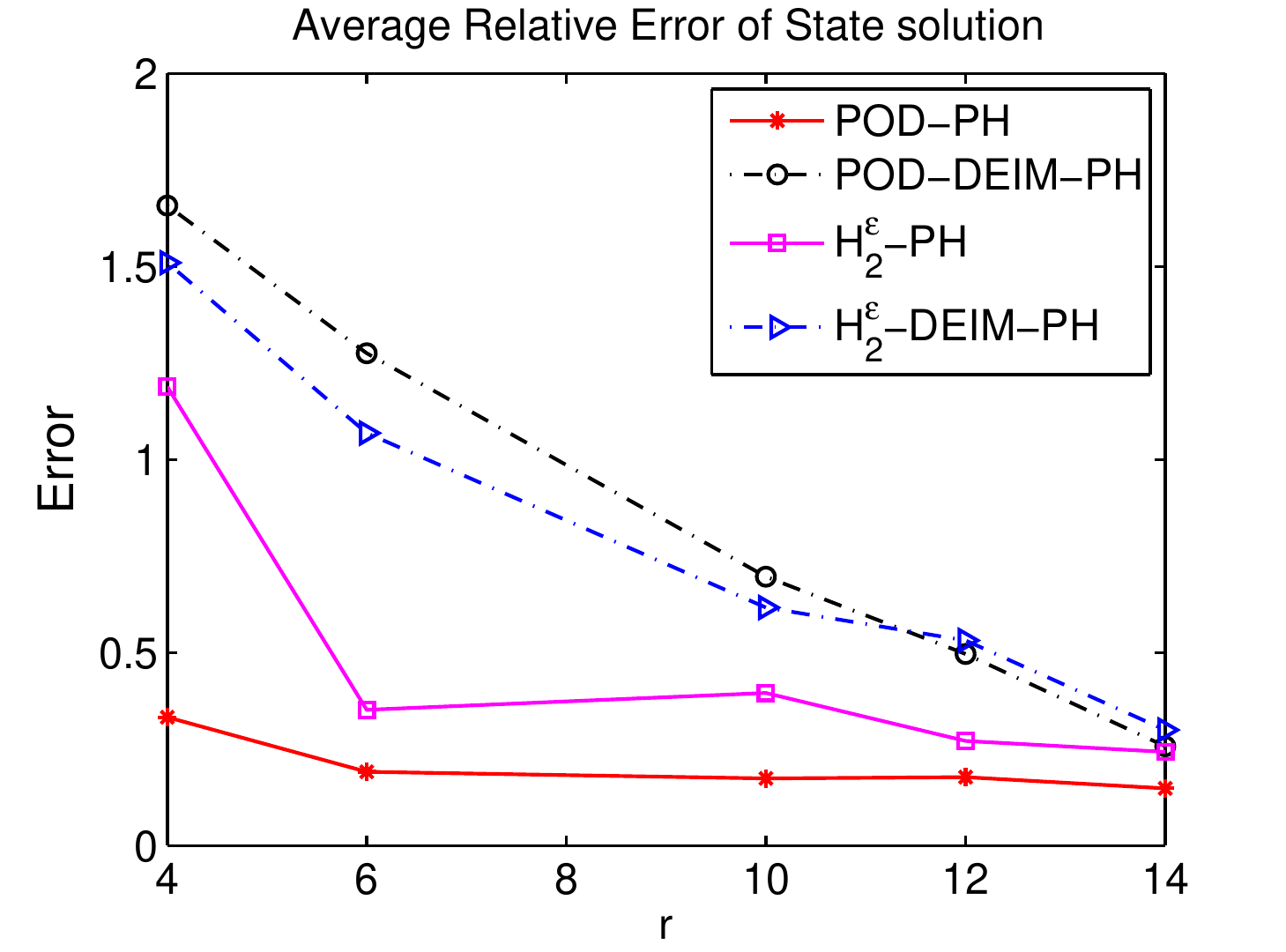}\hspace{-2.5ex}
 \includegraphics[scale=0.35]{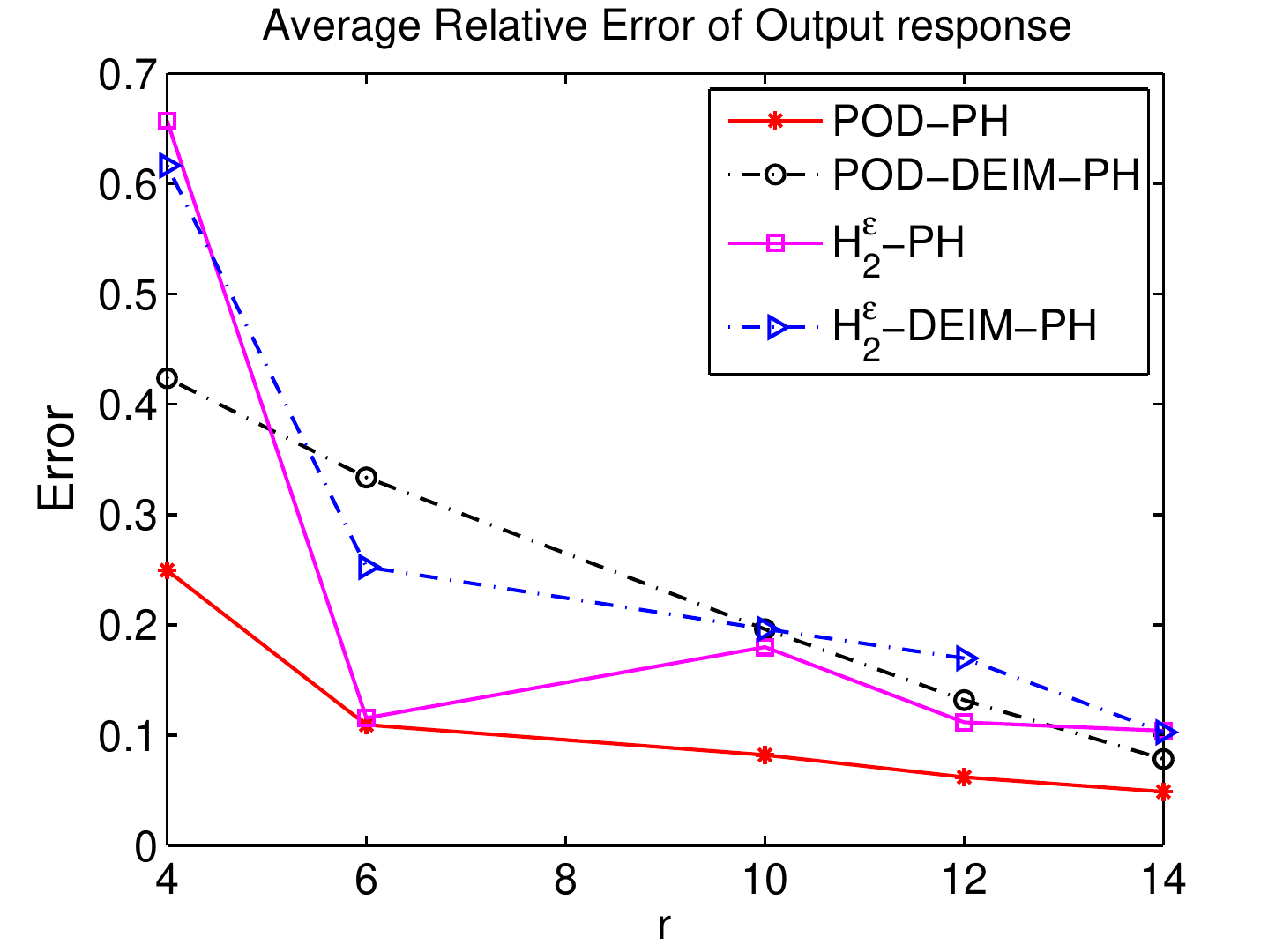}\hspace{-2.5ex}
 }
\label{fig:POD_H2_DEIM_errCPU}
\end{figure}
\begin{figure}[h!]
  \caption{
  \scriptsize Ladder Network:
  Average relative errors using the \podph, \poddeimph, \HtwoEpsph, \HtwoEpsdeimph, and the hybrid 
  \podHtwoEpsph \ and \podHtwoEpsdeimph  \ bases
  (with  $r = 2, 4, ...., 18$ and  \textsc{deim} dimension $m=3r$).  
   }
  \centerline{
\includegraphics[scale=0.5]{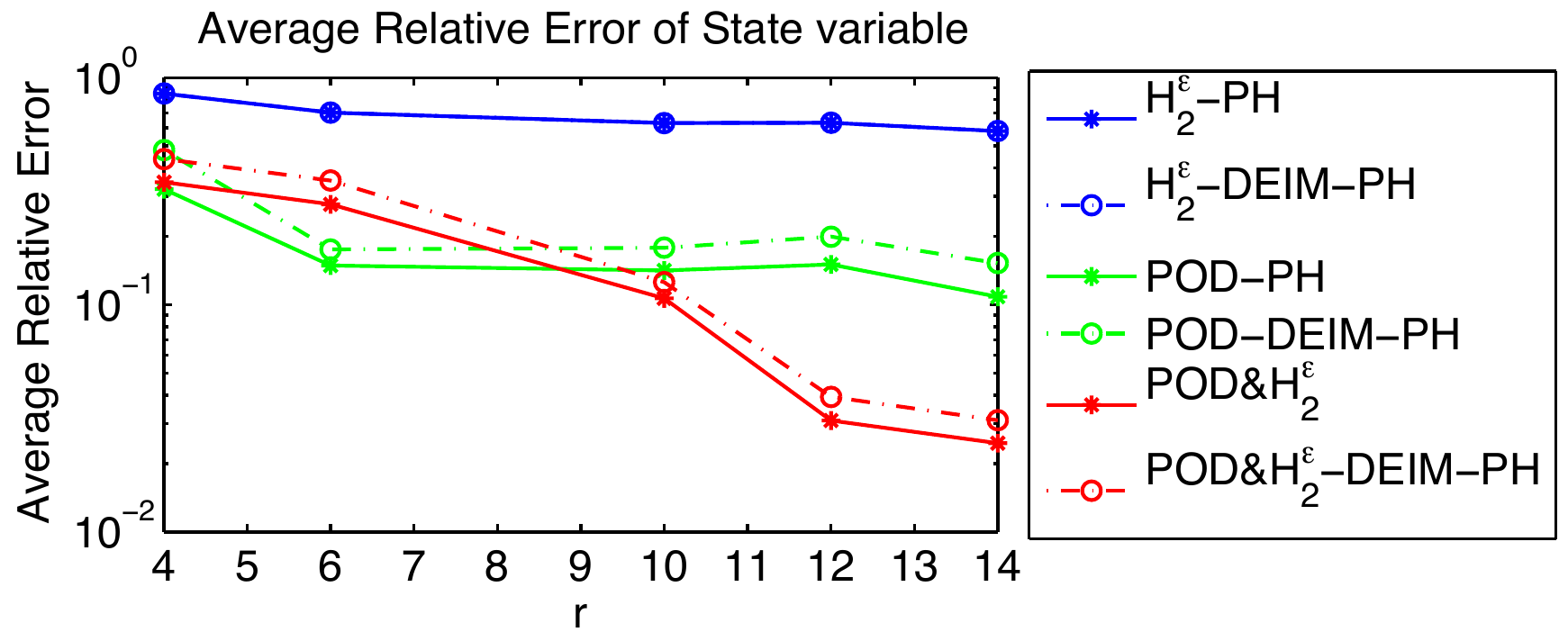}\hspace{-2.5ex}
  }
\label{fig:combPOD_H2_DEIM_err}
\end{figure}

\subsubsection{Toda Lattice}

A Toda lattice model describes the motion of a chain of particles, each one connected to its nearest neighbors with 
'exponential springs'. The equations of motion for the $N$-particle Toda lattice with such exponential interactions
can be written in the form of a  nonlinear port-Hamiltonian system as
in (\ref{PHdef}) 
with
$$\bJ = 
\left[\begin{array}{cc}
\rm{\bf 0} & \bI \\
-\bI & \rm{\bf 0}
\end{array}
\right]\in \real^{n \times n},~
\bR = 
\left[\begin{array}{cc}
\rm{\bf 0} & \rm{\bf 0} \\
\rm{\bf 0} & \rm{diag }(\gamma_1, \dots, \gamma_N)
\end{array}
\right]
\in \real^{n \times n},~ 
\bB = \left[\begin{array}{c}
\rm{\bf 0} \\
\be_1 
\end{array}
\right] \in \real^{n},$$
where  
$n=2N$ and 
$$\bx = \left[\begin{array}{c}
\bq\\
\bp    
\end{array}
\right] \in \real^{n},~~
\bq = [q_1, \dots, q_N]^T,~~\bp = [p_1, \dots, p_N]^T$$
with $q_j$ and $p_j$ being, respectively, the displacement of the $j$-th particle from its equilibrium position and the momentum
of the $j$-th particle for $j= 1, \dots, N$.
In this example; we use 
$N = 1000$ (i.e., the system dimension is $n=2000$), $\bx_0 = {\bf 0}$, and
$\gamma_j   =0.1 $, for $j= 1, \dots, N$.  The corresponding nonlinear Hamiltonian  is given by 
\[
H(\bx) = 
H ([\bq;\bp])= \sum_{k = 1}^{N} \frac{1}{2} p_k^2 + \sum_{k= 1}^{N-1}  \exp(q_k-q_{k+1} )+ \exp(q_N) - q_1 - N.
\]
The decomposition $H(\bx) =  \mbox{$\frac{1}{2}$}\bx^T\bQ\bx + h(\bx)$ is used with 
$$
\bQ=\left[\begin{array}{cc} \bQ_0 & 0 \\ 0 & \bI  \end{array}\right], \quad
\bQ_0=\left[\begin{array}{rrrrrrr} 1 & -1 &  0 & \ldots & 0 & 0\\
-1 & 2 &  -1 & &  & 0\\
 0 & -1 & 2 &  \ddots &  & 0 \\
\vdots &  & \ddots  & \ddots &  & & \\
   &   &            &      & -1 & 0\\
0 &  &  \ldots & -1 & 2 &  -1\\
0 & 0 &  &  0 & -1 & 2  \end{array}\right], 
$$
and $h(\bx)= \sum_{k= 1}^{N-1}(q_k-q_{k+1})^3\varphi(q_k-q_{k+1})$ where 
$\varphi(z)=\frac12\int_0^1\theta^2\,e^{(1-\theta)z}\,d\theta$.
The system was excited with two different inputs: $u(t) = 0.1$ and $u(t) = 0.1 \sin(t)$.
Figure~\ref{fig:Toda_outputs2} shows that the outputs due to both inputs are accurately approximated by the outputs of the \podph~and \poddeimph~reduced systems. The average relative errors and the  simulation times relative to the simulation time of the full model  are illustrated  in Figures~\ref{fig:Toda_POD_H2_DEIM_u01} and~\ref{fig:Toda_POD_H2_DEIM_u01sin}.   
Notice that the accuracy of the  \poddeimph~reduced model captures that of the \podph~model as the \textsc{deim} dimension $m$ increases. Note also that the\poddeimph~reduced model cuts the simulation time by nearly  $96\%$
while retaining  accuracy.
\begin{figure}[th!]
\caption{\scriptsize 
Toda Lattice: The outputs from inputs $u(t) = 0.1$ and $u(t) = 0.1 \sin(t)$. 
}
\centerline{
 \includegraphics[scale=0.34]{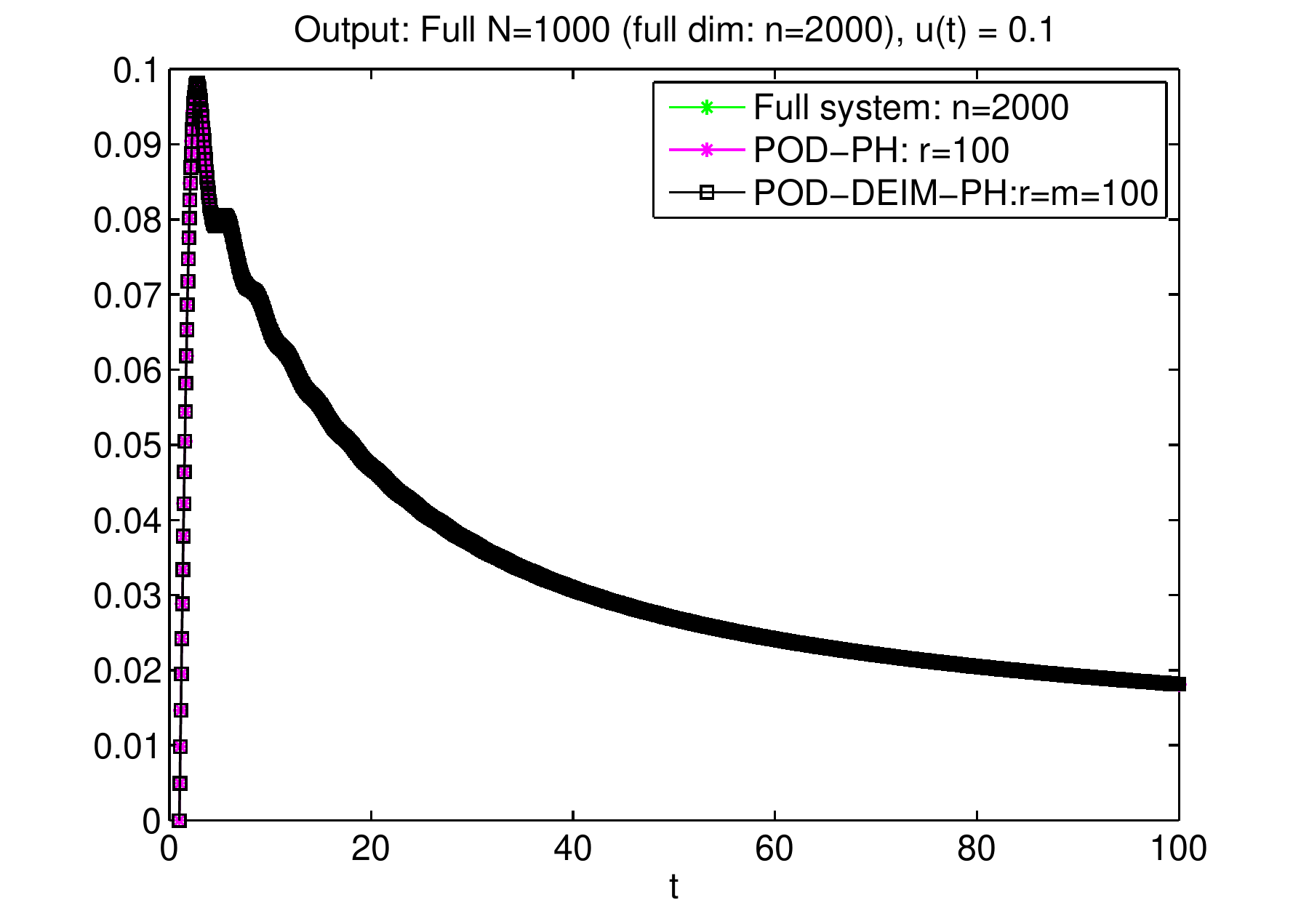}\hspace{-2ex}
 \includegraphics[scale=0.34]{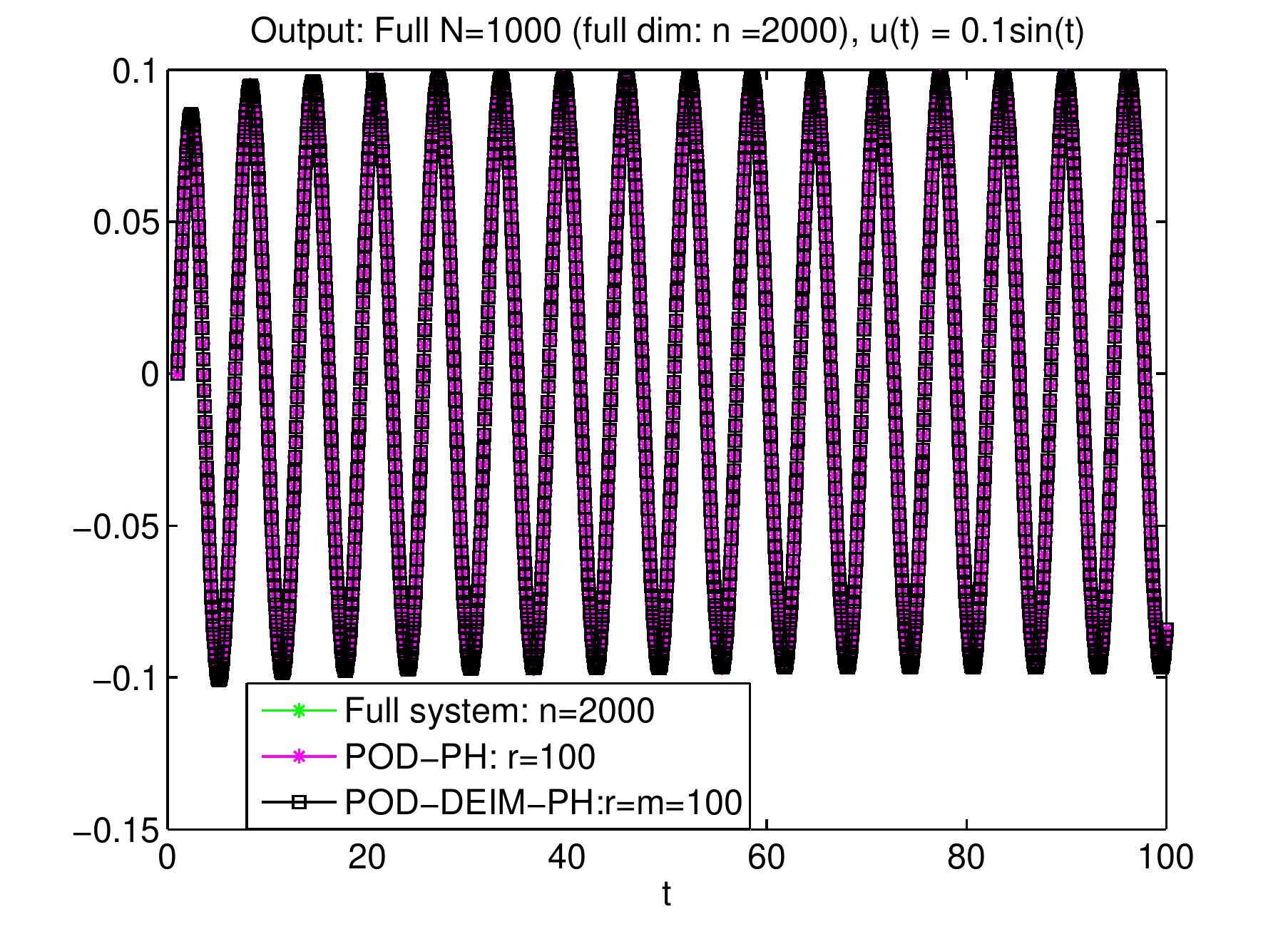}\hspace{-2ex}
} 
\label{fig:Toda_outputs2}
\end{figure}
%
\begin{figure}[th!]
\caption{\scriptsize 
Toda Lattice with input $u(t) = 0.1$: Relative errors of the outputs and the online CPU times for 
\podph~and \poddeimph~reduced systems 
with POD basis dimension $r$ and DEIM dimension $m = r, m_1, m_2$ where  
$m_1 =r +$ {\tt ceil(r/3)},
$m_2 =r +${\tt ceil(2r/3)};
 full-order system: $n = 2000$. 
}
  \centerline{
 \includegraphics[scale=0.32]{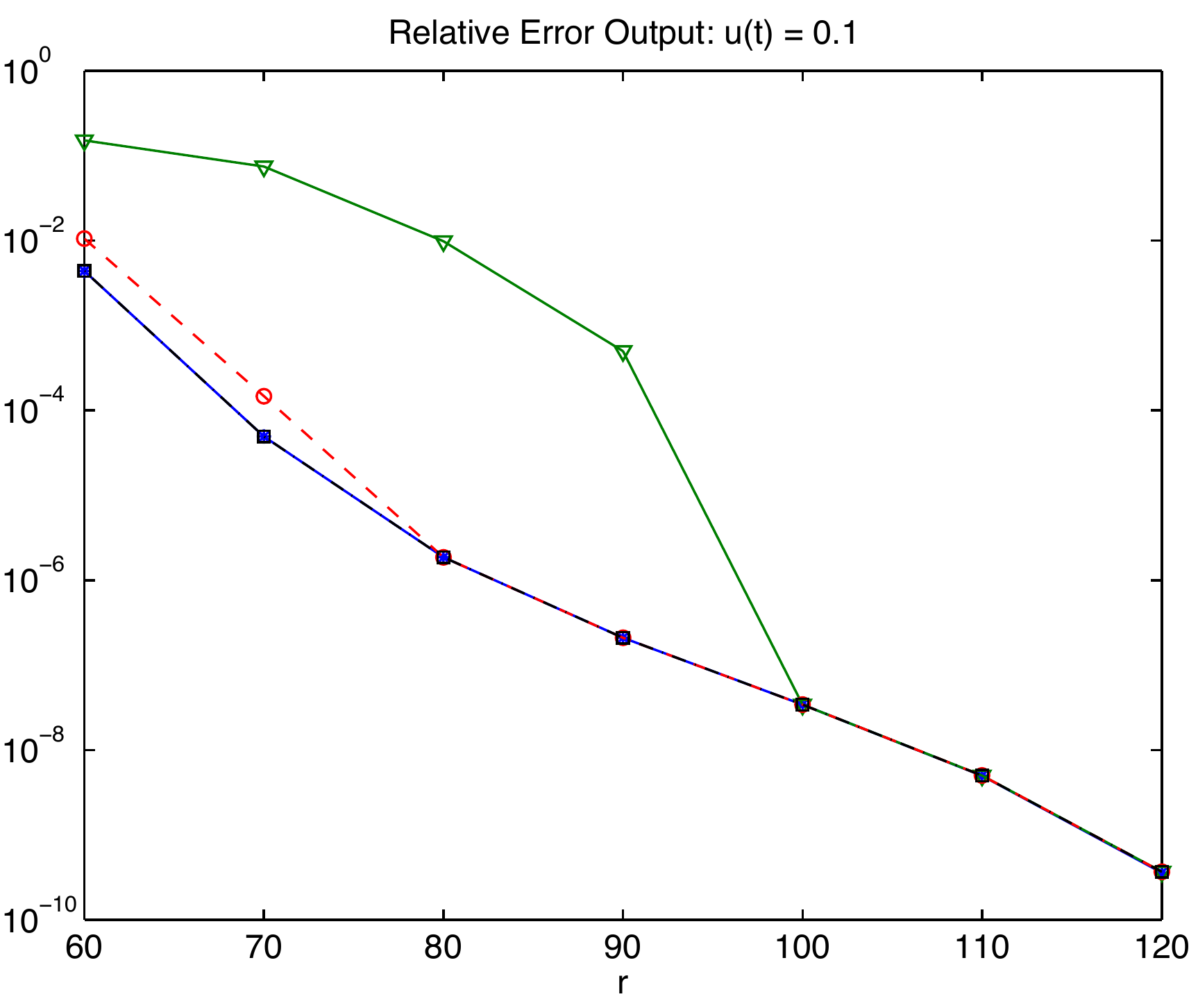}
  \includegraphics[scale=0.25]{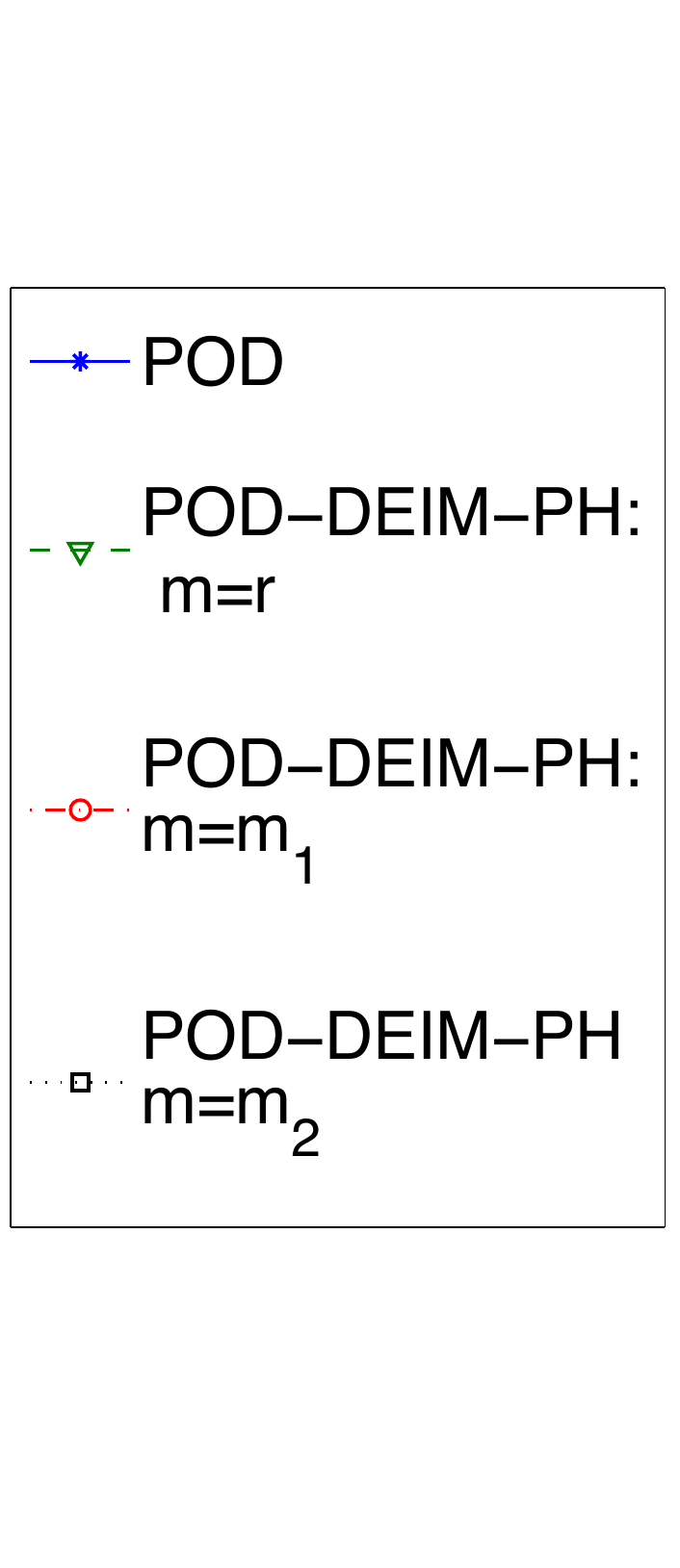}
  \hspace{-0ex}
 \includegraphics[scale=0.32]{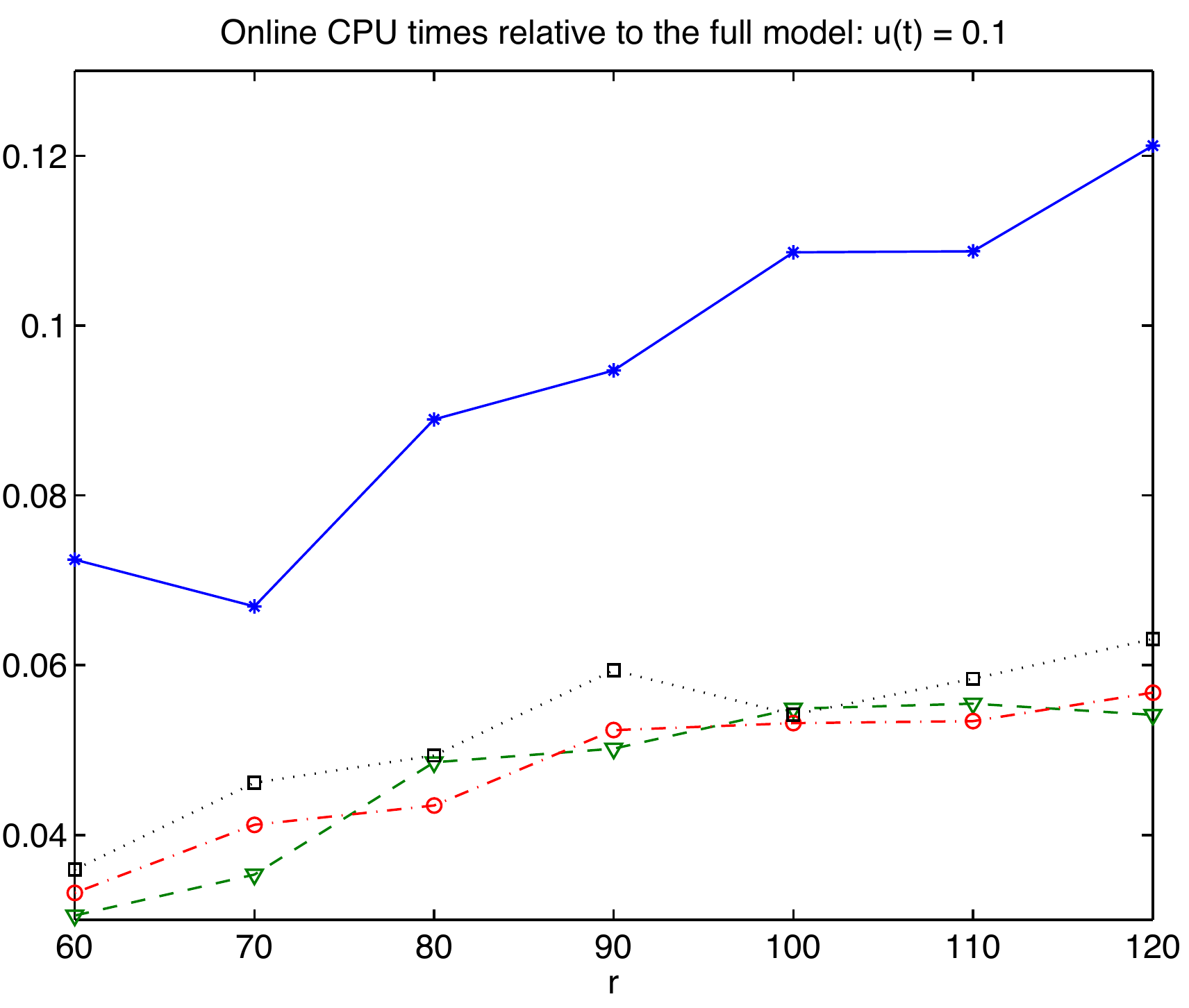}
 }   
\label{fig:Toda_POD_H2_DEIM_u01}
\end{figure}

\begin{figure}
\caption{\scriptsize 
Toda Lattice with input $u(t) = 0.1 \sin(t)$: Relative errors of the outputs and the online CPU times for 
\podph~and \poddeimph~ reduced systems 
with POD basis dimension $r$ and DEIM dimension $m = r, m_1, m_2$ where  
$m_1 =r +$ {\tt ceil(r/3)},
$m_2 =r +${\tt ceil(2r/3)};
 full-order system: $n = 2000$. 
}
 \centerline{
 \includegraphics[scale=0.32]{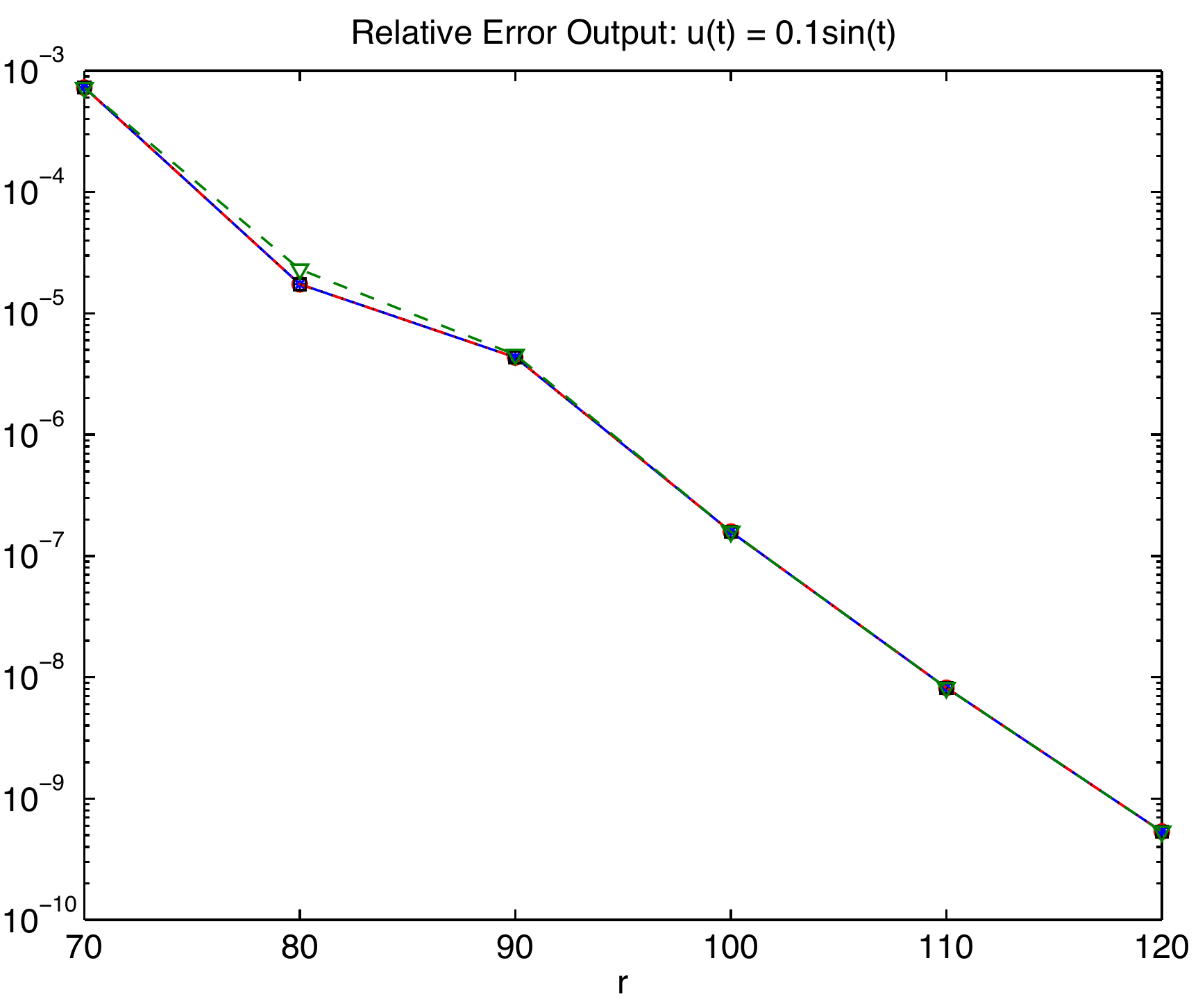}
 \includegraphics[scale=0.25]{errStatePODdeimr60_120_Legend_Dec30}
 \includegraphics[scale=0.32]{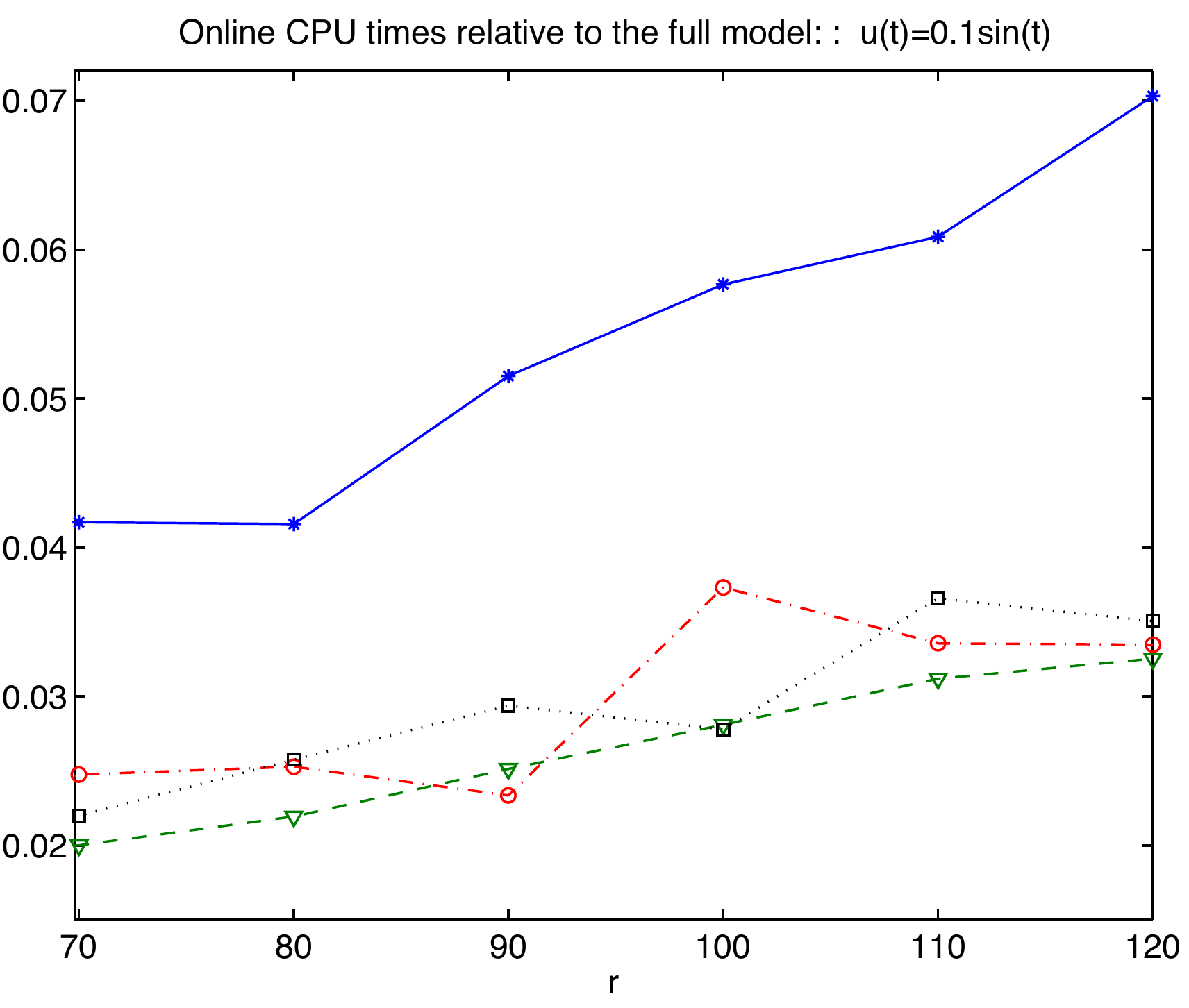}
 }
\label{fig:Toda_POD_H2_DEIM_u01sin}
\end{figure}

\subsection{An \emph{a priori} error bound for PH-preserving DEIM reduction}
\label{sec:MOR_DEIMstr_err}
We derive error bounds for 
a \textsc{deim}-based reduced order model
preserving PH-structure by
%
estimating additional errors that occur by introducing 
the symmetrized-\textsc{deim} approximations into the structure-preserving reduction framework of \S \ref{sec:MOR_Proj_NPH}.     In particular, suppose that reduction 
bases, $\bV_r$ and $\bW_r$,
have been chosen in some manner (e.g., as described as in \S \ref{sec:basis_POD}, 
\S \ref{sec:basis_HtwoEps}, or \S\ref{sec:PODH2}, say), and then are used 
to produce a reduced model that preserves the PH structure of the original system:
\begin{equation} \label{redNPH} 
\begin{array}{ll}
&\dot{\bx}_r  = (\bJ_r -\bR_r )\nabla_{\bx_r}H_r(\bx_r)  +\bB_r \bu(t) \\[2mm]
& \by_r    \ = \ \bB^T_r \nabla_{\bx_r}H_r(\bx_r).
\end{array}
\end{equation}
We then symmetrically "sparsify" the nonlinear interactions in the Hamiltonian gradient evaluation and introduce a further \textsc{deim}-reduction as described in \S \ref{SectPreserveportHamStruct} and Algorithm \ref{alg:poddeimph} in order to produce: 
\begin{equation} \label{DEIMredNPH} 	
\begin{array}{ll}
& \bhatxdot_r = (\bJ_r -\bR_r ) \nabla_{\bhatx_r}\widehat{H}_r(\bhatx_r)  +\bB_r \bu(t) \\[2mm]
& \bhaty_r    \ = \ \bB_r^T \nabla_{\bhatx_r}\widehat{H}_r(\bhatx_r) 
\end{array}
\end{equation}
where $\bx(0)=\bhatx(0)$ and the \textsc{deim} Hamiltonian,  $\widehat{H}$, has been defined in (\ref{DEIMHam}).  
\begin{theorem}
\label{thm:errBd_MORDEIMstr}
  Let $\bx_r(t)$ and $\by_r(t)$ be the state trajectory and output 
 associated with the reduced port-Hamiltonian system (\ref{redNPH}). 
 Suppose a \textsc{deim} basis of order $m$ has been chosen and used to 
 define a \textsc{deim} projection $\mathbb{P}$ and  \textsc{deim} Hamiltonian
$\widehat{H}(\bx)$ as introduced in (\ref{DEIMHam}). 
Recalling the discussion of \S \ref{DEIMHamiltonian}, suppose $\Omega\subset \cV_r $ contains the 
trajectories $\bfV_r\bx_r(t)$ and $\mathbb{P}^T\bfV_r\bx_r$ and  that $\varepsilon_{h}>0$ satisfies
$$
\sup_{\bfxi\in\Omega} \|\nabla_{\bx}h(\bfxi)-\mathbb{P}\nabla_{\bx}h(\mathbb{P}^T\bfxi)\|_{\bQ} \ \leq \ \varepsilon_{h}
$$
 Let $\bhatx_r(t)$ and $\bhaty_r(t)$ be the state trajectory and output 
 associated with the  \textsc{deim}-reduced port-Hamiltonian system (\ref{DEIMredNPH}).

Then with 
$$
\begin{array}{ccc}
\alpha  = \cL_{\bfQ}[\bG]- \rho_{min}, & \bG(\bfxi) = \cbfP_r\cbfA \cbfP_r^T \mathbb{P}\nabla_{\bx}h(\mathbb{P}^T\bfxi),&  \mbox{and}\quad \rho_{min} = \min \lambda(\bfR_r)\\[2mm]
 \beta=\|\cbfP_r\cbfA \cbfP_r^T\|_{\bQ}, & \gamma = 1+L_{\bQ}[\mathbb{P}\,\nabla_{\bx}h\circ \mathbb{P}^T], & \delta= \| \bB\| \, \|\cbfP_r\| 
\end{array}
$$
we have, for $\alpha\neq 0$,
\begin{equation}  \label{DEIMerrBnd}
\begin{array}{c}
\| \bx_r(t) - \bhatx_r(t)\| \leq \frac{\beta}{\alpha} \left(e^{\alpha\,t}-1\right)\, \varepsilon_{h} \\[2mm]
\| \by_r(t)  - \bhaty_r(t) \| \leq  \delta\left(1+\frac{\beta\, \gamma}{\alpha} \left(e^{\alpha\,t}-1\right)\right)\,\varepsilon_{h},
\end{array}
\end{equation}
whereas for $\alpha=0$,
$$
\begin{array}{c}
\| \bx_r(t) - \bhatx_r(t)\| \leq \beta\,t\, \varepsilon_{h} \\[2mm]
\| \by_r(t)  - \bhaty_r(t) \| \leq  \delta\left(1+\beta\,t\,\right)\,\varepsilon_{h}.
\end{array}
$$
\end{theorem}
\begin{proof}
Let
$\be_r(t): = \bx_r(t) - \bhatx_r(t)$ be the difference between state trajectories generated 
by (\ref{redNPH}) and (\ref{DEIMredNPH}). 
Define $\bff(\bx) =\nabla_{\bx}h(\bx)$ and observe that 
\begin{align*}
   \frac{d}{dt} \| \be_r(t)\|   & =  \left\langle \frac{\be_r(t)}{\| \be_r(t)\|}, \dot{\be}_r(t) \right\rangle 
     =  \left\langle \frac{\be_r(t)}{\| \be_r(t)\|}, \bA_r \left[\nabla_{\bx_r}H_r(\bx_r)  -  \nabla_{\bhatx_r}\widehat{H}_r(\bhatx_r)  \right]  \right\rangle \\
    &  =  \left\langle \frac{\be_r(t)}{\| \be_r(t)\|}, \bA_r \left[\be_r  + \bV_r^T \left(\nabla_{\bx}h(\bV_r\bx_r)  -  
    \mathbb{P}\, \nabla_{\bx}h(\mathbb{P}^T \bV_r \bhatx_r) \right) \right]  \right\rangle 
  \end{align*}
We have 
\begin{align*}
& \left\langle \be_r, \bA_r \left[\be_r  + \bV_r^T \left(\nabla_{\bx}h(\bV_r\bx_r)  -  
    \mathbb{P}\, \nabla_{\bx}h(\mathbb{P}^T \bV_r \bhatx_r) \right) \right]  \right\rangle \\
 &\qquad  = -\left\langle \be_r, \bR_r \be_r \right\rangle 
 + \left\langle \bV_r\be_r,\,\cbfP_r\bA \cbfP_r^T \left(\nabla_{\bx}h(\bV_r\bx_r)  -  
    \mathbb{P}\, \nabla_{\bx}h(\mathbb{P}^T \bV_r \bhatx_r) \right)  \right\rangle_{\bQ}  \\
 &\qquad  = -\left\langle \be_r, \bR_r \be_r \right\rangle 
  + \left\langle \bV_r\be_r,\,\cbfP_r\bA \cbfP_r^T \left( \mathbb{P}\, \nabla_{\bx}h(\mathbb{P}^T \bV_r \bx_r) -  
  \mathbb{P}\, \nabla_{\bx}h(\mathbb{P}^T \bV_r \bhatx_r) \right)  \right\rangle_{\bQ}   \\
& \qquad\qquad  + \left\langle \bV_r\be_r,\,\cbfP_r\bA \cbfP_r^T \left(\nabla_{\bx}h(\bV_r\bx_r)  
 -  \mathbb{P}\, \nabla_{\bx}h(\mathbb{P}^T \bV_r \bx_r) \right)  \right\rangle_{\bQ}  \\
 & \quad \leq -\rho_{min} \| \be_r\|^2 +  \cL_{\bfQ} [\bG]  \, \| \be_r\|^2  \\ 
 & \qquad\qquad  + \| \be_r\|\ \|\cbfP_r\bA \cbfP_r^T\|_{\bQ}\, \|\nabla_{\bx}h(\bV_r\bx_r)  -  \mathbb{P}\, \nabla_{\bx}h(\mathbb{P}^T \bV_r \bx_r)\|_{\bQ}
  \end{align*}
Thus, 
\begin{align*}
   \frac{d}{dt} \| \be_r(t)\| \leq & \left(\cL_{\bfQ} [\bG]-\rho_{min}\right)  \, \| \be_r(t)\|  \\
  & \qquad + \|\cbfP_r\bA \cbfP_r^T\|_{\bQ}\, \|\nabla_{\bx}h(\bV_r\bx_r(t))  -  \mathbb{P}\, \nabla_{\bx}h(\mathbb{P}^T \bV_r \bx_r(t))\|_{\bQ}.
 \end{align*}
If $\alpha=\cL_{\bfQ} [\bG]-\rho_{min}=0$ then trivially, 
{\small 
$$
\|\bx_r(t) - \bhatx_r(t)\| \leq\ \|\cbfP_r\bA \cbfP_r^T\|_{\bQ}\,\int_0^t \|\nabla_{\bx}h(\bV_r\bx_r(\tau))  -  \mathbb{P}\, \nabla_{\bx}h(\mathbb{P}^T \bV_r \bx_r(\tau))\|_{\bQ} \, d\tau \leq \beta\, t \, \varepsilon_{h}. 
$$
}
If $\alpha\neq 0$ then by Gronwell's inequality, 
 {\small 
 \begin{align*}
\|\bx_r(t) - \bhatx_r(t)\| &\leq\ \|\cbfP_r\bA \cbfP_r^T\|_{\bQ}\,\int_0^t e^{\alpha(t-\tau)}\, \|\nabla_{\bx}h(\bV_r\bx_r(\tau))  -  \mathbb{P}\, \nabla_{\bx}h(\mathbb{P}^T \bV_r \bx_r(\tau))\|_{\bQ} \, d\tau \\
& \leq \frac{\beta}{\alpha} \left(e^{\alpha\,t}-1\right)\, \varepsilon_{h}
 \end{align*}
}
This expression holds with a positive bound regardless of whether $\alpha>0$ or $\alpha<0$. 

The resulting difference in output maps may be directly bounded:
 \begin{align*}
\|\by_r  -\bhaty_r \| \ &= \ \|\bB^T_r \left(\nabla_{\bx_r}H_r(\bx_r)-\nabla_{\bhatx_r}\widehat{H}_r(\bhatx_r)\right)\|\\
&= \ \|\bB^T_r \left(\bx_r  -\bhatx_r + \bV_r^T \left(\nabla_{\bx}h(\bV_r\bx_r)  
              -  \mathbb{P}\, \nabla_{\bx}h(\mathbb{P}^T \bV_r \bhatx_r) \right)\right)\|\\
      &\leq\    \|\bB^T_r \be_r\| + \|\bB^T_r\bV_r^T \left(\nabla_{\bx}h(\bV_r\bx_r)  
      -  \mathbb{P}\, \nabla_{\bx}h(\mathbb{P}^T \bV_r \bx_r) \right) \|  \\ 
      &\   \qquad  + \|\bB^T_r\bV_r^T \left( \mathbb{P}\, \nabla_{\bx}h(\mathbb{P}^T \bV_r \bx_r) -  
  \mathbb{P}\, \nabla_{\bx}h(\mathbb{P}^T \bV_r \bhatx_r) \right)\| \\
   & \leq \| \bB\| \, \|\cbfP_r\|\, \left(1+L_{\bQ}[\mathbb{P}\,\nabla_{\bx}h\circ \mathbb{P}^T]\right)\|\bx_r(t) - \bhatx_r(t)\| 
   +  \| \bB\| \, \|\cbfP_r\|\, \varepsilon_{h},
   \end{align*}
which leads to the second conclusions respectively for $\alpha\neq 0$ and $\alpha=0$.
\end{proof}


\section{Conclusions} 
\label{sec:conclusion}
We have introduced a  structure-preserving projection-based  model reduction framework for large-scale multi-input/multi-output nonlinear port-Hamiltonian systems.
We constructed projection subspaces
using three different 
approaches: \textsc{pod},  $\mathcal{H}_2$-based, and a hybrid \textsc{pod}-$\mathcal{H}_2$ based.
We showed that for the same reduced order, the hybrid basis significantly outperforms the other two.  
We introduced a modification of \textsc{deim} within this framework to approximate the nonlinear part of the Hamiltonian gradient symmetrically. In all cases, the resulting reduced system preserves
port-Hamiltonian structure, and thus retains the stability and passivity of the original
system. We have derived the corresponding \emph{a priori} error bounds of the state variables
and outputs by using an application of generalized logarithmic norms for unbounded nonlinear
operators. The effectiveness of the proposed approaches were shown on a nonlinear ladder
network and a Toda lattice model. 

%
%



\bibliography{litbib}
\bibliographystyle{siam}


\end{document}